\documentclass{article}

\usepackage{PRIMEarxiv}
\usepackage[utf8]{inputenc} 
\usepackage[T1]{fontenc}    
\usepackage{hyperref}       
\usepackage{url}            
\usepackage{booktabs}       
\usepackage{amsfonts}       
\usepackage{nicefrac}       
\usepackage{microtype}      
\usepackage{lipsum}
\usepackage{fancyhdr}       
\usepackage{graphicx}       
\graphicspath{{media/}}     

\usepackage[
backend=bibtex,
sorting=none
]{biblatex}
\usepackage{graphicx}
\usepackage{amsmath}
\usepackage{amsthm}
\usepackage[version=4]{mhchem}
\usepackage{siunitx}
\usepackage{longtable,tabularx}
\usepackage{dsfont}
\usepackage{calc}
\usepackage{color}
\usepackage[ruled,vlined]{algorithm2e}
\usepackage{verbatim}
\newcommand{\norm}[1]{\left\lVert#1\right\rVert}
\newcommand{\abs}[1]{\left|#1\right|}
\newcommand{\menge}[1]{\left\lbrace#1\right\rbrace}

\newcommand{\intD}{\;\mathrm{d}}
\newcommand{\rom}[1]{\uppercase\expandafter{\romannumeral #1\relax}}
\newcommand{\landauO}{\mathcal{O}}
\newtheorem{theorem}{Theorem}

\newtheorem{definition}{Definition}
\newcommand{\inner}[1]{\left< #1 \right>}

\title{Neural network-based, structure-preserving entropy closures for the Boltzmann moment system 
\thanks{\textit{\underline{Citation}}: 
\textbf{Steffen Schotthöfer, Tianbai Xiao, Martin Frank and Cory D. Hauck. Neural network-based, structure-preserving entropy closures for the Boltzmann moment system.}} 
}

\author{
  Steffen Schotthöfer \\
  Faculty of Mathematics \\
  KIT \\
  Karlsruhe, Germany\\
  \texttt{steffen.schotthoefer@kit.edu} \\
   \And
   Tianbai Xiao \\
  Faculty of Mathematics \\
  KIT \\
  Karlsruhe, Germany\\
  \texttt{tianbai.xiao@kit.edu} \\
  \And
  Martin Frank\\\
  Faculty of Mathematics \\
  KIT \\
  Karlsruhe, Germany\\
  \texttt{martin.frank@kit.edu} \\
   \And
  Cory D. Hauck\\\
  Computer Science and Mathematics Division \\
  Oak Ridge National Laboratory, and
Department of Mathematics (Joint Faculty),  University of Tennessee \\
  Oak Ridge, TN 37831 USA\\
  \texttt{hauckc@ornl.gov} \\
}

\begin{document}
\maketitle

\begin{abstract}

This work presents neural network based minimal entropy closures for the moment system of the Boltzmann equation, that preserve the inherent structure of the system of partial differential equations, such as entropy dissipation and hyperbolicity. The described method embeds convexity of the moment to entropy map in the neural network approximation to preserve the structure of the minimal entropy closure. Two techniques are used to implement the methods. The first approach approximates the map between moments and the minimal entropy of the moment system and is convex by design. The second approach approximates the map between moments and Lagrange multipliers of the dual of the minimal entropy optimization problem, which present the gradients of the entropy with respect to the moments, and is enforced to be monotonic by introduction of a penalty function. 
We derive an error bound for the generalization gap of convex neural networks which are trained in Sobolev norm and use the results to construct data sampling methods for neural network training.
Numerical experiments are conducted, which show that neural network-based entropy closures provide a significant speedup for kinetic solvers while maintaining a sufficient level of accuracy. The code for the described implementations can be found in the Github repositories~\cite{NeuralEntropy,KITRT}.

\end{abstract}


\section{Introduction}
In many applications, a macroscopic description of the physical systems  is no longer applicable and one has to rely on a more general description, which is given by kinetic equations such as the Boltzmann equation.
Example include neutron transport~\cite{NeutronTransport}, radiative transport~\cite{Chahine1987FoundationsOR} and semiconductors~\cite{Markowich1990SemiconductorE} and rarefied gas dynamics~\cite{BoltzmannApplications}.
The Boltzmann equation is a high dimensional integro-differential equation, with phase space dependency on space and particle velocity. This high dimensionality of the phase space presents a severe computational challenge for large scale numerical simulations. \\
Several methods for phase space reduction have been proposed to solve the Boltzman equation, including the discrete ordinate/velocity methods ~\cite{NeutronTransport,Camminady2019RayEM,xiao2020velocity,XIAO2021110689,xiao2017well} and moment methods~\cite{AlldredgeFrankHauck,AlldredgeHauckTits,GarretHauck,KRISTOPHERGARRETT2015573, Levermore}.
Discrete ordinate methods evaluate the velocity space at specific points, which yields a system of equations only coupled by the integral scattering operator. While computationally efficient, these methods suffers from numerical artifacts, which are called ray effects~\cite{Camminady2019RayEM}.
Moment methods eliminate the dependency of the phase space on the velocity variable by computing the moment hierarchy of the Boltzmann equation. Due to the structure of the advection term, the resulting moment system is typically unclosed. One distinguishes moment methods according to the modelling of their closure. 
The classical $P_N$ closure uses a simple truncation that results in a system of linear hyperbolic equations. The main drawback of this method is its numerical artifacts, specifically large oscillations of the particle density, which may even result in negative particle concentrations. This effect is particularly present in the streaming particle regime~\cite{Brunner2002FormsOA}. \\
A moment closure, which preserves important physical and mathematical properties~\cite{Levermore1996MomentCH} of the Boltzmann equation, such as entropy dissipation, hyperbolicity, and the H-theorem, is constructed by solving a convex constrained optimization problem  based on the entropy minimization principle~\cite{AlldredgeHauckTits, Levermore}. The method, which is commonly referred to as $M_N$ closure, is accurate in the diffusive limit~\cite{GOUDON2013579} and unlike the $P_N$ closure, it is also accurate in the streaming limit~\cite{DUBROCA1999915}. Although the $M_N$ closure is methodically superior to the $P_N$ closure, it is by far more expensive to compute. Garret et al.~\cite{GarretHauck} have demonstrated, that in a high performance implementation, more than $80$\% of the computational time of the whole solver is required for the solution of the entropy minimization problem. This motivates the development of a neural network surrogate model to accelerate the $M_N$ closure.\\
Several machine learning inspired methods have been proposed recently. 
The authors of~\cite{huang2021machine} close the moment system by learning the spatial gradient of the highest order moment. In~\cite{Han21983}, the authors pursue two strategies. First, they use a encoder-decoder network to generate generalized moments and then learn the moment closure of the system with its dynamics in mind. Second, they learn directly the correction term to the Euler equations. 
In~\cite{huang2020learning}, Galilean invariant machine learning methods for partial differential equations are developed using the conservation dissipation formalism. 
Using convolutional networks, a closure for the one dimensional Euler-Poisson system was constructed in ~\cite{bois2020neural}.
In~\cite{XIAO2021110521}, a dense neural network was used to model the deviation from the Maxwellian in the collision term of the Boltzmann equation.
The authors of~\cite{Maulik2020neural} use neural networks to reproduce physical properties of known magnetized plasma closures. 
In~\cite{ma20machine}, fully connected, dense and discrete Fourier transform networks are used to learn the Hammett-Perkins Landau fluid closure. 
Physics informed neural networks were employed  to solve forward and inverse problems via the Boltzmann-BGK formulation to model flows in continuum and rarefied regimes in~\cite{lou2020physicsinformed}, to solve the radiative transfer equation ~\cite{mishra2021physics} and the phonon Boltzmann equation in~\cite{li2021physicsinformed}. In~\cite{porteous2021datadriven}, the authors propose a data driven surrogate model of the minimal entropy closure using convex splines and empirically convex neural networks. To ensure convexity at the training data points, the authors penalize a non symmetric positive definite Hessian of the network output.  \\
The goal of this work is to construct structure-preserving deep neural network surrogate models for the entropy closure of the moment system of the Boltzmann Equation. 
The motivation is the estabilshed result from~\cite{AlldredgeFrankHauck}, that convex approximations to the entropy will preserve important mathematical properties of the moment system.
The first proposed neural network maps a given moment to its minimal mathematical entropy. In contrast to the work proposed in~\cite{porteous2021datadriven}, the neural network is input convex by design using the methods of Amos et al.~\cite{Amos2017InputCN}. By this ansatz, the learned closure automatically inherits all structural properties of the entropy closure for the moment system, due to the result of ~\cite{AlldredgeFrankHauck}, that any convex approximation to the entropy preserves said properties. The derivative of the network with respect to the moments maps to the corresponding optimal Lagrange multipliers of the entropy minimization problem. We train the neural network on the output, the Lagrange multiplier and additionally on the reconstructed moments, whereas in~\cite{porteous2021datadriven}, the authors train on the output, the reconstructed moments and the Hessian of the network.
The second approach of in this work is a monotonic neural network that maps the moments directly to the Lagrange multipliers of the entropy minimization problem. We use a penalty function to train the neural network to be monotonic and otherwise use the same loss as in the input convex approach.\\
The remainder of this paper is structured as follows.  In Section~\ref{sec_kinEq}, we give a brief introduction to the Boltzmann equation. We review the moment system with minimal entropy closure and examine its benefits and shortcomings.
In Section~\ref{sec_neuralEntropyClosure}, we present our findings for structure-preserving, neural network based surrogate models for the minimal entropy closure. To this end we describe two structure-preserving neural network architectures and their integration in an end-to-end numerical solver. We show that the intrinsic structure of the moment system is preserved. Additionally, we analyze the data-to-solution map of the minimal entropy closure and perform a dimension reduction.
In Section~\ref{sec_data}, we give a review over characterizations of the boundary of the set of feasible moments for the minimal entropy closure. Afterwards we propose an error bound for the generalization gap for the gradient of input convex neural networks trained in Sobolev norm. Then, we propose a sampling strategy to generate training data for closures in arbitrary spatial dimension and moment order, based on the analysis of the generalization gap.
Lastly, Section~\ref{sec_numResults} presents a range of numerical studies that show that the neural entropy closure is computationally more efficient than the reference solution. Furthermore, a series of synthetic tests as well as simulation tests are conducted to inspect the numerical accuracy of the neural network based entropy closures.
\section{Kinetic Theory}\label{sec_kinEq}
\subsection{Kinetic equations}
Classical kinetic theory is profoundly built upon the Boltzmann equation,
which describes the space-time evolution of the one-particle kinetic density function $f(t, \mathbf{x},\mathbf{v})$ in a many-particle system
\begin{equation}  \label{eq_boltzmann}
    \partial_t f+\mathbf{v} \cdot \nabla_{\mathbf{x}} f = Q(f).
\end{equation}
The phase space consists of time $t>0$, space $\mathbf x\in \mathbf{X}\subset\mathbb R^3$, and particle velocity  $\mathbf{v}\in\mathbf{V}\subset\mathbb R^3$.
The left-hand side of the equation describes particle transport, where the advection operator $\mathbf{v} \cdot \nabla_{\mathbf x}$ describes the movement of the particle density with velocity $\mathbf v$ in the spatial directions.
The integral operator $Q(f)$ on the right hand side of the equation models interaction of the particle with the background medium and collisions with other particles. If the particles only collide with a background material one can model this behavior with the  linear Boltzmann collision operator
\begin{equation}
    Q(f)(\mathbf{v})=\int_{\mathbf{V}} \mathcal B(\mathbf v_*, \mathbf v) \left[ f(\mathbf v_*)-f(\mathbf v)\right] d\mathbf v_*,
\end{equation}
where the collision kernel $\mathcal B(\mathbf v_*, \mathbf v)$ models the strength of collisions at different velocities. If the interactions among particles are considered, the collision operator becomes nonlinear. For example, the two-body collision results in
\begin{equation}
\label{eq_coll}
    Q(f,f)=\int_{\mathbf{V}} \int_{\mathcal S^2} \mathcal B(\cos \beta, |\mathbf{v}-\mathbf{v_*}|) \left[ f(\mathbf v')f(\mathbf v_*')-f(\mathbf v)f(\mathbf v_*)\right] d\mathbf \Omega d\mathbf v_*,
\end{equation}
where $\{\mathbf{v},\mathbf{v_*}\}$ are the pre-collision velocities of two colliding particles, and $\{\mathbf{v}',\mathbf{v_*}'\}$ are the corresponding post-collision velocities and $\mathcal S^2$ is the unit sphere. The right-hand side is a fivefold integral, where $\beta$ is the so-called  deflection angle.
In the following, we use the notation
\begin{equation}
\inner{\cdot} = \int_{\mathbf{v}}\cdot \intD \mathbf{v}
\end{equation}
to define integrals over velocity space. \\
Well-posedness of Eq.~\eqref{eq_boltzmann} requires appropriate initial and boundary conditions. 
The Boltzmann equation is a first-principles model based on direct modeling.
It possesses some key structural properties, which are intricately related to the physical processes and its mathematical existence and uniqueness theory. We briefly review some of these properties, where we follow~\cite{AlldredgeFrankHauck,Levermore1996MomentCH}.
First, the time evolution of the solution is invariant in range, i.e. if $f(0,\mathbf{x},\mathbf{v})\in B\subset[0,\infty)$, then $f(t,\mathbf{x},\mathbf{v})\in B\subset[0,\infty)$ for all $t>0$. Particularly this implies non-negativity of $f$.
Second, if $\phi$ is a collision invariant fulfilling
\begin{align}\label{eq_collision_invariant}
        \inner{\phi Q(g)} = 0,\quad \forall g\in\mathrm{Dom}(Q),
\end{align}
the equation
\begin{align}
        \partial_t\inner{\phi f} + \nabla_{\mathbf x}\cdot\inner{\mathbf v\phi f} = 0
\end{align}
is a local conservation law. Third, for each fixed direction $\mathbf{v}$, the advection operator, i.e. the left-hand side term of Eq.~\eqref{eq_boltzmann}, is hyperbolic in space and time. Forth, let $D\subset\mathbb{R}$. There is a twice continuously differentiable, strictly convex function $\eta: D\rightarrow\mathbb{R}$, which is called kinetic entropy density.  It has the property
\begin{align}
        \inner{\eta'(g)Q(g)} \leq 0, \quad \forall g\in\mathrm{Dom}(Q) \ \mathrm{s.t.} \ \mathrm{Im}(g)\subset D.
\end{align}
Applied to Eq.~\eqref{eq_boltzmann}, we get the local entropy dissipation law
\begin{align}
    \partial_t\inner{\eta(f)} + \nabla_{\mathbf x}\cdot\inner{\mathbf v\eta(f)} \leq 0.
\end{align}
Usually we set $D=B$. Lastly, the solution $f$ fulfills the H-theorem, i.e. equilibria are characterized by any of the three equivalent statements,
\begin{align}
     \inner{\eta'(g)Q(g)} &= 0,\\ \qquad  Q(g) &= 0,\\ \eta'(g)&\in E,
\end{align}
where $E$ denotes the linear span of all collision invariants.\\
\subsection{Moment methods for kinetic equations}\label{sec_momMethods}
The Boltzmann equation is an integro-differential equation model defined on a seven-dimensional phase space. With the nonlinear five-fold integral, it is  challenging to solve accurately and efficiently. 
The well-known moment method encode the velocity dependence of the Boltzmann equation by multiplication with a vector of velocity dependent basis functions $m(\mathbf{v})\in\mathbb{R}^{\tilde{N}}$, that consists of polynomials up to order $N$ and subsequent integration over $\mathbf{V}$.
In one spatial dimension, usually we have $\tilde{N}=N+1$, whereas in higher spatial dimensions $\tilde{N}$ equals the number of basis functions up to order $N$. 
The solution of the resulting moment equation is the moment vector $u\in\mathbb{R}^{\tilde{N}}$ and is calculated by

\begin{align}\label{eq_momentDef}
 u(t,\mathbf x)=\inner{m(\mathbf v) f(t,\mathbf{x},\mathbf v)}.
 \end{align} 
Common choices for the basis functions are monomials or spherical harmonics, depending on the application. Typically, they include the collision invariants defined in Eq.~\eqref{eq_collision_invariant}.
The moment vector satisfies the system of transport equations
\begin{align}\label{eq_momentRTa}
\partial_t u(t,\mathbf{x}) + \nabla_\mathbf{x}\cdot\inner{\mathbf{v} m(\mathbf{v})f}= \inner{m(\mathbf{v})Q(f)},
\end{align}
which is called moment system.
By construction, the advection operator depends on $f$ and thus, the moment system is unclosed. Moment methods aim to find a meaningful closure for this system. 
Since the kinetic equation dissipates entropy and fulfills a local entropy dissipation law, one can close the system by choosing the reconstructed kinetic density $f_u$ out of the set of all possible functions $F_m=\menge{g\in \text{Dom}(Q) : \text{Range}(g)\subset D \text{ and} \inner{mg}<\infty}$, that fulfill $u(t,\mathbf{x})=\inner{mg}$ as the one with minimal entropy $h$.  The minimal entropy closure can be formulated as a constrained optimization problem for a given vector of moments $u$.
\begin{align}\label{eq_entropyOCP} 
\min_{g\in F_m} \inner{\eta(g)}\quad  \text{ s.t. }  u=\inner{m g}
\end{align}
The minimal value of the objective function is denoted by $h(u)=\inner{\eta(f_u)}$ and $f_u$ is the minimizer of Eq.~\eqref{eq_entropyOCP}, which we use to close the moment system
\begin{align}\label{eq_momentRT}
    \partial_t u(t,\mathbf{x}) + \nabla_\mathbf{x}\cdot\inner{\mathbf{v} m(\mathbf{v})f_u}= \inner{m(\mathbf{v})Q(f_u)}.
\end{align}
The set of all moments corresponding to a kinetic density $f$ with $\text{Range}(f)\subset D$ is called the realizable set 
\begin{align}
    \mathcal{R}=\menge{u: \inner{mg}=u,\, g\in F_m}.
\end{align}
$\mathcal{R}$ is the set of all moments correpsonding to kinetic densities $f$ that fulfill the invariant range condition of the kinetic equation. 
There does not always exists a solution for the minimal entropy problem~\cite{Hauck2008ConvexDA}. However, if a solution exists for $u\in\mathcal{R}$, it is unique and of the form 
\begin{align}\label{eq_entropyRecosntruction}
    {f}_u = \eta'_*(\alpha(u)\cdot m).
\end{align}
where the Lagrange multiplier $ \alpha_u:\mathbb{R}^{\tilde{N}}\rightarrow\mathbb{R}^{\tilde{N}} $ maps $u$ to the solution of the convex dual problem
\begin{align}\label{eq_entropyDualOCP}
    \alpha_u =  \underset{\alpha\in\mathbb{R}^{\tilde{N}}}{\text{argmax}} \menge{ \alpha\cdot u - \inner{\eta_*(\alpha\cdot m)}}
\end{align}
and $\eta_*$ is the Legendre dual of $\eta$.
By the strong duality of the minimal entropy problem, the maximum of \eqref{eq_entropyDualOCP} equals the minimum of \eqref{eq_entropyOCP} and we can write at the optimal point $(u,\alpha_u)$
\begin{align}\label{eq_entropyFunctionalH}
h(u) =  \alpha_u\cdot u - \inner{\eta_*(  \alpha_u\cdot m)}.
\end{align}
The twice differentiable and convex function $h(u)$ serves as the entropy of the moment system~\cite{AlldredgeFrankHauck}. We can recover the moment $u$ by using first order optimality conditions
\begin{align}
    \frac{\intD}{\intD \alpha_u} h = u- \inner{m \eta'_*(\alpha_u\cdot m)} = 0
\end{align}
which yields also Eq.~\eqref{eq_entropyRecosntruction}, since $\inner{m f_u} = u=\inner{m \eta'_*(\alpha_u\cdot m)}$. This yields the inverse of the solution map $\alpha_u$ of the dual optimization problem.
Furthermore,  the derivative of $h$ recovers the optimal Lagrange multipliers of Eq.~\eqref{eq_entropyDualOCP},
\begin{align}\label{eq_derivH}
\frac{\intD}{\intD u}h =  \alpha_u.
\end{align}
This minimal entropy closure  also conserves the above listed structural properties of the Boltzmann equation . We present the above properties for the moment system for the sake of completeness, where we follow~\cite{AlldredgeFrankHauck,Levermore}. First, the invariant range property of the solution $f$ translates to the set of realizable moments $\mathcal{R}$. One demands that $u(t,\mathbf{x})\in\mathcal{R}$ for all $t>0$. Second, if a moment basis function $m_i(\mathbf{v})$ is a collision invariant, then 
 \begin{align}\label{eq_momentRT_lhS}
        \partial_t u(t, \mathbf x) + \nabla_{\mathbf{x}}\cdot\inner{\mathbf{v} m f_u}=0,
\end{align}
is a local conservation law. Third, one can write Eq.~\eqref{eq_momentRT} as a symmetric hyperbolic conservation law in $\alpha_u$. Forth,  for  $u\in\mathcal{R}$, $h(u)$ and $j(u)=\inner{\mathbf{v}\eta(f_u)}$ is a suitable entropy and entropy-flux pair compatible with the advection operator $\inner{\mathbf{v} mf_u}$ and yield a semi-discrete version of the entropy dissipation law. 
\begin{align}
    \partial_th(u) +\nabla_{\mathbf{x}} j(u)=h'(u)\cdot\inner{mQ(f_u(\alpha_u)}\leq 0
\end{align}
Note that convexity of $h(u)$ is crucial for the entropy dissipation property. Lastly, the moment system fulfills the H-theorem, which states equality of the following statements
\begin{align}
    \alpha_u\cdot\inner{mQ(f_u)} = 0,\\
    \inner{mQ(f_u)}=0, \\
    \alpha_u\cdot m\in E.
\end{align}
A numerical method to solve the moment system therefore consists of an iterative discretization scheme for the moment system~\eqref{eq_momentRT} and a   Newton optimizer for the dual minimal entropy optimization problem in Eq.~\eqref{eq_entropyDualOCP}. The former scheme can be a finite volume or discontinuous Garlerkin scheme, for example~\cite{KRISTOPHERGARRETT2015573}. 
The drawback of the method is the high computational cost associated with the Newton solver. The optimization problem in Eq.~\eqref{eq_entropyDualOCP} needs to be solved in each grid cell at every time step of the kinetic solver. The computational effort to solve the minimal entropy optimization problem grows over proportionately with the order $N$ of the moment basis $m$. Using three basis functions, the optimizer requires $80$\% of the computation time and $87$\% when using seven basis functions, as Garrett et al. have demonstrated in a computational study~\cite{KRISTOPHERGARRETT2015573}. 
Furthermore, the optimization problem is ill-conditioned, if the moments $u$ are near the boundary of the realizable set $\mathcal{R}$~\cite{AlldredgeHauckTits}. At the boundary $\partial\mathcal{R}$, the Hessian of the objective function becomes singular and the kinetic density $f_u$ is a sum of delta functions~\cite{Curto_recursiveness}.
\section{Structure-preserving entropy closures using neural networks}\label{sec_neuralEntropyClosure}
The following section tackles the challenge of solving the minimal entropy closure computationally efficiently while preserving the key structural properties of the Boltzmann equation. 
We propose two neural network architectures, which map a given moment vector to the solution of the minimal entropy problem, replacing the costly Newton solver that is used in traditional numerical methods. Whereas a Newton solver requires the inversion of a near singular Hessian matrix multiple times, the usage of a neural network needs a comparatively small amount of fast tensor operations to compute the closure. \\
A neural network $\mathcal{N}_\theta: \mathbb{R}^n\mapsto\mathbb{R}^m$ is a parameterized mapping from an input $x$ to the network prediction $y=\mathcal{N}_\theta(x)$. Typically a network is a concatenation of layers, where each layer is a nonlinear parameterized function of its input values. The precise structure of the network $\mathcal{N}_\theta$ depends on basic architectural design choices as well as many hyperparameters.
A simple multi-layer neural network $\mathcal{N}_\theta$ is a concatenation of layers $z_k\in\mathbb{R}^{n_k}$ consisting of non-linear (activation) functions $f_k$ applied to weighted sums of the previous layer's output $z_{k-1}\in\mathbb{R}^{n_{k-1}}$. An $M$ layer network can be described in tensor formulation as follows.
\begin{align}\label{eq_nnKons}
z_k &= f_k(W_kz_{k-1} +b_k), \qquad	k = 1,\dots,M \\
x &= z_0,  \\
\mathcal{N}_\theta(x) &= z_M
\end{align}
where $W_k$ is the weight matrix of layer $k$ and $b_k$ the corresponding bias vector. In the following, we denote the set of all trainable parameters of the network, i.e. weights and biases by $\theta$.
Usually, one chooses a set of training data points  $X_T=\menge{(x_i,y_i)}_{i\in T}$ with index set $T$ and evaluates the networks output using a loss function, for example the mean squared error between prediction and data
\begin{align}\label{eq_loss}
L_T(x,y;\theta) = \frac{1}{\vert T \vert}\sum_{i\in T} \norm{y_i -\mathcal{N}_\theta(x_i)}^2_2.
\end{align} 
Then one can set up the process of finding suitable weights, called training of the network, as an optimization problem
\begin{align}
\min_\theta \,  &L_T(x,y;\theta)
\end{align}
The optimization is often carried out with gradient-based algorithms, such as stochastic gradient descent~\cite{robbins1951} or related methods as ADAM~\cite{kingma2017adam}, which we use in this work.
\subsection{Data structure and normalization}\label{subsec_Data}
The structure of the underlying data is crucial for the construction of meaningful machine learning models. 
In the following we consider the primal and dual  minimal entropy closure optimization problem, review the characterization of the realizable set as well as a dimension reduction and finally describe helpful relations between the moment $u$, Lagrange multiplier $\alpha_u$, the entropy functional $h$ and the corresponding variables of reduced dimensionality.\\
The minimal entropy optimization problem in Eq.~\eqref{eq_entropyDualOCP} and the set of realizable moments $\mathcal{R}$ is  studied in detail by by~\cite{AlldredgeFrankHauck,Levermore, Curto_recursiveness,Hauck2008ConvexDA,Junk,Junk1999,JunkUnterreiter2001,Pavan2011GeneralEA}. The characterization of $\mathcal{R}$ uses the fact that the realizable set is uniquely defined by its boundaries~\cite{Kren1977TheMM}. First we remark that the realizable set $\mathcal{R}\subset\mathbb{R}^{\tilde{N}}$ of the entropy closure problem of order $N$ is generally an unbounded convex cone. To see this consider the moment of order zero, $u_0=\inner{f}$ for any kinetic density function $f\in F_m$, which can obtain values in $(0,\infty)$. For a fixed moment of order zero $u_0$, the subset of the corresponding realizable moments of higher order is bounded and convex~\cite{Kershaw1976FluxLN,Monreal_210538}. 
Consequently, we consider the normalized realizable set ${\mathcal{R}}$ and the reduced normalized realizable set ${\mathcal{R}}^r$
\begin{align}
   \mathcal{R}^n &=  \menge{u\in\mathcal{R}: u_0 =1}\subset\mathbb{R}^{\tilde{N}},\\
  \mathcal{R}^r &=  \menge{u^r\in\mathbb{R}^{N}:[1,{u^r}^T]^T\in{\mathcal{R}^n}}\subset\mathbb{R}^{\tilde{N}-1},
\end{align}
which are both bounded and convex~\cite{Kershaw1976FluxLN,Monreal_210538}. 
This approach is also used in computational studies of numerical solvers of the minimal entropy closure problem ~\cite{AlldredgeHauckTits}. We denote normalized moments and reduced normalized moments as
\begin{align}\label{eq_normalizedMoments}
    {u}^n &= \frac{u}{u_0} = [1,{u}_1^r,\dots,{u}_{\tilde{N}}^r]^T\in\mathbb{R}^{\tilde{N}},\\
    {u}^r &= [{u}_1^r,\dots,{u}_{\tilde{N}}^r]^T\in\mathbb{R}^{\tilde{N}-1}.
\end{align}
We establish some relations between the Lagrange multiplier $\alpha_u$, the Lagrange multiplier of the normalized moment ${\alpha}^n_u$ and of the reduced normalized moment $\alpha_u^r$,
\begin{align}\label{eq_alpha_n}
    {\alpha}_u^n &= \left[\alpha_{u,0}^r,\alpha_{u,1}^r,\dots,\alpha_{u,{\tilde{N}}}^r\right]^T\in\mathbb{R}^{\tilde{N}},\\
    \alpha_u^r &= \left[\alpha_{u,1}^r,\dots,\alpha_{u,{\tilde{N}}}^r\right]^T\in\mathbb{R}^{\tilde{N}-1}.
\end{align}
We define the reduced moment basis, which contains all moments of order $n=1,\dots,N$, as
\begin{align}
    m^r(\mathbf v) = [m_1(\mathbf v),\dots,m_{\tilde{N}}(\mathbf v)]^T,
\end{align}
since $m_0(v)=1$ is the basis function of order $0$.
For the computations we choose the Maxwell-Boltzmann entropy and a monomial basis, however the relations can be analogously computed for other choices of entropy function and moment basis. 
The Maxwell-Boltzmann entropy has the following definition, Legendre dual and derivative.
\begin{align}
    \eta(z) &= z\ln(z)-z, \qquad z\in D=\mathbb{R}_+ \\
    \eta'(z) &= \ln(z), \qquad\qquad\  z\in D=\mathbb{R}_+ \\
    \eta_*(y) &= \exp(y),  \qquad\quad\ \ y\in\mathbb{R}\\
    \eta_*'(y) &= \exp(y),  \qquad\quad\ \ y\in\mathbb{R}
\end{align}
In one spatial dimension, we have $\mathbf{v}=v\in\mathbb{R}$ and a monomial basis is given by $m(v)=[1,v, v^2,\dots]$.
Assuming knowledge about the Lagrange multiplier  $\alpha_u^r$ of the reduced normalized moment  we can derive an expression for $\alpha_0^r$ using the definition of the moment of order zero, 
\begin{align}
1={u}^n_0 = \inner{m_0 \eta_*'({\alpha}_u^n\cdot m)} = \inner{\exp({\alpha}_u^n\cdot m)} =\inner{\exp(\alpha_u^r\cdot m^r)\exp(\alpha_{u,0}^r\cdot m_0)},
\end{align}
which we can transform to 
\begin{align}\label{eq_alpha0_recons}
    \alpha_{u,0}^r = -\ln(\inner{\exp(\alpha_u^r\cdot m^r)})
\end{align}
using $m_0(\mathbf v)=1$.
This yields the complete Lagrange multiplier $\alpha_u^n$ of the complete normalized moment vector $u^n$. Finally, we use a scaling relation~\cite{AlldredgeHauckTits} to recover the Lagrange multiplier of the original moment vector $u$, by considering again the definition of the normalized moment 
\begin{align}\label{eq_recons_u}
     {u}^n = \inner{m\exp({\alpha}_u^n\cdot m)} =\inner{m\exp(\alpha_u^r\cdot m^r)\exp(\alpha_{u,0}^r)} 
 \end{align}
 and multiply both sides with $u_0>0$
\begin{align}
     u  = \inner{\exp(\alpha_u^r\cdot m^r)\exp(\alpha_{u,0}^r)u_0} =  \inner{\exp(\alpha_u^r\cdot m^r)\exp(\alpha_{u,0}^r+\ln(u_0))},
\end{align}
which yields the original Lagrange multiplier $\alpha_u$
\begin{align}\label{eq_scaled_alpha}
    \alpha_u = [\alpha_{u,0}^r + \ln(u_0),\alpha^r_{u,1},\dots,\alpha^r_{u,N}]^T.
\end{align}
This also implies that $\alpha_{u,i}=\alpha_{u,i}^r$ for all $i=1,\dots,\tilde{N}$.
For completeness, the entropy of the normalized moments $h^n =h(u^n)$ and the entropy $h(u)$ of the original moments have the relation
\begin{align}
    h(u) &= \alpha\cdot u - \inner{\exp(\alpha\cdot m)} \\
   &= u_0\left(\alpha^n\cdot u^n + \ln(u_0)\right)- \inner{\exp(\alpha^n\cdot m +\ln(u_0))} \\
   &= u_0\left(\alpha^n\cdot u^n + \ln(u_0)\right)- \inner{\exp(\alpha^n\cdot m }u_0) \\
    &= u_0h(u^n) + u_0\ln(u_0),
\end{align}
where we use Eq.~\eqref{eq_entropyFunctionalH} and \eqref{eq_scaled_alpha}. We denote the entropy of a normalized moment vector $h^n=h(u^n)$.
These scaling relations enable a dimension reduction for faster neural network training. Furthermore, we use these relations to integrate the neural network models, which are trained on ${\mathcal{R}}^r$, into the kinetic solver that operates on $\mathcal{R}$.
\subsection{Neural network approximations to the entropy functional}
In the following we present two ideas for neural network entropy closures, which are based on the results of~\cite{AlldredgeFrankHauck}, where the authors propose a regularized entropy closure with a corresponding entropy $h_\gamma$ and a regularization parameter $\gamma$,
\begin{align}
    h_\gamma = \inf_{g\in F_m} \inner{\eta(g)} + \frac{1}{2\gamma}\norm{\inner{mg}-u}^2.
\end{align}
In the limit, we have $h_\gamma\rightarrow h$ as $\gamma\rightarrow\infty$. The regularized entropy $h_\gamma$, which is twice differentiable and convex, acts as an entropy for the regularized moment system.
Furthermore, the authors have shown, that this approximation to the original entropy $h$ satisfies the conditions for the mathematical properties of the moment system presented in Section~\ref{sec_momMethods}, most importantly hyperbolicity, entropy dissipation and the H-Theorem.
In a similar manner, we present a twice differentiable, convex approximations $h^n_\theta$ to the moment to entropy map $h^n=h(u^n)$ of the normalized moment system. A neural network approximation, which we denote by $\mathcal{N}_\theta$, constructed with these properties in mind preserves the structural properties of the moment system. 
Assuming the neural network is trained, i.e. it approximates the entropy $h^n$ sufficiently well, we have the following relations,
\begin{align}\label{eq_hn}
    h^n_\theta =& \mathcal{N}_\theta(u^r) \approx h^n,\\
    \alpha^r_\theta  =& \frac{\intD}{\intD u^r}\mathcal{N}_\theta(u^r) \approx \frac{\intD}{\intD u^r}h^n = \alpha_u^r,\\
    \alpha^r_{\theta,0} =& -\ln(\inner{\exp(\alpha_\theta^r\cdot m^r)})\approx \alpha^r_{u,0}\\
    f_\theta =&\eta'_*(  \alpha^n_\theta \cdot m) \approx  \eta'_*(\alpha^n_u\cdot m) = f_u, \\
    u^n_\theta =& \inner{m\eta'_*(   \alpha^n_\theta \cdot m)} \approx  \inner{m\eta'_*(\alpha_u^n\cdot m)} = u^n, \label{eq_u_theta}
\end{align}
by using Eq.~\eqref{eq_entropyRecosntruction}, Eq.~\eqref{eq_derivH}, Eq.~\eqref{eq_alpha0_recons} and the definition of the moment vector.  \\
The idea of the second  neural network closure for the dual minimal entropy problem in Eq.~\eqref{eq_entropyDualOCP}, makes use of the following characterization of multivariate convex functions via montonicity of their gradients~\cite{berkovitz2003convexity}.
Let  $U\subset\mathbb{R}^N$ be a convex set. A function $G:U\rightarrow \mathbb{R}^d$ is \textit{monotonic}, if and only if $(G(x)-G(y))\cdot (x-y)\geq 0$ for all $x,y\in U$.
Let $g:U\rightarrow\mathbb{R}$ differentiable. Then $g$ is convex, if and only if  $\nabla g:U\rightarrow\mathbb{R}^d$ is monotonic. As a consequence, if the mapping $u^n\rightarrow\alpha_u^n$ is monotonic for all $u^n\in\mathcal{R}^n$, then the corresponding entropy functional is $h^n$ is convex in $u^n$. A trained monotonic neural network, that approximates the moment to Lagrange multiplier map, fulfills the following relations, 
\begin{align}
    \alpha^r_\theta =& \mathcal{N}_{\theta}{(u^r)} \approx  \alpha_u^r,\\
     \alpha^r_{\theta,0} =& -\ln(\inner{\exp(\alpha_\theta^r\cdot m^r)})\approx \alpha^r_{u,0}\\
    h^n_\theta=& \alpha^n_\theta\cdot u^n -  \inner{\eta_*(\alpha^n_\theta \cdot m)}\approx h^n, \\
    f_\theta =&\eta'_*(  \alpha^n_\theta \cdot m) \approx  \eta'_*(\alpha^n_u\cdot m) = f_u, \\
    u^n_\theta =& \inner{m\eta'_*(   \alpha^n_\theta \cdot m)} \approx  \inner{m\eta'_*(\alpha_u^n\cdot m)} = u^n.
\end{align}
We briefly examine the structural properties of a convex neural network based entropy closure. 
The invariant range property of $f_\theta$ depends solely on the range of $\eta_*'$. 
By definition of the Maxwell-Boltzmann entropy, the neural network based entropy closure is of invariant range, since $f_\theta(\mathbf v)=\exp(\alpha^n_\theta\cdot m(\mathbf v))>0$.
Interchanging the entropy functional by a neural network does not affect the conservation property of the moment system.
Consider the hyperbolicity requirement. In order to define the Legendre dual of $h$, it must be convex. Note, that $h^n$ is convex, if and only if $h$ is convex. In the proof of the hyperbolicity property, which is conducted in \cite{Levermore1996MomentCH} for $\alpha_u$ and $u$ as the system variable, $h''$, respectively $h_*''$, must be symmetric positive definite. As a consequence, $h^n$ and therefore the neural network $\mathcal{N}_\theta(u^n)$ must be strictly convex in $u^n$. Strict convexity of the entropy functional $h$ is the crucial requirement for the related properties entropy dissipation and the H-theorem~\cite{Levermore1996MomentCH} as well.

\subsubsection{Input convex neural network approximation of the entropy functional}
Convex neural networks have been inspected in \cite{Amos2017InputCN}, where the authors propose several deep neural networks that are strictly convex with respect to the input variables by design. The design is led by the following principles~\cite{boyd_vandenberghe_2004} that yield sufficient conditions to build a convex function. First, a positive sum of convex functions is convex. Second, let $f:\mathbb{R}^n\rightarrow\mathbb{R}$ be the concatenation of the functions $h:\mathbb{R}^k\rightarrow\mathbb{R}$ and $g:\mathbb{R}^n\rightarrow\mathbb{R}^k$. Then $f(x) = h(g(x))$ is convex, if $h$ is convex, $h$ is non-decreasing in each argument and all $g_{i=1,\dots,k}$ are convex.
Applying these conditions to the definition of a layer of a neural network, Eq.~\eqref{eq_nnKons}, yields that all entries of the weight matrix $W_k$ must be  positive in all layers except the first. Furthermore, the activation function of each layer must be convex.
The authors of~\cite{chen2019optimal} have shown, that such a network architecture with ReLU activations is able dense in the space of convex functions. They first show that an input convex network can approximate any maximum of afine functions, which itself can approximate any convex function in the limit of infinite layers.
However, in practice it turns out that very deep networks with positive weights have difficulties to train. The authors of~\cite{Amos2017InputCN} therefore modify the definition of a hidden layer in Eq.~\eqref{eq_nnKons} to
\begin{align}\label{eq_ICnnKons}
z_{k} &= \sigma(W_k^z z_{k-1} + W_k^x x + b_k^z), \qquad k=2,\dots, M, \\
z_{k} &= \sigma(W_k^x u + b_k^z), \qquad\qquad\qquad\, k=1, \label{eq_ICNN_layer1}
\end{align}
where $W_k^z$ must be non-negative, and $W_k^x$ may attain arbitrary values. 
We choose the strictly convex softplus function 
\begin{align}
    \sigma:\mathbb{R}\rightarrow\mathbb{R}_+,\quad \sigma(y) =  \ln(\exp(y)+1)
\end{align}
as the layer activation function for $k=1,\dots,M-1$
and a linear activation for the last layer, since we are dealing with a regression task. This leads to an at least twice continuously differentiable neural network.
Non-negativity can be achieved by applying a projection onto $\mathbb{R}_+$ to the elements of $W_k^z$ after a weight update. 
Next, we modify the first layer in Eq.~\eqref{eq_ICNN_layer1} to include two prepossessing layers.
We first zero center the input data w.r.t the mean vector of the training data set $\mu_u$, then we decorrelate the channels of the input vector w.r.t to the covariance matrix of the training data.
\begin{align}
    z_1^* &= u - \mu_u,\label{eq_meanshift}\\
    z_1^{**} &= \Lambda_u^T z_1*,\label{eq_decorrelation}\\
    z_{1} &= \sigma(W_k^x z_1^{**} + b_k^z),
\end{align}
where   $\Lambda_u$ is the eigenvector matrix of the covariance matrix of the training data set. 
The first two operations and the weight multiplication of the dense layer are a concatenation of linear operations and thus do not destroy convexity as well. Centering and decorrelation of the input data accelerate training, since the gradient of the first layer directly scales with the mean of the input data. Thus a nonzero mean may cause zig-zagging of the gradient vector ~\cite{LeCun2012}.
Lastly, we rescale and center the entropy function values of the training data. Note, that in the following we switch to notation corresponding to the entropy closure problem. We scale the entropy function values $h^n$ to the interval $[0,1]$ via
\begin{align}\label{eq_rescaleH}
    h^{n,*}= \frac{h^n-\min_{l\in T}h^n_l}{\max_{l\in T}h^n_l-\min_{l\in T}h^n_l},
\end{align}
which is equivalent to a shift and scale layer after the output layer of the neural network. Thus the gradient of the scaled neural network output $\alpha^*_\theta$ needs to be re-scaled to recover the original gradient,
\begin{align}\label{eq_rescaleAlpha}
    \alpha_\theta = \alpha^*_\theta\left(\max_{l\in T}h^n_l-\min_{l\in T}h^n_l\right)
\end{align}
Both operations are linear with a positive multiplicator, thus do not break convexity.
\\
\begin{figure}
  \centering
     \includegraphics[width=\textwidth]{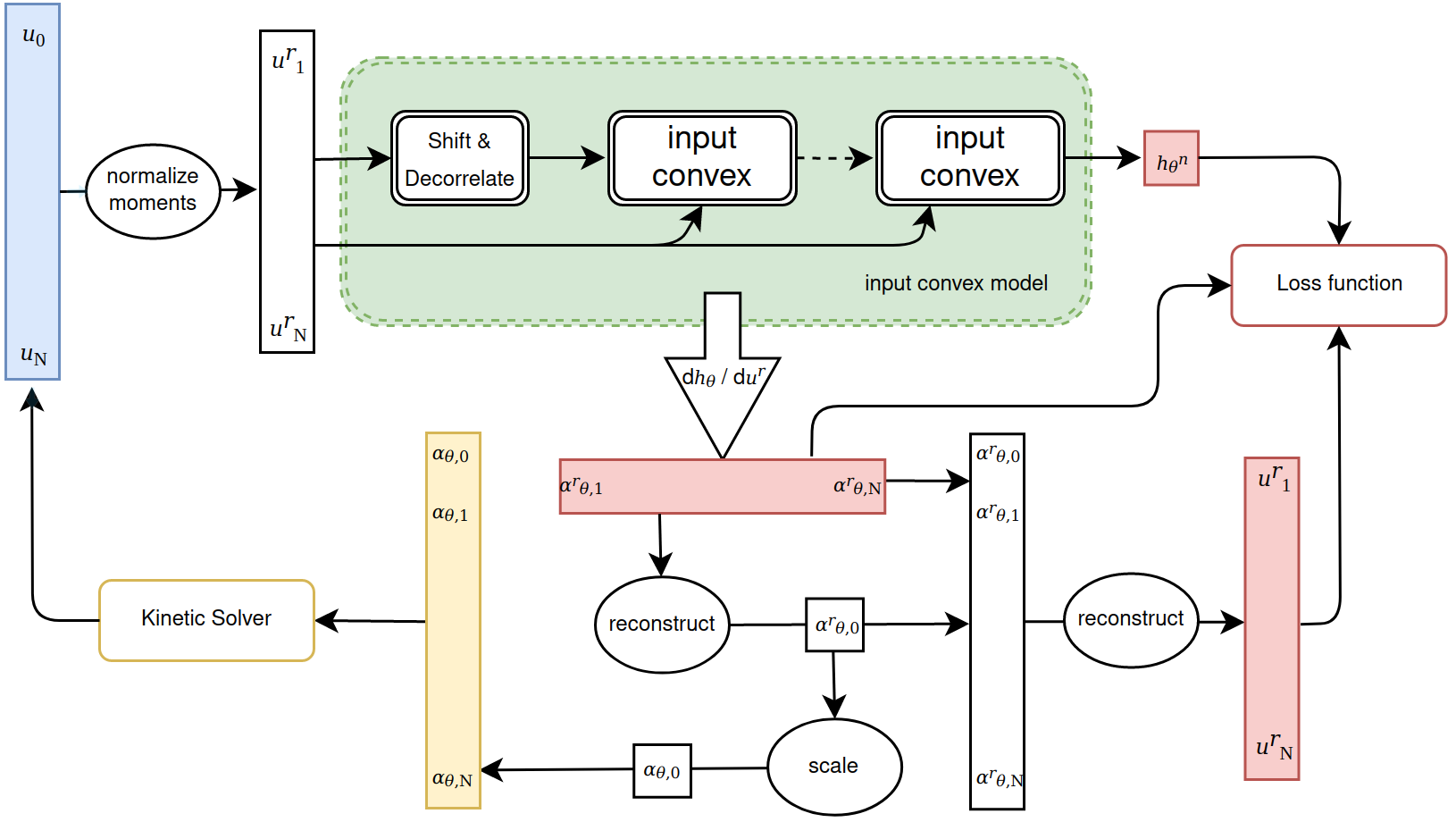}
  \caption{Input convex neural network closure. Model input vectors are depicted in blue. Red vectors are outputs, on which the model is trained on. When the trained model is employed, the yellow solution vector is used to construct the flux for the kinetic solver .}
  \label{fig_ICNN}
\end{figure}
We briefly describe the workflow of the neural network in training and execution time, which is illustrated in Fig.~\ref{fig_ICNN}.
For training a given input convex neural network architecture, we use a training data-set $X_T=\menge{u^r_i,\alpha_{u,i}^r,h^n_i}_{i\in T}$, where we first scale $h^n$ according to Eq.~\eqref{eq_rescaleH} and compute mean and covariance of $\menge{u_i^r}_{i\in T}$ for the shift and decorrelation layer. 
After a forward pass through the modified input convex neural network, we obtain $h_\theta^{n,*}$ and by automatic differentiation through the network w.r.t. $u^r$, we obtain $\alpha_{\theta}^{r,*}$, which we scale using Eq.~\eqref{eq_rescaleAlpha} to get $\alpha_\theta^r$. Using Eq.~\eqref{eq_alpha0_recons} we reconstruct $\alpha_{\theta,0}^r$ and therefore  $\alpha_{\theta}^n$ with Eq.~\eqref{eq_alpha_n}. The  normalized moments $u_\theta^n$ and the reduced normalized moments $u_\theta^r$ are computed using Eq.~\eqref{eq_recons_u}. The training loss function is evaluated on the mean squared error of $u_\theta^r$, $h_\theta^{n,*}$ and $\alpha_{\theta}^{r,*}$,
\begin{align}\label{eq_loss_icnn}
    L(u^r,\alpha_u^{r,*},h^{n,*};\theta) =  \frac{1}{\vert T\vert}\sum_{i\in T} \norm{h^{n,*}_i -h_{\theta,i}^{n,*}}^2_2 + \lambda\norm{\alpha_{u,i}^{r,*}- \alpha_{\theta,i}^{r,*}}^2_2 +\norm{u_i^r -u_{\theta,i}^r}^2_2.
\end{align}
The parameter $\lambda$ is used to scale the loss in $\alpha_u^r$ to the same range as the loss in $h^n$ and $u^r$.
Training the neural network on the Lagrange multiplier $\alpha_u^{r,*}$ corresponds to fitting the neural network approximation to the entropy functional $h^{n,*}$ in Sobolev norm. The authors of \cite{SobolevTraining} found that neural network models trained on the additional knowledge of derivative information archive lower approximation errors and generalize better. 
\\
When integrating the neural network in the kinetic solver, we gather the moments of all grid cells of the spatial domain from the current iteration of the used finite volume scheme. The moments are first normalized in the sense of Eq.~\eqref{eq_normalizedMoments},then  the predicted $\alpha_\theta^r$ are obtained in the same manner as in the training workflow. Afterwards, we use Eq.~\eqref{eq_alpha0_recons} and ~\eqref{eq_scaled_alpha} to obtain $\alpha_\theta$ corresponding to the non-normalized moments $u$. Finally, Eq.~\eqref{eq_entropyRecosntruction} yields the closure of the moment system, from which the numerical flux for the finite volume scheme can be computed.

\subsubsection{Monotone neural network approximation of the Lagrange multiplier}
No particular design choices about the neural network are made to enforce monotonicity, since the characterization of monotonic functions is not constructive. To the best of our knowledge, there exists no constructive definition of multidimensional monotonic function.
Instead we construct an additional loss function to make the network monotonic during training time. This is an important difference to the first approach, where the network is convex even in the untrained stage and on unseen data points.
\begin{definition}[Monotonicity Loss]
Consider a neural network $\mathcal{N}_\theta:x\mapsto y$. Let $X_T$ the training data set. The monotonicity loss is defined as 
\begin{align}
  L_\text{mono}\left(x,\theta\right)  =  \frac{1}{\abs{T}^2}\sum_{i\in T} \sum_{j\in T} \text{ReLU}\left(-\left(\mathcal{N}_\theta(x_i) - \mathcal{N}_\theta(x_j)\right)\cdot\left(x_i-x_j\right)\right).
\end{align}
The ReLU function is defined as usual,
\begin{align}
    \text{ReLU}(x)= 
    \begin{cases}
      x & \text{if $x>0$}\\
      0 & \text{if $x\leq0$}.
    \end{cases} 
\end{align}
\end{definition}
The monotonicity loss checks pairwise the monotonicity property for all datapoints of thetraining data set. If the dot product is negative, the property is violated and the value of the loss is increased by the current dot product. This is a linear penalty function and can be potentiated by a concatenation with a monomial function.
Note, that we only validate the monotonicity of the networkpointwise in a subset of the training data.
 As a consequence, the mathematical structures of the resulting moment closure is only preserved in an empirical sense, i.e. if the realizable set and more importantly, the set of Lagrange multipliers is sampled densely.
The resulting neural network architecture is illustrated in Fig.~\ref{fig_DirectNet}. 
Normalization and the meanshift and decorrelation layers in Eq.~\eqref{eq_meanshift} and Eq.~\eqref{eq_decorrelation} is implemented analogously to the input convex neural network. The core network architecture consists of a number of $M$ ResNet blocks. The ResNet architecture has been successfully employed in multiple neural network designs for multiple applications and was first presented in~\cite{HeZRS15}. The ResNet blocks used in this work read as
\begin{subequations}\label{eq_mono_layer}
\begin{align}
  z_{k}^1 &= \text{BN}(z_{k-1}), \\
  z_{k}^2 &= \sigma( z_{k}^1 ), \\
  z_{k}^3 &= W_k^* z_k^2 + b_k^*, \\
  z_{k}^4 &= \text{BN}(z_{k}^3), \\
  z_{k}^5 &= \sigma( z_{k}^1 ), \\
  z_{k}^6 &= W_k^{**} z_k^2 + b_k^{**}, \\
  z_{k} &= z_k^6+z_{k-1},\label{eq_skipConnection}
  \end{align}
\end{subequations}
with the idea, that the skip connection in Eq.~\eqref{eq_skipConnection} mitigates the gradient vanishing problem for deep neural networks. 
Furthermore, we include a batch normalization (BN) layer  in front of each activation, which reduces the problem internal covariance shift~\cite{IoffeS15}, that many deep neural network structures suffer from, and which slows down the training. Batch normalization is performed by applying pointwise the following two transformation to the previous layers output $z_k$,
\begin{align}
    z_k^* &= \frac{z_{k-1}-\mathbb{E}[z_{k-1}]}{\sqrt{\text{Var}[z_{k-1}]+\epsilon}},\\
    z_k &= \theta_0  z_k^* + \theta_1,
\end{align}
where $\theta_0$ and $\theta_1$ are trainable weights and $\mathbb{E}[z_{k-1}]$ and $\text{Var}[z_{k-1}]$ denote the expectation value and the variance of the current batch of training data, respectively.\\
One transforms the network output $\alpha_\theta^r$ to the values of interest $\alpha_\theta$ and $u_\theta^r$ analogously to the input convex network design. The entropy functional $h_\theta^r$ directly computed from $u_\theta^n$ and $\alpha_\theta^n$ using Eq.~\eqref{eq_entropyFunctionalH}. Training data rescaling and integration in the kinetic solver follow the ideas of the input convex network design.
The batchwise monotonicity loss is calculated using $u^r$ and $\alpha_\theta^r$, the gradient of the convex entropy functional $h^r$. 
The loss function for the network training becomes
\begin{align}\label{eq_nwLoss}
    L(u^r,\alpha_u^{r,*},h^{n,*};\theta) =  \frac{1}{\vert B\vert}\sum_{i\in B}\left( \norm{h^{n,*}_i -h_{\theta,i}^{n,*}}^2_2 + \norm{\alpha_{u,i}^{r,*}- \alpha_{\theta,i}^{r,*}}^2_2 +\norm{u_i^r -u_{\theta,i}^r}^2_2\right) + L_\text{mono}\left(u^r,\theta\right).
\end{align}
\begin{figure}
  \centering
     \includegraphics[width=\textwidth]{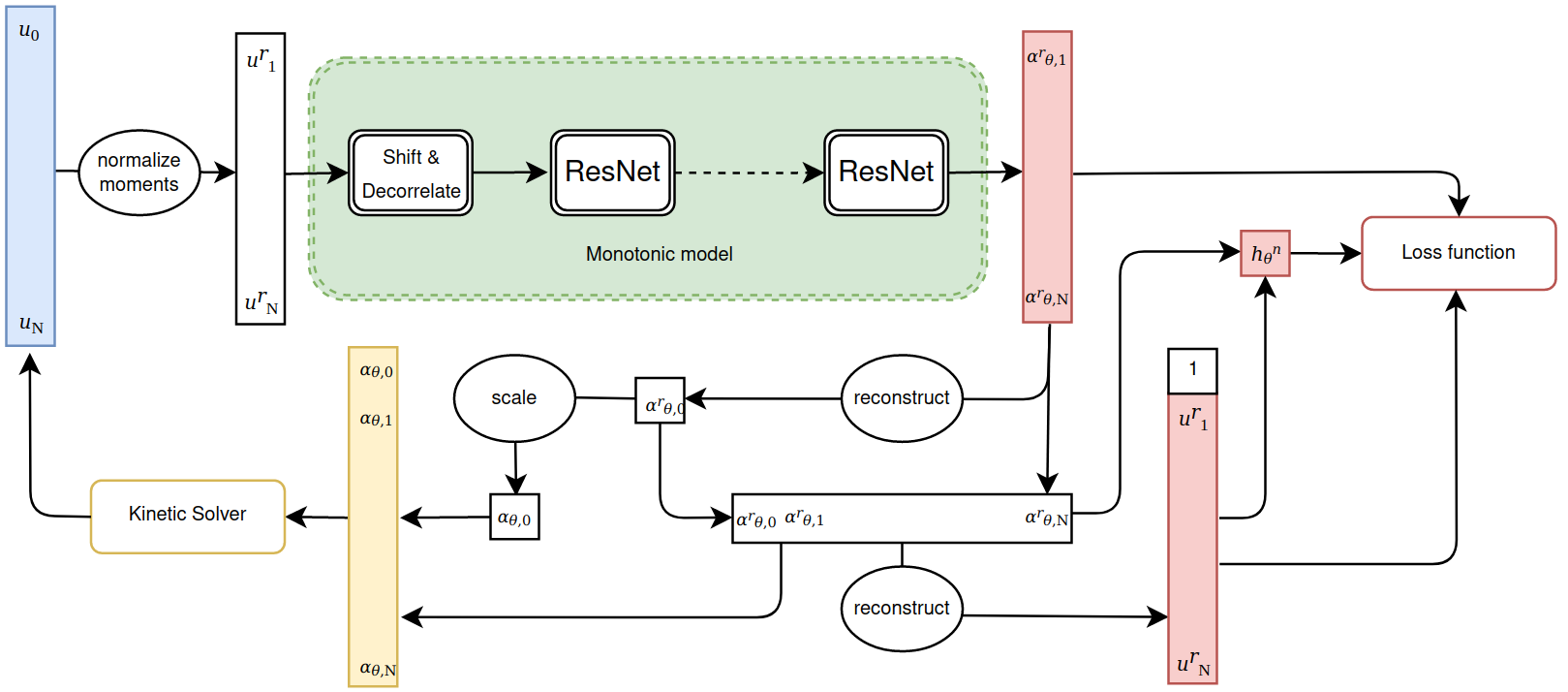}
  \caption{Input convex neural network closure. Model input vectors are depicted in blue. Red vectors are outputs, on which the model is trained on. When the trained model is employed, the yellow solution vector is used to construct the flux for the kinetic solver .}
  \label{fig_DirectNet}
\end{figure}
\section{Training Data and the generalization gap}\label{sec_data}
In this section, we present methods to generate training data for the neural network entropy closures and construct a local bound to the generalization gap for the approximated gradient of a input convex neural network.
\subsection{Data generation}\label{sec_data_gen}
In contrast to many applications of neural networks, the minimal entropy closure is a self contained problem with a clear data to solution map. Furthermore, the set of potential inputs ${\mathcal{R}^r}$ to the neural network is bounded and convex. This provides more options to sample training data than common machine learning applications. The training data distribution heavily influences the trained model and therefor the generalization gap~\cite{math_understanding_weinan}. The generalization gap is defined as 
\begin{align}
    \abs{L\left(X_T,\theta^*\right)-L\left(X,\theta^*\right)},
\end{align}
where $\theta^*=\min_\theta L\left(X_T,\theta,\right)$ is the set of parameters, that minimizes the training loss. The generalization gap describes the performance difference of the neural network with parameters $\theta^*$ between the training data set $X_T$ and any real world data $X$, i.e. the perfomance on unseen data. Thus we are left with a modelling decision about the data generation. \\
In related work~\cite{porteous2021datadriven}, the Lagrange multiplier $\alpha_u^r$ is sampled from a uniform grid in a cube  $[\alpha^r_\text{min},\alpha^r_\text{max}]^{\tilde{N}-1}\subset\mathbb{R}^{\tilde{N}-1}$ and then Eq.~\eqref{eq_recons_u} and Eq.~\eqref{eq_alpha0_recons} is used to reconstruct the corresponding moments $u^r\in\mathcal{R}^r$. In~\cite{SADR2020109644}, the authors sample $\alpha_u$ analogously to ~\cite{porteous2021datadriven} before reconstructing the kinetic density using Eq.~\eqref{eq_entropyRecosntruction}. However they iteratively update $\alpha$ until the reconstructed kinetic density has zero mean and unit variance, then they compute the moments of this kinetic density.\\
A popular method for generating data for neural network models that interact with numerical differential equation solvers is to generate the training data by direct simulation, see e.g.~\cite{huang2021machine,XIAO2021110521,Lei_2020}.
The idea of this concept is, to concentrate the training data generation on regions, which are likely to occur in real world. This is done by running simulation configurations similar to those expected in the application case. 
One expects that the model then performs better in a corresponding simulation test case than a model trained on general data. However, the model error might be much higher when encountering out of sample data and the model is specifically tailored to a small range of application configurations.
Another way to sample data using simulation is to use a Fourier series with random coefficients~\cite{huang2021machine2} to generate initial and boundary conditions. \\
In the following we present two data sampling strategies that take advantage of the structure of the data to solution map. We investigate the generalization gap for the prediction of $\alpha_\theta^r$ using input convex neural networks and derive a local error bound for the predicted $\alpha_\theta^r$ on unseen data. A further point of interest is the control over the boundary distance for a given sampling method.\\
\subsection{The boundary of the normalized realizable set}
The entropy minimization problem of Eq.~\eqref{eq_entropyDualOCP} becomes increasingly difficult to solve near the boundary of the realizable set $\partial\mathcal{R}$~\cite{AlldredgeHauckTits}. Close to $\partial\mathcal{R}^r$, the condition number of the  Hessian matrix of the entropy functional $h$ in Eq.~\eqref{eq_entropyFunctionalH} can become arbitrarily large, which causes numerical solvers to fail rather unforgivingly. This situation appears for moments of highly anisotropic distributions, vaccuum states, where $f(x,\cdot,t)=0$ or in the presence of strong sources~\cite{AlldredgeHauckTits}. At the boundary $\partial\mathcal{R}^r$, the Hessian matrix of $h$ is singular, and the minimal entropy problem has no solution. 
In the space of Lagrange multipliers, this translates to $\alpha_u^r$ growing beyond all bounds, which leads to numerical instabilities when computing the reconstruction of $u$. 
The simplest case of the minimal entropy closure, the $1D$ $M_1$ closure, already incorporates these difficulties. We can see in Fig.~\ref{fig_m1_1D_closure}a) the map $u_1^n\mapsto(\alpha^n_{0,u},\alpha^n_{1,u})$ and in Fig.~\ref{fig_m1_1D_closure}b the minimal entropy functional $h(u^n)$.  Since $\alpha^n_u$ is the gradient of $h$ with respect to $u^n$, the minimal entropy functional $h$ becomes steeper as $u^n_1$ approximates $\partial\mathcal{R}^r$. 
\begin{figure}   
    \centering
    \begin{minipage}{0.49\textwidth}
    \centering
        \includegraphics[width=\textwidth]{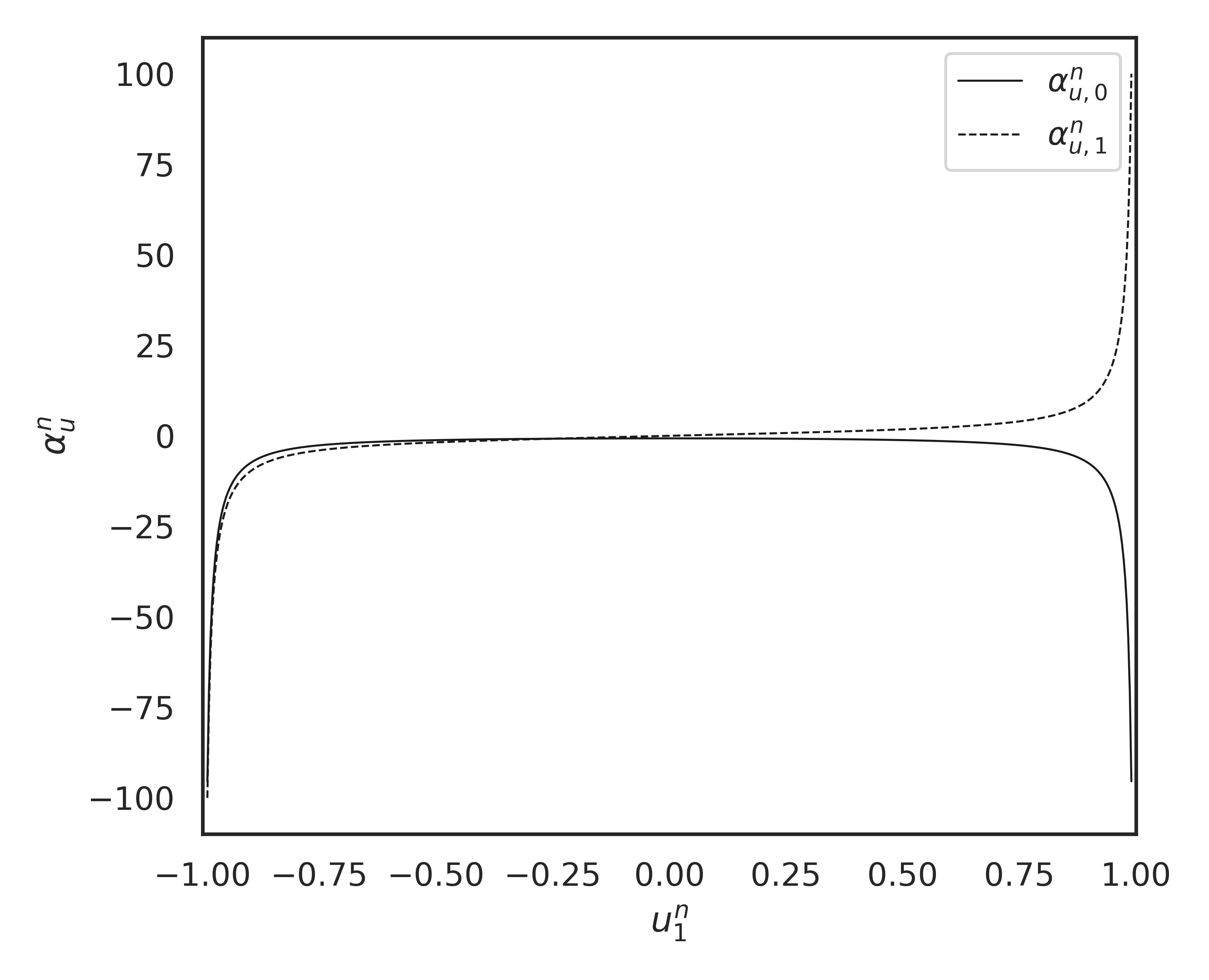}
            {\\a) $\alpha^n_u$ over $u^n_1$}
    \end{minipage}
     \begin{minipage}{0.49\textwidth}
       \centering
        \includegraphics[width=\textwidth]{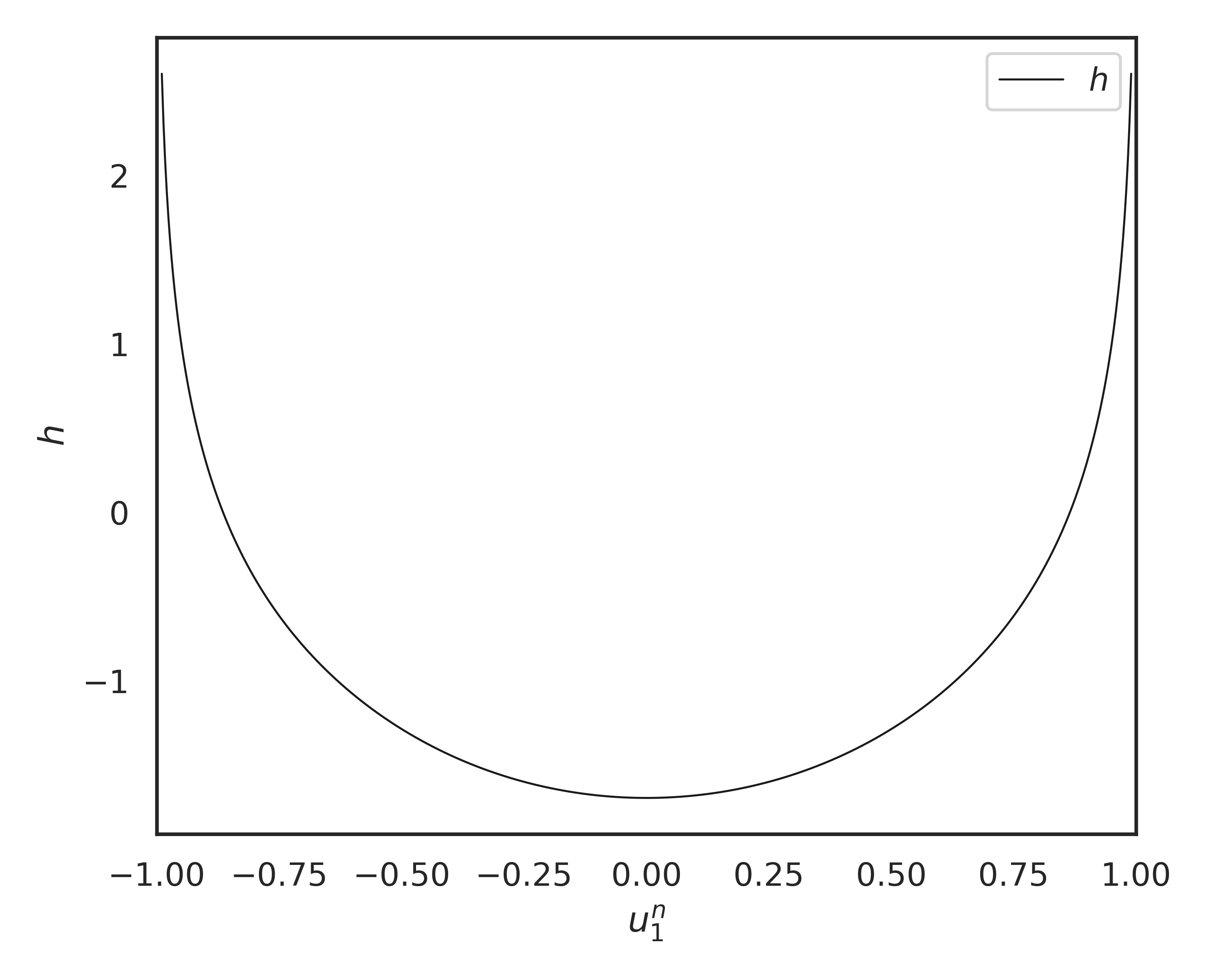}
            {\\b) $h$ over $u^n_1$} 
    \end{minipage}
    \caption{Data to solution maps for the $1D$ $M_1$ closure}
    \label{fig_m1_1D_closure}
\end{figure}
Note that both network architectures require this computation of Eq.~\eqref{eq_alpha0_recons} and \eqref{eq_recons_u} and thus need to evaluate the exponential during training time, which induces high vulnerability for numerical overflows, especially, when the networks are in the first iterations of the training process. A further critical issue for the training of neural networks is the fact that a wide range of output values causes the exploding gradient problem during the gradient descent for the weight updates. No matter if we sample $u$ and then compute $\alpha$ or vice versa, a sampling strategy must incorporate a meaningful distance measure to $\partial\mathcal{R}$. \\
Let us first consider proximity to the boundary in ${\mathcal{R}^r}$ directly. There exist extensive studies about the characterization of the boundary $\partial\mathcal{R}^r$ and we use results by Kershaw~\cite{Kershaw1976FluxLN} and Monreal~\cite{Monreal_210538}. For the Maxwell-Boltzmann entropy and a monomial basis, $\mathcal{R}^r$ can be described in one spatial dimension, i.e. $V,X\subset\mathbb{R}^1$ up to order $N=4$ using the inequalities
\begin{subequations}\label{eq_kershaw}
\begin{align}
    1&\geq u^r_1 \geq -1, \\
    1&\geq {u}^r_2 \geq ({u}^r_1)^2,\\
    {u}^r_2 -\frac{({u}^r_1-{u}^r_2)^2}{1-{u}^r_1} &\geq {u}^r_3 \geq - {u}^r_2 + \frac{({u}^r_1+{u}^r_2)^2}{1+{u}^r_1}, \\
    {u}^r_2 - \frac{({u}^r_1 -{u}^r_3)^2}{(1-{u}^r_2)} &\geq {u}^r_4 \geq  \frac{({u}^r_2)^3+({u}^r_3)^2-2 {u}^r_1 {u}^r_2 {u}^r_3}{{u}^r_2-({u}^r_1)^2},
\end{align}
\end{subequations}
whereas higher moment order moments can be characterized using the more general results in~\cite{Curto_recursiveness}.
Equation~\eqref{eq_kershaw} gives direct control over the boundary $\partial{\mathcal{R}^r}$, since equality in one or more of the equations describes a boundary of the normalized realizable set. In this case, the distance measure to $\partial\mathcal{R}^r$ is the norm distance. An example for normalized moments of the $M_2$ closure in $d=1$ spatial dimensions with norm boundary distance $0.01$ is shown in Fig.~\ref{fig_dataDist}a) and the corresponding Lagrange multipliers are shown in Fig.~\ref{fig_dataDist}b). Note, that in Fig.~\ref{fig_dataDist}a),c) and e), $\partial\mathcal{R}^r$ is displayed by the dotted black line.
More general results for arbitrarily high order moments in one spatial dimension can be found in~\cite{Kershaw1976FluxLN}. In three spatial dimensions necessary and sufficient conditions have been constructed by~\cite{Monreal_210538} for up to order $N\leq2$, but a full characterization of $\partial\mathcal{R}$ remains an open problem~\cite{lasserre}. \\
From a numerical point of view, it is interesting to construct a notion of distance to $\partial\mathcal{R}^r$ directly in the space of Lagrange multipliers, since it turns out that the magnitude of the $\norm{\alpha_u^r}$ has implications on the numerical stability of the neural network training process. A first idea consists of a norm bound of $\alpha_u^r$, i.e. $\norm{\alpha_u^r}<M<\infty$~\cite{AlldredgeFrankHauck,porteous2021datadriven,SADR2020109644}, which yields a convex subset of Lagrange multipliers.
Fig.~\ref{fig_dataDist}d) shows a uniform distribution of $\alpha_1^r$ and $\alpha_2^r$, where $\alpha^r_i\in[-40,40]$, and  Fig.~\ref{fig_dataDist}c) displays the corresponding reconstructed moments $u^n$.
\begin{figure}
    \centering
   \begin{minipage}{0.49\textwidth}
   \centering
       \includegraphics[width=\textwidth]{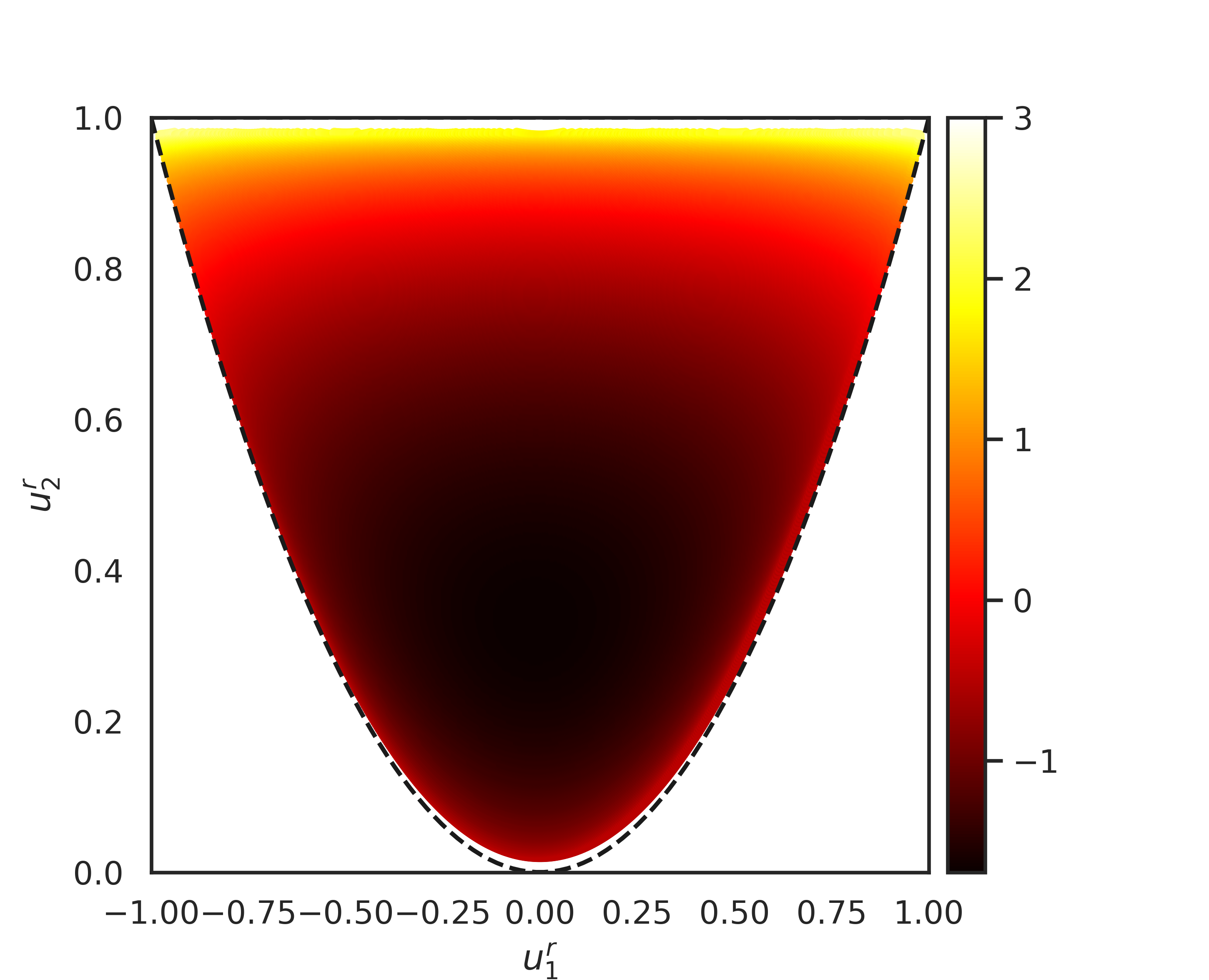}
               {\\a)$\mathcal{R}^r$, using uniform grid sampling of $u^r$}
   \end{minipage}
   \begin{minipage}{0.49\textwidth}
   \centering
       \includegraphics[width=\textwidth]{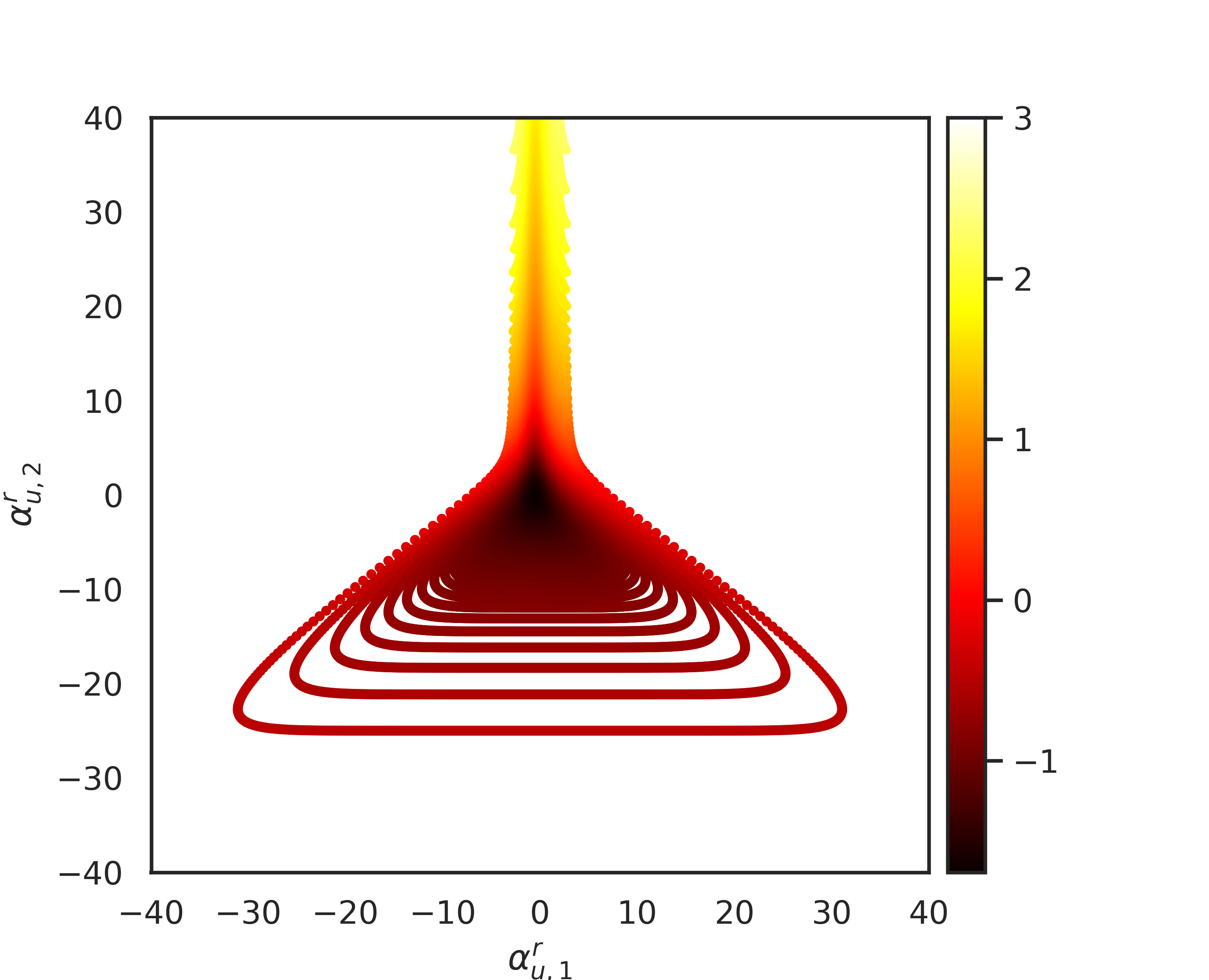}
               {\\b) $\alpha_u^r$, using uniform grid sampling of $u^r$}
   \end{minipage}
   \begin{minipage}{0.49\textwidth}
     \centering
        \includegraphics[width=\textwidth]{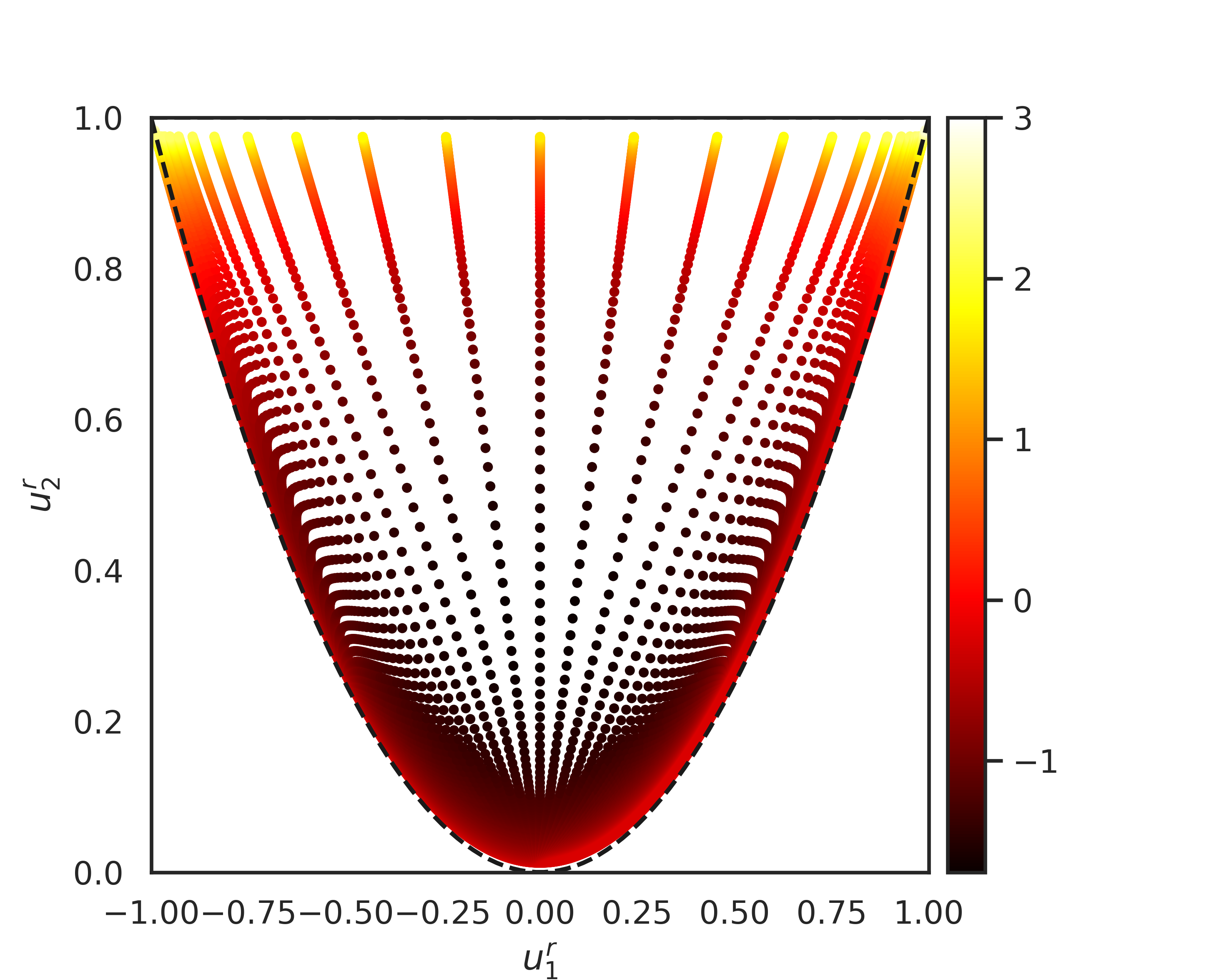}
        {\\c) $\mathcal{R}^r$, using uniform grid sampling of $\alpha_u^r$ with $\infty$ norm bound}
   \end{minipage}
   \begin{minipage}{0.49\textwidth}
     \centering
        \includegraphics[width=\textwidth]{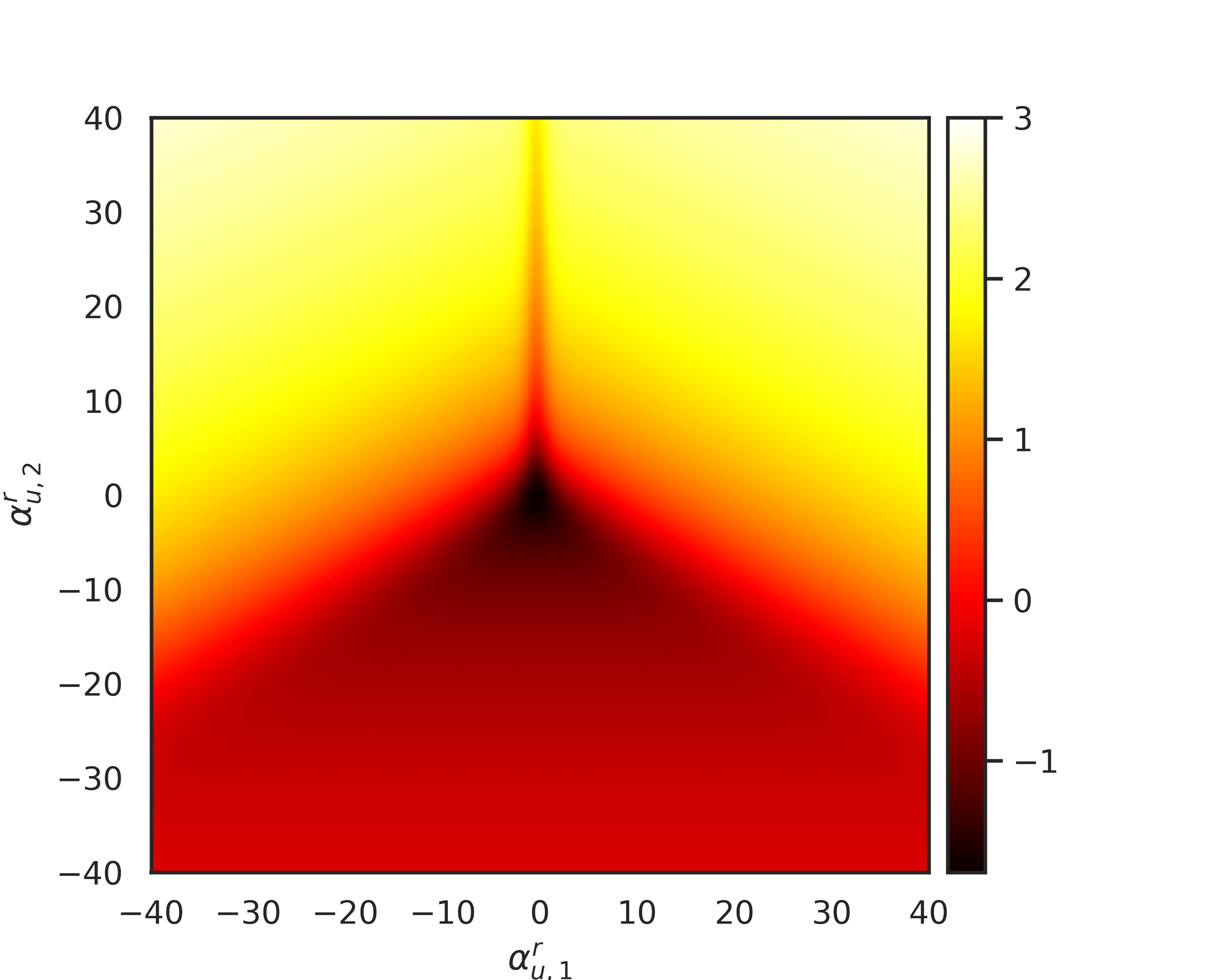}
        {\\d)  $\alpha_u^r$, using uniform grid sampling of $\alpha_u^r$ with $\infty$ norm bound}
   \end{minipage}
    \begin{minipage}{0.49\textwidth}
     \centering
        \includegraphics[width=\textwidth]{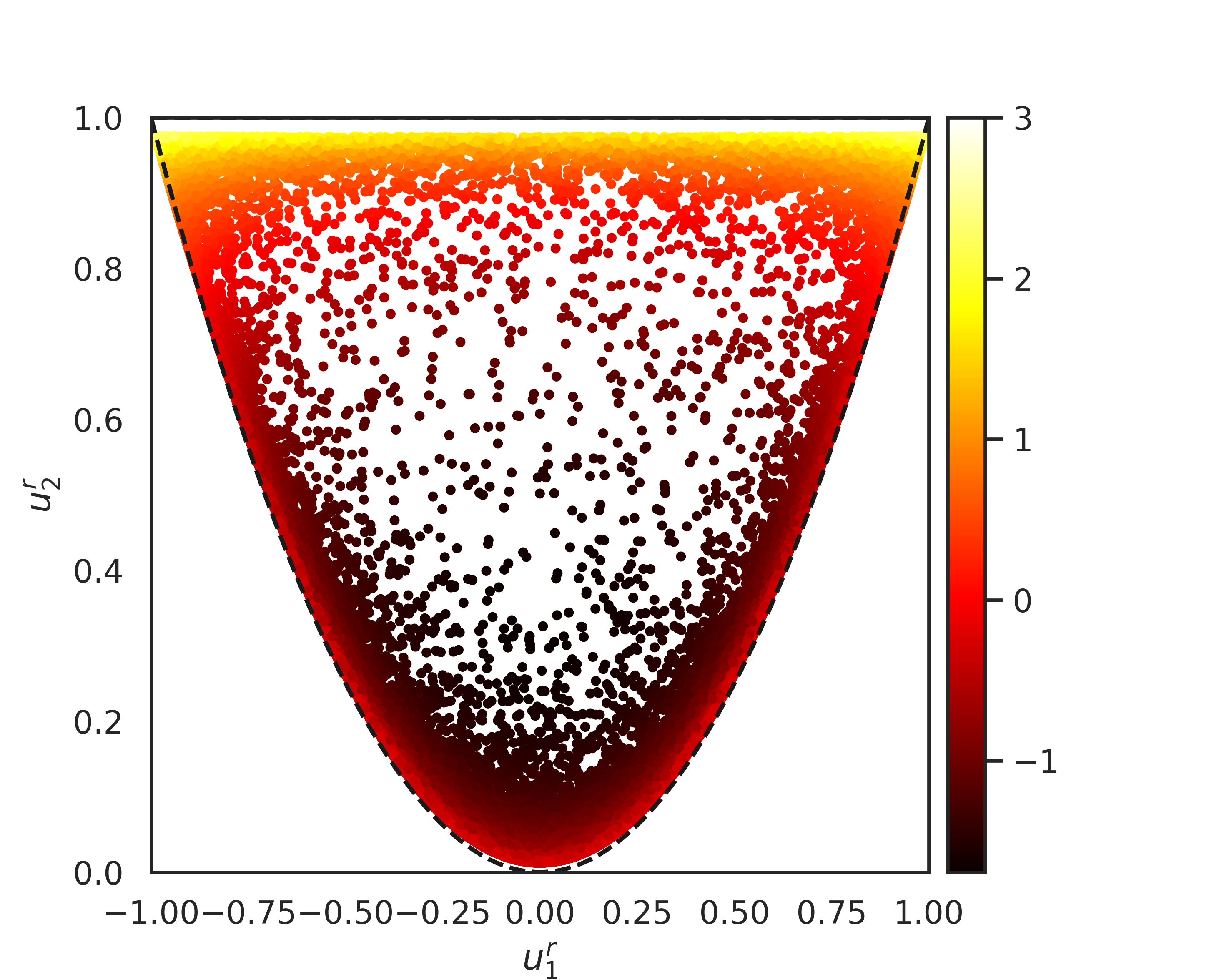}
        {\\e) $\mathcal{R}^r$, using uniform low-discrepancy sampling of $\alpha_u^r$ with eigenvalue bound}
   \end{minipage}
    \begin{minipage}{0.49\textwidth}
     \centering
        \includegraphics[width=\textwidth]{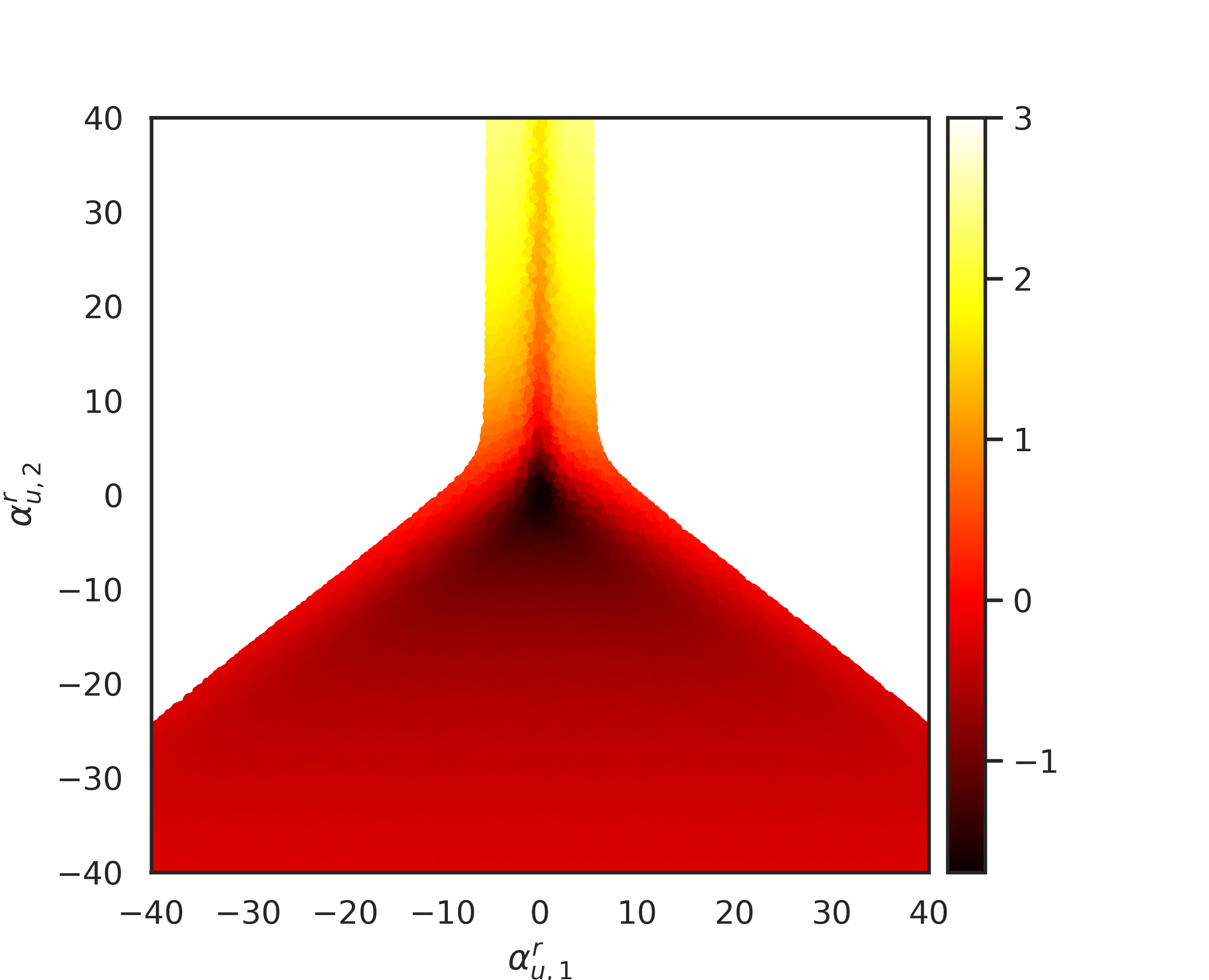}
        {\\f)  $\alpha_u^r$, using uniform low-discrepancy sampling of $\alpha_u^r$ with eigenvalue bound}
   \end{minipage}
    \caption{Scatter plots of $70000$ data points for the $1D$ $M_2$ model with data generated from different sampling strategies. The color bar indicates the value of the minimal entropy functional $h$.}
    \label{fig_dataDist}
\end{figure}
However, this approach gives no control over the boundary distance and produced very biased data distributions of $u^r$. A comparison of Fig.~\ref{fig_dataDist}c) and d) shows, that two thirds of the sampled moments are concentrated in the regions near $u^r=(-1,1)$ and $u^r=(1,1)$, which correspond to high entropy values $h>1.5$ and are colored yellow. In contrast, Fig.~\ref{fig_dataDist}a) and b) show that there are no samples in the regions $\alpha^r_{u,1}>10$ and $\abs{\alpha^r_{u,2}}>10$, since the corresponding moments $u^r$ are too close to the boundary $\partial\mathcal{R}^r$. As a conclusion, the second sampling strategy does not produce data with uniform distance to $\partial\mathcal{R}^r$.\\ 
Another approach is to use the condition number of the Hessian, of $h$ w.r.t $\alpha_u^n$ directly. Since the Hessian
Eq.~\eqref{eq_entropyDualOCP} w.r.t $\alpha$
\begin{align}\label{eq_hessian_dualOCP}
    H(\alpha_u^n)=\inner{m\times m\eta_*(\alpha_u^n\cdot m)},
\end{align}
is symmetric and positive definite, the condition number is the ratio of the biggest and smallest eigenvalue. The Hessian $H(\alpha_u^n)$ is singular at $\partial\mathcal{R}^r$, so the smallest possible eigenvalue $\lambda_{\min}$ is $0$, and we use $\lambda_{\min}$ to measure the distance to the boundary of the realizable set. Figure~\ref{fig_dataDist}e) and f) show a uniform sampling, where $\alpha_u^r$ is sampled with $\lambda_{\text{min}}>1\mathrm{e}{-7}$.
Note, that on the one hand, the near boundary region of $\mathcal{R}^r$ is more densely sampled than the interior, compare Fig.~\ref{fig_dataDist}a) and e), whereas there is no over-representation of the regions near $u^r=(-1,1)$ and $u^r=(1,1)$ and the set of sampled Lagrange multipliers, see Fig.~\ref{fig_dataDist}f), is similar in shape to the Lagrange multipliers in Fig.~\ref{fig_dataDist}b).

\subsection{Generalization gap for input convex neural networks trained in Sobolev norm}
In this sections we present our findings to the generalization gap for the derivative approximation of a convex neural network $\mathcal{N}_\theta$ that approximates a convex function $f^*$. The network is trained (at least) in Sobolev norm, i.e. the training loss reads
\begin{align}
    L\left(X_T,\theta^*\right) = \frac{1}{\abs{T}}\sum_{i\in T}\left( \norm{f^*(x_i)-\mathcal{N}_{\theta^*}(x_i)}_2^2 + \norm{\nabla f^*(x_i)-\nabla\mathcal{N}_{\theta^*}(x_i)}_2^2\right),
\end{align}
when evaluating the loss over the whole data set.
In the following, we assume that the network is trained, i.e. $L\left(X_T,\theta^*\right) = 0$. Thus we have
\begin{align}
    f^*(x_i) = \mathcal{N}_\theta(x_i), \qquad \nabla f^*(x_i) = \nabla \mathcal{N}_\theta(x_i) \qquad \forall x_i\in X_T.
\end{align}
Furthermore, let the sampling domain $X\subset\mathbb{R}^d$ be convex and bounded and the neural network be convex by design. We are interested in the generalization gap of the derivative neural network with respect to its input variable.
To this end, we consider the local generalization gap of the neural network when using  $d+1$ training data points $X_d=\menge{x_0,\dots,x_d}$, if the sampling space $X\subset\mathbb{R}^d$ has dimension $d$. Let $\mathcal{C}(X_d)$ be the convex hull of $X_d$ and $x^*\in\mathcal{C}(X_d)$, which we call the point of interest. We assume w.l.o.g $x^*=0$; if this does not hold, one can consider the shifted setting $\mathcal{C}^\dag(X_d)=\mathcal{C}(X_d)-x^*$, $f^\dag=f^*(\cdot + x^*)$, $x^\dag =  x - x^*$ instead. 
Using the characterization of a monotonic function, we define the set $A$
\begin{align}\label{eq_feasibleGradients}
    A=\menge{v\in\mathbb{R}^d \vert v\cdot x_i \leq \nabla f^*(x_i)\cdot x_i, i=0,\dots,d}
\end{align}
\begin{figure}
    \centering
        \includegraphics[width=\textwidth]{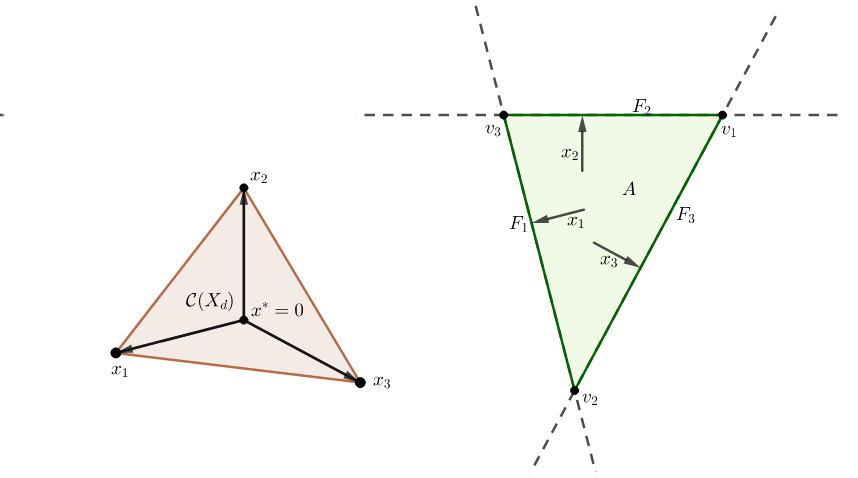}
    \caption{Illustration of the convex hull of the training points $\mathcal{C}(X_d)$ (left) and the set of feasible gradients $A$ (right) for $d=2$. The normal vectors to the faces $F_i$ are the vectors of the training points $x_i$.}
    \label{fig_feasible_polygon}
\end{figure}
which is the dual polygon defined by the gradients at the sampling points and the point of interest and can be seen in Fig.~\ref{fig_feasible_polygon}. $A$ contains all values which the gradient of a convex function that has fixed gradients at the sampling points $x\in X_d$ can attain at the point of interest $x^*$.
\begin{theorem}\label{theo_A}
Let $f^*$ be convex, $x^*=0$ the point of interest in the interior of $\mathcal{C}(X_d)$. Then $A$ is a bounded polyhedron, whith $d+1$ faces, defined by $F_i=\menge{v\in\mathbb{R}^d\vert v\cdot x_i = \nabla f^*(x_i)\cdot x_i}$ and vertices $v_i=\bigcap_{j\not=i}F_j$.
\end{theorem}
\begin{proof}
The proof is structured in two parts. First, we show that the vertices $v_i\in\mathbb{R}^d$ are well defined, if $x^*$ is element of the interior of $\mathcal{C}(X_d)$. Second, we show that all $v_i\in A$. Thus any convex combination of $v_i$ is in $A$ and therefore, $A$ is defined by a (bounded) polyhedron with vertices $v_i$. \\
\textbf{1.} We show that $v_i$ are well defined. 
First, if the point of interest is element of the interior of $\mathcal{C}(X_d)$, then all $x_i\in{X_d}$ are linearly independent.
The boundary of the set of feasible gradients with respect to the sampling point $x_i$ and the point of $x^*$ interest consists of the hyperplane given by
\begin{align}
    F_i=\menge{v\in\mathbb{R}^d\vert v\cdot x_i = \nabla f^*(x_i)\cdot x_i}. 
\end{align}
Clearly, if all $x_i\not=0$ are linearly independent, no hyperplanes are parallel or lie in each other. The proper intersection of $d$ hyperplanes in $\mathbb{R}^d$ yields a single point,
\begin{align}\label{eq_vertex_polyhedron}
    v_i = \bigcap_{j\not=i} F_j.
\end{align}
which we define as vertex $v_i\in\mathbb{R}^d$, that touches all hyperplanes except $F_i$. \\
\textbf{2.} We show that all $v_i\in A$. This means, that we have to show
\begin{align}
    v_j\cdot x_i \leq \nabla f^*(x_i)\cdot x_i,\qquad \forall i,j =0,\dots,d
\end{align}
By the definition of $v_j$, we have
\begin{align}
    v_j\in F_i ,\qquad j\not = i,
\end{align}
so we are only concerned with  
\begin{align}
     v_i\cdot x_i \leq \nabla f^*(x_i)\cdot x_i.
\end{align}
We start by stating an auxiliary statement. Let $p_i^j=v_j-v_i$ for $i\not=j$. If $X_d$ is linearly independent and $x^*=0$ is in the interior  of $\mathcal{C}(X_d)$, then 
\begin{align}
    \text{sign}(p_i^j\cdot x_i) =  \text{sign}(p_k^l\cdot x_k),\qquad \forall i\not=j,k\not=l
\end{align}
Linear independence of $x_i\in X_d$ and $x^*=0$ being in the interior  of $\mathcal{C}(X_d)$ translates to
\begin{align}
    0 = \sum_{i=0}^N a_ix_i, \qquad a_i>0.
\end{align}
We have
\begin{align}
    p_i^j\cdot x_i&=\frac{-1}{a_i}\sum_{m\not=i}a_m\left(v_j-v_i\right)\cdot  x_m\\
    &=\frac{-1}{a_i}\left(\sum_{m\not=i,j}a_m\left(v_j-v_i\right)\cdot x_m + a_j\left(v_j-v_i\right)\cdot x_j\right)\\
    &=\frac{-1}{a_i}\left(\sum_{m\not=i,j}a_m\left(v_j\cdot x_m-v_i\cdot x_m\right) + a_j\left(v_j-v_i\right)\cdot x_j\right)\\
    &=\frac{-1}{a_i}\left(\sum_{m\not=i,j}a_m\left(\nabla f^*(x_m)\cdot x_m-\nabla f^*(x_m)\cdot x_m\right) + a_j(v_j-v_i)\cdot x_j\right)\\
     &=\frac{-1}{a_i}a_j(v_j-v_i)x_j = \frac{a_j}{a_i}(v_i-v_j)x_j = \frac{a_j}{a_i}p_j^i\cdot x_j,
\end{align}
where we use the definition of the Face $F_m$.
Since $\frac{a_j}{a_i}$ is positive $\text{sign}(p_i^j\cdot x_i) =  \text{sign}(p_j^i\cdot x_j)$ follows for all $i\not=j$.
Assume $p_i^j\cdot x_i>0$ and $p_i^h\cdot x_i <0$. Then
\begin{align}
    v_j\cdot x_i>v_i\cdot x_i>v_h\cdot x_i.
\end{align}
Thus, we have
\begin{align}
   0< (v_j-v_h)\cdot x_i =\nabla(f^*(x_i)\cdot x_i - \nabla(f^*(x_i)\cdot x_i = 0,
\end{align}
which is an contradiction to monotonicity of the gradient. Thus, 
\begin{align}
    \text{sign}(p_i^j\cdot x_i) =  \text{sign}(p_i^k\cdot x_i) = \text{sign}(p_k^i\cdot x_k) = \text{sign}(p_k^l\cdot x_l),\qquad \forall i\not=j,k\not=l.
\end{align}
This means, that all face normals $x_i$ are either facing outward of the polyhedron defined by the vertices $\menge{v_i}$ or all face inward. 
Assume inward facing normals, then for each face of the polyhedron created by $A$, the feasible set is the half space outside the current face of the polyhedron. Due to convexity the polyhedron defined by $\menge{v_i}$, this would imply, that $A=\emptyset$, which contradicts continuity of the gradient of $f^*$
Thus we have outward facing normals.
Finally, we have
\begin{align}
    0<(v_j-v_i)\cdot x_i = \nabla f^*(x_i)\cdot x_i - v_i\cdot x_i,
\end{align}
and thus $ v_i\cdot x_i < \nabla f(x_i)\cdot x_i $, i.e. $v_i\in A$ for all $i$. Thus $A$ is indeed a polygon defined by the vertices $v_i$. By convexity, the polyhedron $A$ contains all feasible gradients of the point of interest. \\
\end{proof}
A direct consequence of Theorem~\ref{theo_A} is, that we get an local upper bound for the generalization gap of the gradient of an input convex network trained on a given training data set $X_T$
\begin{align}
    \norm{\nabla f^*(x) -\nabla\mathcal{N}_\theta(x) } \leq \text{diam}(A_{x^*}),
\end{align}
where $A_{x^*}$ is the polyhedron of feasible gradients w.r.t the point of interest $x^*$ and the local training points $X_d$. A first conclusion is that the diam$(A)$ does not depend on the distance between the point of interest and any of the local training data points $X_d$, since by definition of $A$ in Eq.~\eqref{eq_feasibleGradients}, one can divide by the norm of $x_i-x^*$ on both sides of the inequality for the boundary of $A$. Thus in the following we assume normalized $x_i$. \\
The following theorem gives a more precise representation of diam$(A_{x^*})$.
\begin{theorem}\label{lem_A}
Let A be defined by Eq.~\eqref{eq_feasibleGradients} and $v_i$ be defined by Eq.~\eqref{eq_vertex_polyhedron}. Let the relative vectors $x_i$ have unit length and $v_i$ is the vertex opposing the face $F_i$. 
The matrix $X_i= [x^n_0,\dots,x^n_{i-1},x^n_{i+1},\dots,x^n_{d}]^T$ contain the vectors of normalized sampling points relative to the point of interest $x_*$, i.e. $x^n_i = x_i/\norm{x_i}_2$. \\
Furthermore, let $b_i=[\nabla f^*(x_0)\cdot x^n_0, \dots, \nabla f^*(x_{i-1})\cdot x^n_{i-1},f^*(x_{i+1})\cdot x^n_{i+1},\dots, \nabla f^*(x_{d})\cdot x^n_{d}]^T$ be a vector. Under the assumptions of Theorem~\ref{theo_A}, the vertex $v_i$ is given by
\begin{align}
    X_iv_i=b_i
\end{align}
Additionally, we can estimate the distance between two vertices $v_i$ and $v_j$ by
\begin{align}
    \norm{v_i - v_j}_2 \leq \left(
\left\|X_i^{-1}\right\|\,+\,\left\|X_j^{-1}\right\|\right) C_{x^*},
\end{align}
where $C_{x^*}=\max_{k,l}\norm{\nabla f^*(x_k)-\nabla f^*(x_l)}_2$ and $\norm{X_i^{-1}}$ denotes the corresponding operator norm of $X_i^{-1}$.
\end{theorem}
\begin{proof}
By definition of $v_i=  \bigcap_{j\not=i} F_j$ and the fact that we can divide Eq.~\eqref{eq_feasibleGradients} by $\norm{x_i}$ we get the linear systems. Let for $\xi\in\mathbb{R}^d$. 
\begin{align}
    C_\xi = \max_{k=0, \dots, d} \norm{\nabla f(x_k)-\xi }_2
\end{align}
Then we have
\begin{align}
    \abs{x_j \cdot (v_i-\xi)  }= \abs{x_i\cdot \left(\nabla f^*(x_i)-\xi\right)}\leq \norm{x_i}_2 C_\xi = C_\xi\quad \forall i=0,\dots,d
\end{align}
since $x_i$ has unit norm.
Thus each entry of the vector $X_iv_i$ has an absolute value smaller than $C_\xi$. We interpret $X_i$ as an linear operator mapping $(\mathbb{R}^d,\norm{\cdot}_2)\rightarrow (\mathbb{R}^d,\norm{\cdot}_\infty)$.
$X_i=[x_0,\dots,x_{i-1},x_{i+1},\dots,x_{d}]^T$ is invertible, if $x^*$ is in the interior of $\mathcal{C}(X_d)$ and defines a mapping $(\mathbb{R}^d,\norm{\cdot}_\infty)\rightarrow (\mathbb{R}^d,\norm{\cdot}_2)$. 
Consequently, we can estimate
\begin{align}
    \norm{X_i (v_i-\xi)}_\infty &\leq C_\xi, \\
    \norm{v_i - \xi}_2 &\leq \norm{X_i^{-1}} C_\xi.
\end{align}
Finally we get
\begin{align}
    \norm{v_i -v_j}_2 \leq  \norm{v_i -\xi}_2 + \norm{\xi - v_j}_2 \leq  
    \left(
\left\|X_1^{-1}\right\|+\left\|X_2^{-1}\right\|\right) C_\xi,
\end{align}
We can choose  $\xi = \nabla f^*(x_l)$ s.t.
\begin{align}
    \max_{k=0, \dots, d} \norm{\nabla f^*(x_k)-\nabla f^*(x_l) }_2 =  \max_{k,l=0, \dots, d} \norm{\nabla f^*(x_k)-\nabla f^*(x_l) }_2 =: C_{x^*}
\end{align}
\end{proof}
Let us draw some conclusions from the proof.
First, we have as a direct consequence 
\begin{align}
    \norm{\nabla f^*(x) -\nabla\mathcal{N}_\theta(x) } \leq \text{diam}(A_{x^*}) \leq \left(
\left\|X_i^{-1}\right\|\,+\,\left\|X_j^{-1}\right\|\right) C_{x^*}.
\end{align}
First, $\text{diam}(A)\rightarrow\infty$, if $\text{dist}(x^*,\partial\mathcal{C}(X_d))\rightarrow 0$, since the normals of at least two neighboring boundaries of $A$ become (anti)parallel to each other.\\ 
Additionally we find, that for a fixed point of interest and angles between the local training points, the size of $\text{diam}(A)$ depends only on the norm distance of $\nabla f^*(x_i)$, $i=0,\dots,d$, which is encoded in the definition of $C_{x^*}$. The smaller the norm distance of the gradients of the sample points, the smaller gets $C_{x^*}$.\\
Lastly, we consider the case of $x^*$ on the boundary of the convex hull of the local training points. Then one selects a new set of local training points, such that $x^*$ is in the interior of their convex hull. In the application case of the neural entropy closure, the input data set is ${\mathcal{R}^r}$, which is bounded and convex. Thus, the argument is viable everywhere except at the boundary of the convex hull of all training data, assuming a suitable distribution of training data points. Remark, that the polyhedron can be shrunken by including more training points to the set $X_d$.
\subsection{Sampling of the normalized realizable set}\label{sec_sampling}
As a consequence of the above considerations, we generate the training data $X_T$ by sampling reduced Lagrange multipliers in a set
\begin{align}\label{eq_sampling_set}
    B_{M,\tau}=\menge{\alpha_u^r : \norm{\alpha_u^r}<M \cap \lambda_{\text{min}}(H(\alpha_u^n))>\tau}
\end{align}
using rejection sampling. 
Uniform distribution of $\alpha_u^r$ is important to achieve a uniform norm distance between the gradients of the approximated function $h$, which reduces the generalization gap. \\
The number generation method has an non negligible influence on the quality of training data, as Fig.~\ref{fig_dataDist}c) and e) display. The former are moments $u^r$ generated by a uniform grid sampling of $\alpha_u^r$, and the latter by uniform sampling of $\alpha_u^r$ using a low-discrepancy sampling method. The deformed grid in $u^r$ consists near $u^r_2=1.0$ of very steep triangles of local training points $X_d$, that means that a point of interest is always close to the boundary of $\mathcal{C}(X_d)$ which implies a big diameter for the polyhedron of admissible gradients $A_{x^*}$.
Low-discrepancy sampling methods have a positive impact for neural network training, especially for high data dimensions~\cite{mishra2020,LOYOLAR201675}. 
\section{Numerical Results} \label{sec_numResults}
In this section, we present numerical results  and  investigate the performance of the neural entropy closure. First, we compare the performance of neural networks trained on data using the sampling strategy discussed in Section~\ref{sec_data}.
We conduct synthetic tests to measure the performance of the networks on the complete realizable set and the computational efficiency in comparison with a Newton optimizer, which is used in typical kinetic solvers.
Then, we employ the network in a $1D$ and $2D$ kinetic solver and compare the results with the benchmark solution in several simulation test cases. To ensure significance of the errors compared to the spacial discretization errors, we perform a convergence analysis of the neural network based and benchmark solvers.
\subsection{Neural network training}
In the following we evaluate the training performance of the neural network architectures, which  are implemented in Keras using Tensorflow 2.6 ~\cite{tensorflow2015-whitepaper} and can be found in the Github repository~\cite{NeuralEntropy}. \\
The neural networks are trained on a subset of $\mathcal{R}^r$ that corresponds with Lagrange multipliers sampled from the set $B_{M,\tau}$ of Eq.~\eqref{eq_sampling_set}. The data sampler can be found in the Github repostiroy~\cite{KITRT}.
$M$ and $\tau$ are chosen such that the neural network training is numerically stable, since for high absolute values of  $\alpha^n_i$, the term $u^n=\inner{m\exp{(\alpha_u^nm)}}$ leads to a numerical overflow in single precision floating point accuracy, if the neural network training is not yet converged.
In this sense, the high condition number   of the minimal entropy closure near $\partial\mathcal{R}$ translates to the neural network approximation.
The sampled data is split into training and validation set, where the validation consists of $10$\% randomly drawn samples of the total data.
Table~\ref{tab_val_losses_cpl} compares the validation losses of different neural network architectures after the training process has converged. The layout of a neural network is defined in the format width $\times$ depth. Width describes the number of neurons in one layer of the network and depth the number of building blocks of the network. A building block of the input convex neural network is one (convex) dense layer. A building block of the monotonic neural network architecture is described by Eq.~\eqref{eq_mono_layer}. In addition to these layers, each model is equipped with a mean shift~\eqref{eq_meanshift} and decorrelation~\eqref{eq_decorrelation} layer followed by a dense layer as a preprocessing head. After the core architecture of the format width $\times$ depth, one more layer of the respective core architecture with half the specified width and finally the output layer follows. The linear output layer of the input convex neural network design is one dimensional, since we approximate the entropy $h_\theta$ and the linear output layer of the monotonic network design has dimension $N$, where $N$ is the order of the moment closure and the length of the reduced Lagrange multiplier vector $\alpha_u^r$. The input convex network with output data scaled to the interval $[0,1]$ uses a ReLU activation, since we do not expect negative output values.\\
The networks are trained on an Nvidia RTX 3090 GPU in single-precision floating-point accuracy. 
For each network architecture, we present the mean squared and mean absolute error for all quantities of interest averaged over the validation data set. For the monotonic network, the monotonicity loss is additionally displayed. The converged  mean squared error on the validation set is in $\landauO(10^{-4})$ to $\landauO(10^{-6})$. These errors are in line with the findings of similar approaches, see~\cite{porteous2021datadriven}. In~\cite{adcock2021gap} the authors have found, that it is hard to train neural networks below $4$ digits of accuracy in single precision training. Further studies need to be conducted about the performance in double precision training of the proposed networks.
\begin{table}[htbp]
\centering
  	\caption{Validation losses for different moment closure} 
	\begin{tabular}{|l|l|l|l|l|l|l|} 
		\hline
		 Closure &\multicolumn{2}{c|}{$M_1$ $1$D}  &\multicolumn{2}{c|}{$M_2$ $1$D}   &\multicolumn{2}{c|}{$M_1$ $2$D}    \\ \hline
		 Architecture & convex & monotone & convex & monotone & convex & monotone \\ \hline
		 Layout                & $10\times 7$          & $30\times 2$ & $15\times 7$ & $50\times 2$ & $18\times 8$&  $100\times 3$ \\ \hline
		 MSE$(h^n,h^n_\theta)$ & $7.87\mathrm{e}{-7}$  & $2.09\mathrm{e}{-5}$ &  $1.33\mathrm{e}{-5}$&$5.04\mathrm{e}{-4}$ & $1.10\mathrm{e}{-6}$ & $4.01\mathrm{e}{-4}$ \\ \hline
		 MSE$(\alpha^r_u,\alpha^r_\theta)$ & $7.52\mathrm{e}{-4}$ &  $5.56\mathrm{e}{-6}$ & $2.81\mathrm{e}{-4}$ & $2.56\mathrm{e}{-4}$ &$3.39\mathrm{e}{-5}$  &$4.54\mathrm{e}{-5}$ \\ \hline
		 MSE$(u^r,u^r_\theta)$& $1.47\mathrm{e}{-6}$ &  $3.64\mathrm{e}{-6}$ & $2.81\mathrm{e}{-4}$& $1.27\mathrm{e}{-4}$&  $3.39\mathrm{e}{-5}$& $8.09\mathrm{e}{-5}$\\ \hline
		 $L_\text{mono}(u^r)$ & n.a. & $1.60\mathrm{e}{-14}$ & n.a. &$1.38\mathrm{e}{-14}$ & n.a. &$5.31\mathrm{e}{-16}$ \\ \hline
		 MAE$(h^n,h^n_\theta)$ & $7.57\mathrm{e}{-4}$ & $3.05\mathrm{e}{-3}$ & $3.11\mathrm{e}{-3}$& $1.66\mathrm{e}{-2}$ & $1.02\mathrm{e}{-3}$  &$1.49\mathrm{e}{-2}$\\ \hline
		 MAE$(\alpha_u^r,\alpha^r_\theta)$ & $1.26\mathrm{e}{-2}$  &  $9.10\mathrm{e}{-3}$ & $1.23\mathrm{e}{-2}$  &  $9.74\mathrm{e}{-3}$& $3.36\mathrm{e}{-3}$ & $4.35\mathrm{e}{-3}$\\ \hline
		 MAE$(u^r,u^r_\theta)$ & $9.59\mathrm{e}{-4}$ &  $1.51\mathrm{e}{-3}$  &$1.23\mathrm{e}{-2}$ & $7.96\mathrm{e}{-3}$ & $9.62\mathrm{e}{-4}$ &$7.02\mathrm{e}{-3}$  \\ \hline
	\end{tabular}  \label{tab_val_losses_cpl}
\end{table}
Notice, that the mean absolute error in $\alpha_u^r$ is significantly higher than the error in $h$ or $u^r$ for all input convex neural networks. The reason for this is again the high range of values, that $\alpha_u^r$ can attain. In case of the input convex neural network, the values are obtained by differentiating through the network with primary output $h_{\theta}$, and thus one always has a scaling difference between $h^n$ and $\alpha_u^r$ of about one order of magnitude. Therefor, the scaling parameter $\lambda$ of Eq.~\eqref{eq_loss_icnn} is set to be $\lambda=1/10$ to balance out the training.
\subsection{Synthetic test cases}
\begin{figure}
    \begin{minipage}{0.49\textwidth}
      \centering
       \includegraphics[width=\textwidth]{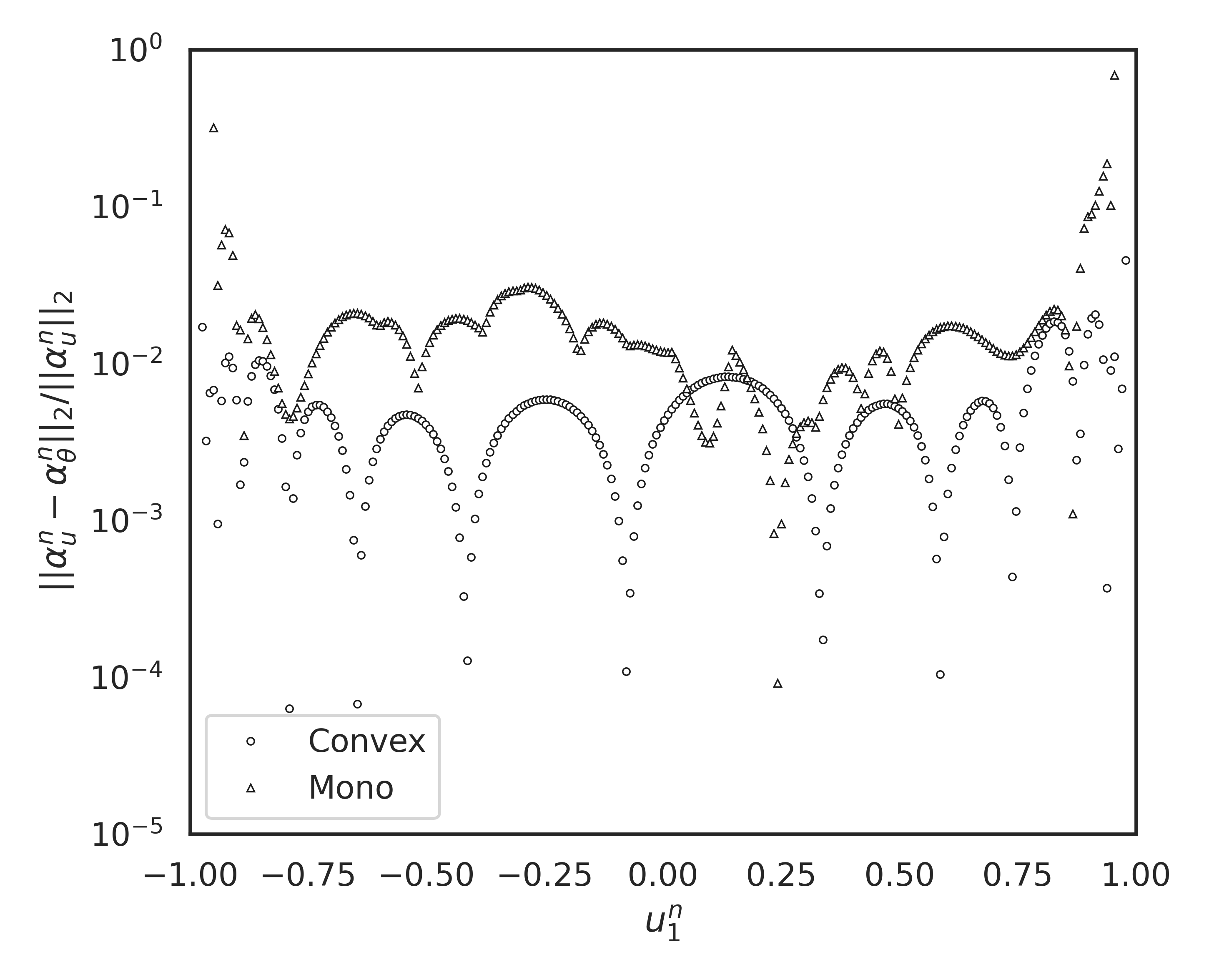}
            {\\a) Relative norm error of the prediction of $\alpha_u^n$ }
    \end{minipage}
    \begin{minipage}{0.49\textwidth}
      \centering
       \includegraphics[width=\textwidth]{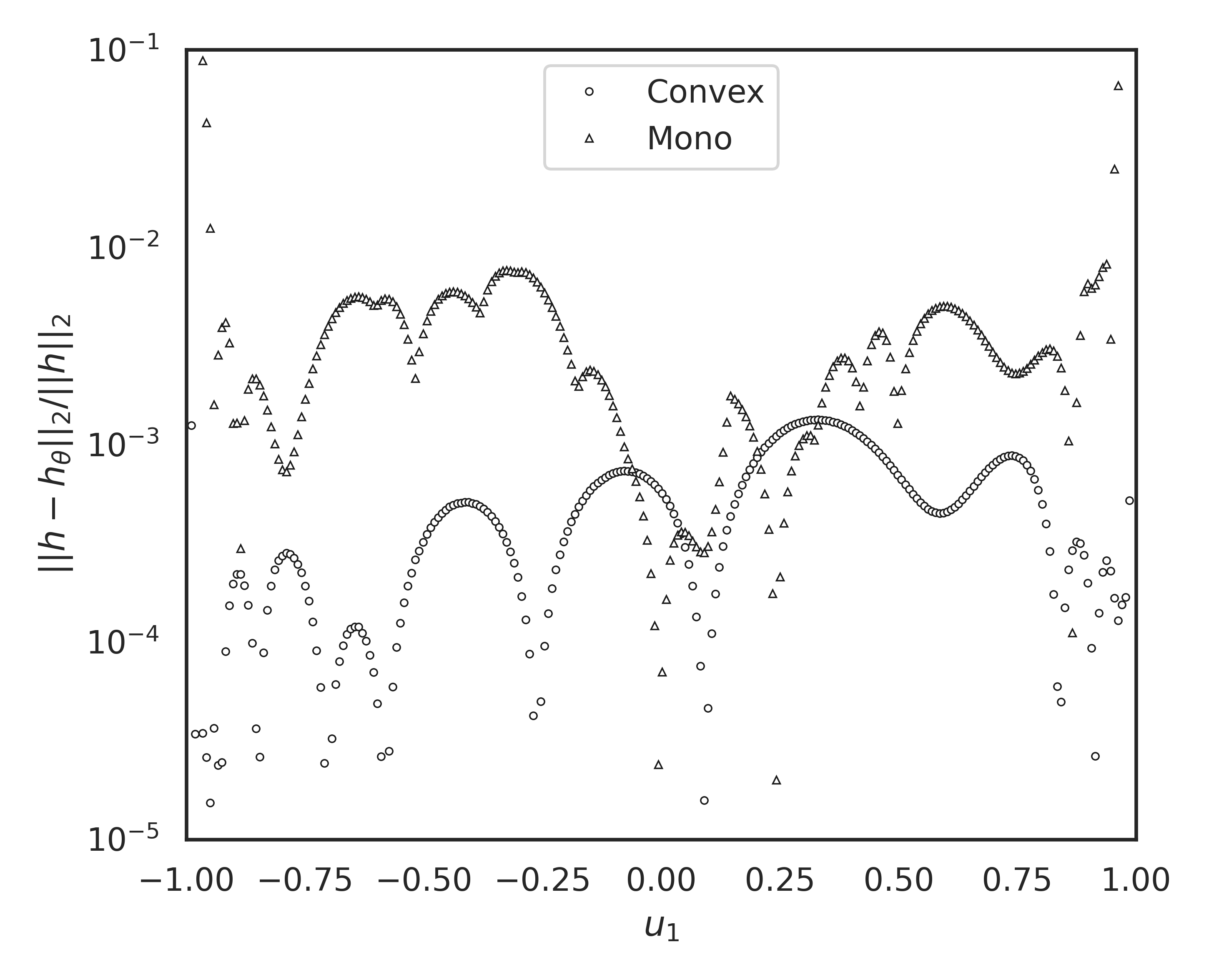}
            {\\b)  Relative norm error of the prediction of $h$}
    \end{minipage}\centering
     \begin{minipage}{0.49\textwidth}
      \centering
       \includegraphics[width=\textwidth]{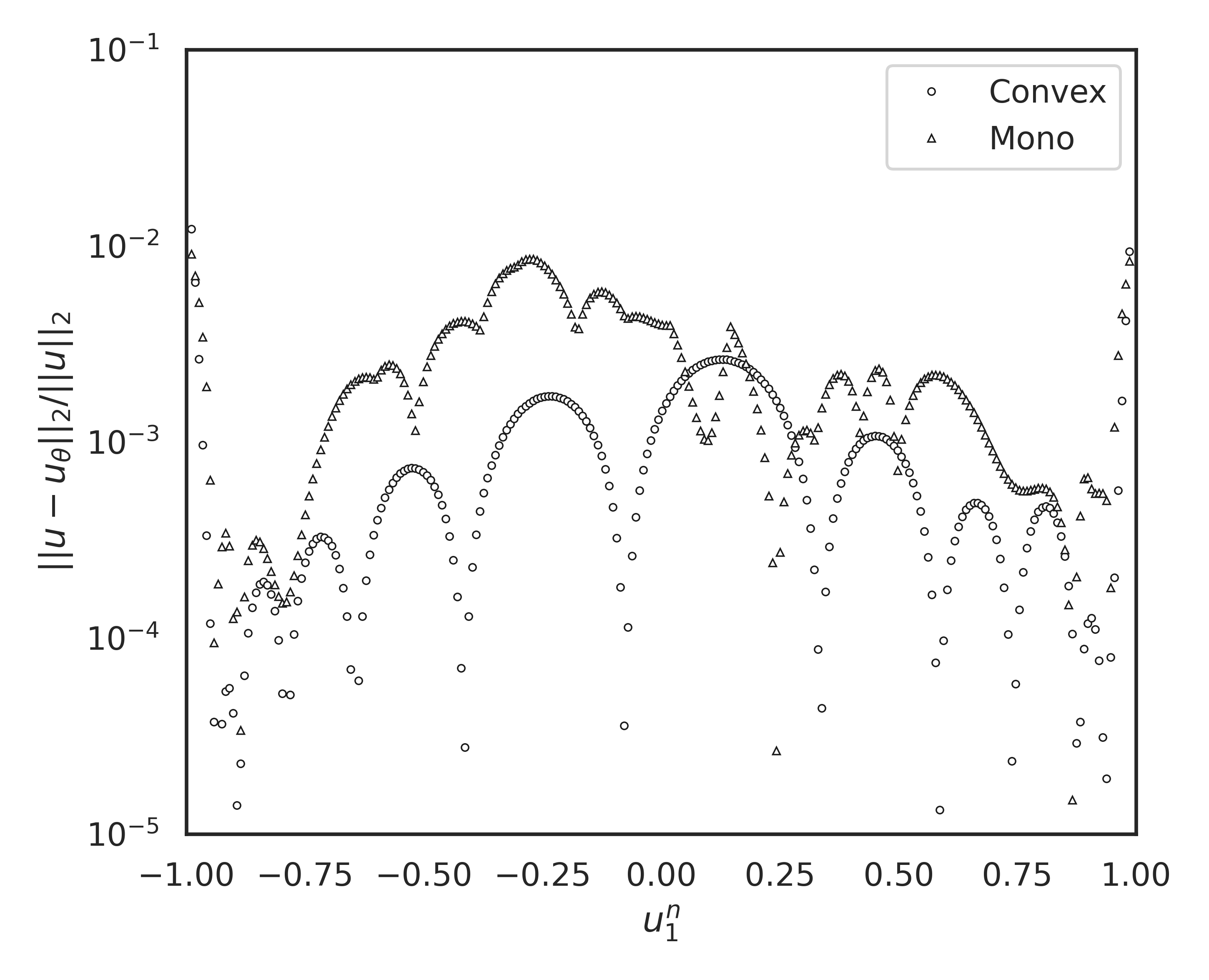}
            {\\c) Relative norm error of the reconstruction of $u^n$}
    \end{minipage}
    \caption{Relative test errors of both neural network architectures for the $1D$ $M_1$ closure. Distance to $\partial\mathcal{R}$ is $0.01$ }
    \label{fig_dataSampling1D}
\end{figure}
In this section, we consider again the $1D$ $M_1$ entropy closure, see Fig.~\ref{fig_m1_1D_closure}, and perform accuracy tests for the input convex and monotnonic neural network architecture.
The networks are trained on a the data set generated from $\alpha^r_{u,1}$ sampled from $[-50,50]$ using the discussed sampling strategy. Then, 
the networks are evaluated on twice as many samples in the displayed data range $u^n_1\in[-0.99,0.99]$ and $\alpha^n_{u,i}\in[-95,80]$, thus the extrapolation areas near the boundary consist of only unseen data and the interpolation area contains at least 50$\%$ unseen data.\\
The relative norm errors of the predictions of both network architectures can be seen in Fig.~\ref{fig_dataSampling1D}. 
Figure~\ref{fig_dataSampling1D}a) compares the input convex and monotonic network on the basis of their relative norm error in the Lagrange multiplier $\alpha^r_u$. Within the intervall $[-0.75,0.75]$ the relative error of the input convex neural network is in $\landauO(10^{-2.5})$ and increases by half an order of magnitude in the extrapolation area. The relative error of the monotonic architecture displays more fluctuation with a mean of $\landauO(10^{-2})$. In the extrapolation area, the error of the monotonic network increases by over an order of magnitude and is outperformed by the convex neural network.  
Remark, that the approximation quality declines as we approach $\partial\mathcal{R}^r$, which is expected, since the neural networks can not be trained close to the boundary and the output data $\alpha_u^r$ and $h$ grow rapidly in this region. \\
Figure~\ref{fig_dataSampling1D}b) displays the relative error in the entropy prediction $h_\theta$ for of the respective neural networks. 
The monotonic architecture exhibits a larger relative error in $h$, compared to the input convex architecture. This can by explained by the fact, that the input convex neural network directly approximates $h$, whereas the monotonic neural network reconstructs $h$ using $\alpha^r_\theta$ and $u_\theta$ and thus the approximation error of both variables influence the error in $h$.
In the extrapolation regime, one can see a similar error increase as in Fig.~\ref{fig_dataSampling1D}a). \\
Overall, both networks do not perform well near $\partial\mathcal{R}^r$, however, when we consult Fig.~\ref{fig_dataSampling1D}c), we see that the error in the reconstructed moment $u^n$ is below $10^{-2}$ for the input convex and the monotonic network, although the error in $\alpha^n_u$ is almost in the order of $10^0$ in this region. This shows, that the nature of the reconstruction map $u=\inner{m \exp(\alpha_u m)}$ mitigates the inaccuracy in $\alpha^n_u$ to some degree. The reconstructed moments $u^n_\theta$ experience less relative error in the interior of $\mathcal{R}^r$ than near the boundary.
For the stability of the solver however, the error in the reconstructed flux $\inner{vmf_u}$, which is Lipschitz continuous in $u$, is the most important quantity.\\
All in all, both network architecture are able to approximate the entropy closure within a reasonable error margin.
\subsection{Computational Efficiency}
In the following, we compare the computational efficiency of the neural network surrogate model and the Newton optimizer in an isolated, synthetic test case. We consider the $M_2$ closure in $1D$ and use only normalized moments. In contrast to the neural network, the performance of the Newton solver is dependent on the proximity of the moments $u^n$ to the boundary $\partial{\mathcal{R}^n}$, thus we consider three test cases. First, the moments are uniformly sampled in $\mathcal{R}^n$, second we only draw moments near the center of $\mathcal{R}^n$ and lastly, we use only moments in proximity to   $\partial{\mathcal{R}^n}$, where the minimal entropy problem is hard to solve and has a high condition number. The Newton solver is implemented in the KiT-RT~\cite{KITRT} framework. In the kinetic scheme, there is one optimization problem for each grid cell in a given time step. Furthermore, optimization problems of different grid cells are independent of each other.
Options for the parallelization of the minimal entropy closure within a kinetic solver differs, whether we choose an CPU based implementation of a Newton solver or the GPU optimized tensorflow backend for the neural network closure. 
A meaningful and straight-forward way to parallelize the minimal entropy closure on CPU is to employ one instance of the Newton optimizer per available CPU-core that handles a batch of cells. On the other hand, we interpret the number of grid cells as the batch size of the neural network based entropy closure. Shared memory parallelization is carried out by the tensorflow backend. 
For comparability, we set the accuracy tolerance of the Newton solver to single-precision floating point accuracy, since the trained neural networks have a validation error between $\landauO(10^{-4})$ and $\landauO(10^{-6})$. \\
We execute the neural network once on the CPU and once on  the GPU using direct model execution in tensorflow. The used CPU is a $24$ thread AMD Ryzen9 3900x with 32GB memory and the GPU is a RTX3090 with 20GB memory. The experiments are reiterated $100$ times to reduce time measurement fluctuations. Table~\ref{tab_timingbenchmark} displays the mean timing for each configuration and corresponding standard deviation. \\
\begin{table}[htbp]
	\caption{Computational cost for one iteration of the 1D solver in seconds} 
	\centering
	\begin{tabular}{|l|l|l|l|} 
		\hline
		 & Newton & neural closure CPU & neural closure GPU \\
		\hline 
		uniform, $10^3$ samples & $0.00648\pm 0.00117$ s & $0.00788\pm 0.00051$ s & $0.00988\pm 0.00476$ s \\ 
		\hline 
		uniform, $10^7$ samples & $5.01239\pm 0.01491$ s & $0.63321\pm 0.00891$ s  & $0.03909\pm 0.00382$ s \\
		\hline
		boundary, $10^3$ samples & $38.35292\pm  0.07901$ s & $0.00802\pm 0.00064$ s & $0.00974\pm 0.00475$ s\\
		\hline 
		boundary, $10^7$ samples &  $27179.51012\pm 133.393$  s & $0.63299\pm 0.00853$ s  & $0.03881\pm 0.00352$ s  \\  
		\hline
		interior, $10^3$ samples & $ 0.00514\pm 0.00121$ s & $0.00875\pm 0.00875$ s & $0.00956\pm  0.00486$ s\\
		\hline 
		interior, $10^7$ samples & $4.24611\pm 0.03862$ s & $0.63409\pm 0.00867$ s  & $0.03846 \pm 0.00357$ s\\
		\hline
	\end{tabular} 
	\label{tab_timingbenchmark}
\end{table}
Considering Table~\ref{tab_timingbenchmark}, we see that the time consumption of a neural network is indeed independent of the condition of the optimization problem, whereas the Newton solver is $6300$ times slower on a on a moment $u^n$ with dist$(u^n,\partial{\mathcal{R}^n}) =0.01$ compared to a moment in the interior. The average time to compute the Lagrange multiplier of a uniformly sampled moment $u^n$ is $27$\% higher than a moment of the interior.
Reason for this is, that the Newton optimizer needs more iterations, the more ill-conditioned the optimization problem is. In each iteration, the inverse of the Hessian must be evaluated and the integral $\inner{\cdot}$ must be computed using a $30$ point Gauss-Legendre quadrature. One needs a comparatively high amount of quadrature points, since the integrand $m\times m \exp(\alpha^n\cdot m)$ is highly nonlinear. 
The neural network evaluation time is independent of the input data by construction and depends only on the neural network architecture and its size. Here we evaluate the input convex neural network for the $1D$ $M_2$ closure, whose size is determined by Table~\ref{tab_val_losses_cpl}. The timings for the other networks are similar, since they do not differ enough in size. 
However, we need to take into account that the neural entropy closure is less accurate near $\partial\overline{\mathcal{R}}$ as shown in Fig.~\ref{fig_dataSampling1D}.
Furthermore, we see that the acceleration gained by usage of the  neural network surrogate model is higher in cases with more sampling data. This is apparent in the uniform and interior sampling test cases, where the computational time increases by a factor of $\approx 73$, when the data size increases by a factor of $10^4$. The time consumption of the Newton solver increases by a factor of $\approx840$ in the interior sampling case, respectively $\approx 782$ in the uniform sampling case. Note, that in this experiment, all data points fit into the memory of the GPU, so it can more efficiently perform SIMD parallelization.
Reason for the smaller speedup of the neural network in case of the smaller dataset is the higher communication overhead of the parallelization relative to the workload. This indicates that the best application case for the neural network is a very large scale simulation.
\begin{table}[htbp]
    \centering
	\caption{Computational setup of the test cases} 
	\begin{tabular}{|l|l|l|l|l|} 
		\hline
		&1D M1 & 1D M2 & 2D M1 \\
		\hline
		$T$ & $(0,0.7]$ & $(0,0.7]$ & $(0,10]$ \\	\hline
		Time steps  & $8750$ & $8750$ & $33333$\\ 	\hline
		$\mathbf X$ & $[0,1]$ & $[0,1]$& $[-1.5,1.5]\times[-1.5,1.5]$ \\	\hline
		Grid cells $n_x$&$5000$ & $5000$ & $2000^2$ \\	\hline
		Quadrature &  Gauss-Legendre & Gauss-Legendre & Tensorized Gauss-Legendre\\	\hline
		$\mathbf{V}$ &  $[-1,1]$ & $[-1,1]$ & $[-1,1]\times[0,2\pi)$ \\ 	\hline
		Quadrature points & $28$ & $28$ & $400$\\	\hline
		Basis & Monomial &  Monomial & Monomial\\ 	\hline
		CFL number & $0.4$ & $0.4$ &  $0.4$ \\ 	\hline
		$\sigma$  & $1.0$ & $1.0$ & $0.0$\\ 	\hline
		$\tau$ & $0.5$ & $0.5$ & $0.0$\\ 	\hline
	\end{tabular} 
	\label{tab_2D_setup}
\end{table}
\subsection{An-isotropic inflow into an isotropically scattering, homogeneous medium in 1D}
Let us first study the particle transport in an isotropic scattering medium.
We consider the one-dimensional geometry, where the linear Boltzmann equation reduces to
\begin{align}
    \partial_t f+\mathbf v \partial_{\mathbf{x}} f&=Q(f)- \tau f\\
    &=\sigma\int_{-1}^1 \frac{1}{2}\left( f(\mathbf v_*) - f(\mathbf v)\right)\intD \mathbf v_* - \tau f,
\end{align}
where $\sigma$ is a scattering coefficient and $\tau$ is an absorbtion coefficient. The corresponding moment model becomes
\begin{align}
  \partial_t u + \partial_{\mathbf{\mathbf x}} \inner{\mathbf v m f_u}&= \sigma\inner{m Q(f_u)}-\tau u\\
      f_u &= \eta_*(\alpha_\theta\cdot m)
\end{align}
The initial condition of the computational domain is set as vacuum with $f(0,\mathbf{x},\mathbf v)=\epsilon$, where $1 \gg\epsilon > 0$ is a safety treshold, since the normalized vector $u^n$ is undefined for $u_0=0$ and $0\in\partial\mathcal{R}$.
An an-isotropic inflow condition is imposed at the left boundary of domain with
\begin{align}
    f(t>0,x=0,v) =  \begin{cases}
      0.5 & \text{if $v>0$}\\
      0 & \text{if $v\leq0$},
    \end{cases} 
\end{align}
and the right hand side boundary is equipped with a farfield condition.
The domain is resolved using a structured grid in space using a kinetic upwind scheme~\cite{KRISTOPHERGARRETT2015573} and an explicit Euler scheme in time. The benchmarking solver uses a Newton based optimizer with linesearch to compute the minimal entropy closure, and the neural network based solver uses the neural network prediction to compute the kinetic flux. The Newton based optimizer is set to single precision accuracy. The CFL number is set to $0.4$ to avoid that the finite volume update steps outside the realizable domain $\mathcal{R}$,~\cite{OLBRANT20125612}.
The detailed computational setup can be found in Table \ref{tab_2D_setup}.
The solution profiles at final time $t_f=0.7$ of the neural network based entropy closed moment system and the reference solver are presented in Fig.~\ref{fig_1D_closures} for the $M_1$ and $M_2$ system. 
We can see that the systems dynamics are well captured by both neural network architectures. \\
In order to verify the significance of the following error discussion, we conduct a convergence analysis of both test cases with the used finite volume solver. 
Figure.~\ref{fig_1D_convergence_plot}a) compares the convergence of the solution of both neural network entropy closure and Newton closed solver of the $1D$ $M1$ test case and Fig.~\ref{fig_1D_convergence_plot}b) the corresponding solutions of the  $2D$ $M1$ test case. We assume the solution of the Newton solver at final time $t_f$ with the finest grid as the ground truth. Due to a fixed CFL number, the amount of time steps needed for each simulation is proportional to the number of used grid cells. The plots display first order convergence for the Newton based solver as expected.
We can see in Fig.~\ref{fig_1D_convergence_plot}a), that the monotonic neural network in the $1D$ $M1$ inflow test case converges with first order accuracy up to an error level of $\landauO(10^{-2.5})$. For finer grid resolutions, the error in the neural network based closure dominates the spatial discretization error. The input convex neural network exhibits similar behavior, but the error plateau is reached at $\landauO(10^{-3.5})$. \\
In Fig.~\ref{fig_1D_closures_err}, we see the corresponding  norm errors of the $M_1$ and $M_2$ solution for each grid cell at final time $t_f$. The point wise norm error is again in the range of $\landauO(10^{-3.5})$ in case of the input convex architecture and in the range of $\landauO(10^{-2.5})$ in case of the monotonic  network architecture in the $M_1$ test case. In the $M_2$ test case, the errors do not exceed $\landauO(10^{-2})$. 
An inspection of the relative errors of these test cases is given in Fig.~\ref{fig_1D_closures_rel_err}. One can spot the maximal relative error in both test cases  at  $x\in(0.7,0.8)$ at final time $t_f$. The wave front is located  in this area  in the an-isotropic inflow simulation and the moments $u$ are closest to the boundary of the realizable set $\partial\mathcal{R}$.
\begin{figure}
    \centering
    \begin{minipage}{0.495\textwidth}
        \includegraphics[width=\textwidth]{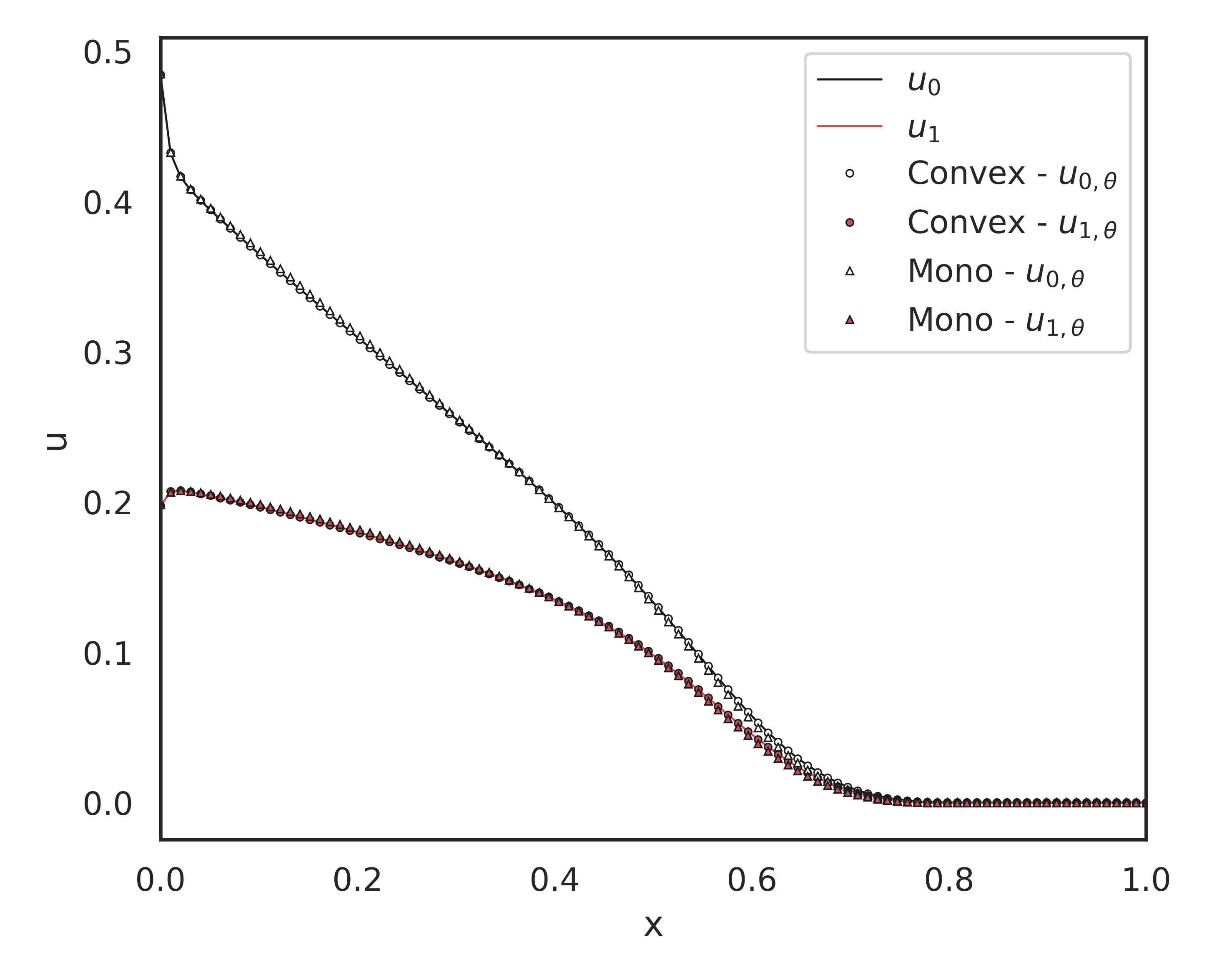}
        \centering{\\a) $1D$ $M_1$ closure}
   \end{minipage}
   \begin{minipage}{0.495\textwidth}
       \includegraphics[width=\textwidth]{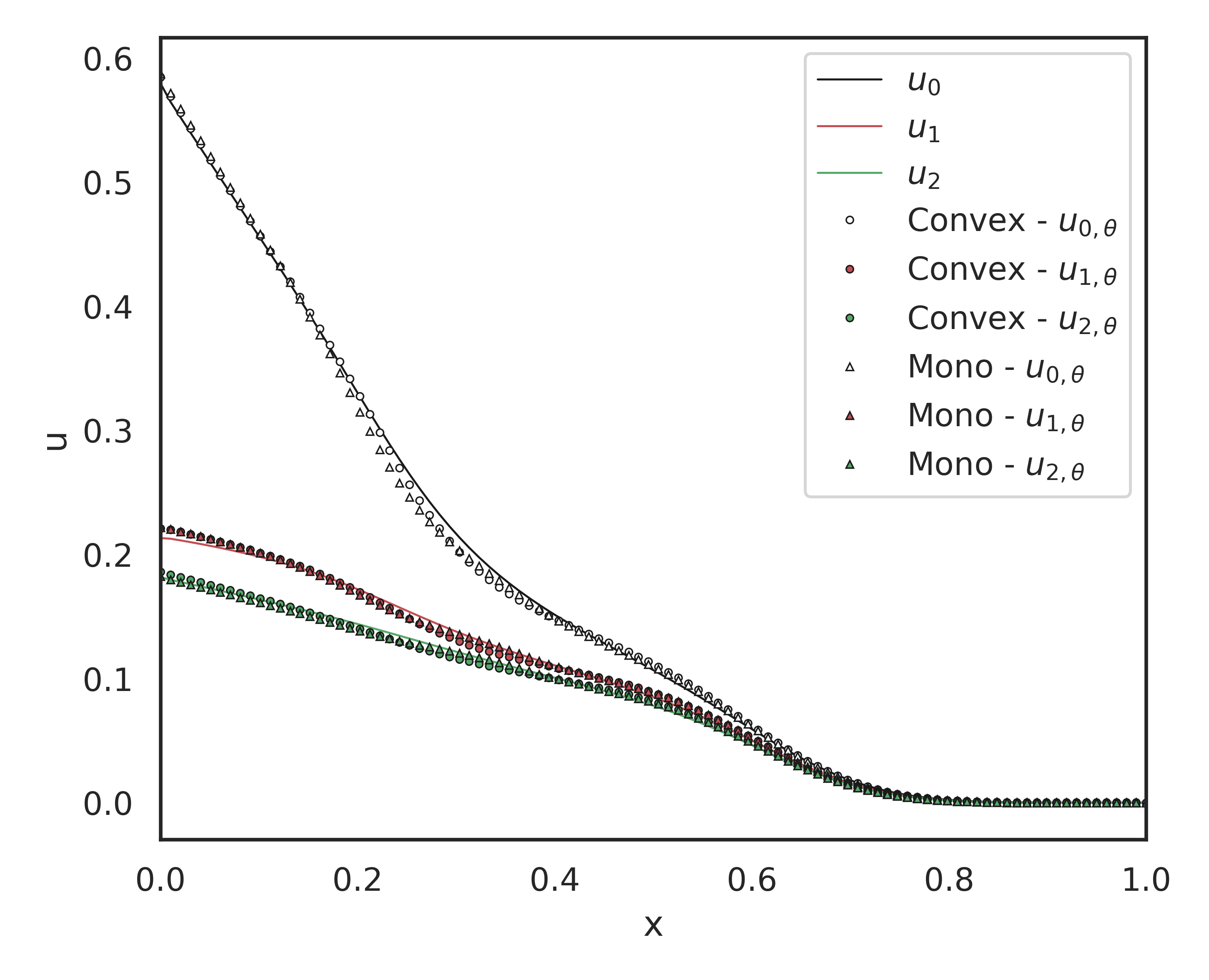}
        \centering{\\b)  $1D$ $M_2$ closure}
   \end{minipage}
   \caption{Moments of the solution of the an-isotropic inflow test case at $t_f=0.7$. Comparison of the benchmark solution, the input convex closure and the monotonic closure.}
   \label{fig_1D_closures}
\end{figure}
\begin{figure}
   \begin{minipage}{0.495\textwidth}
       \includegraphics[width=\textwidth]{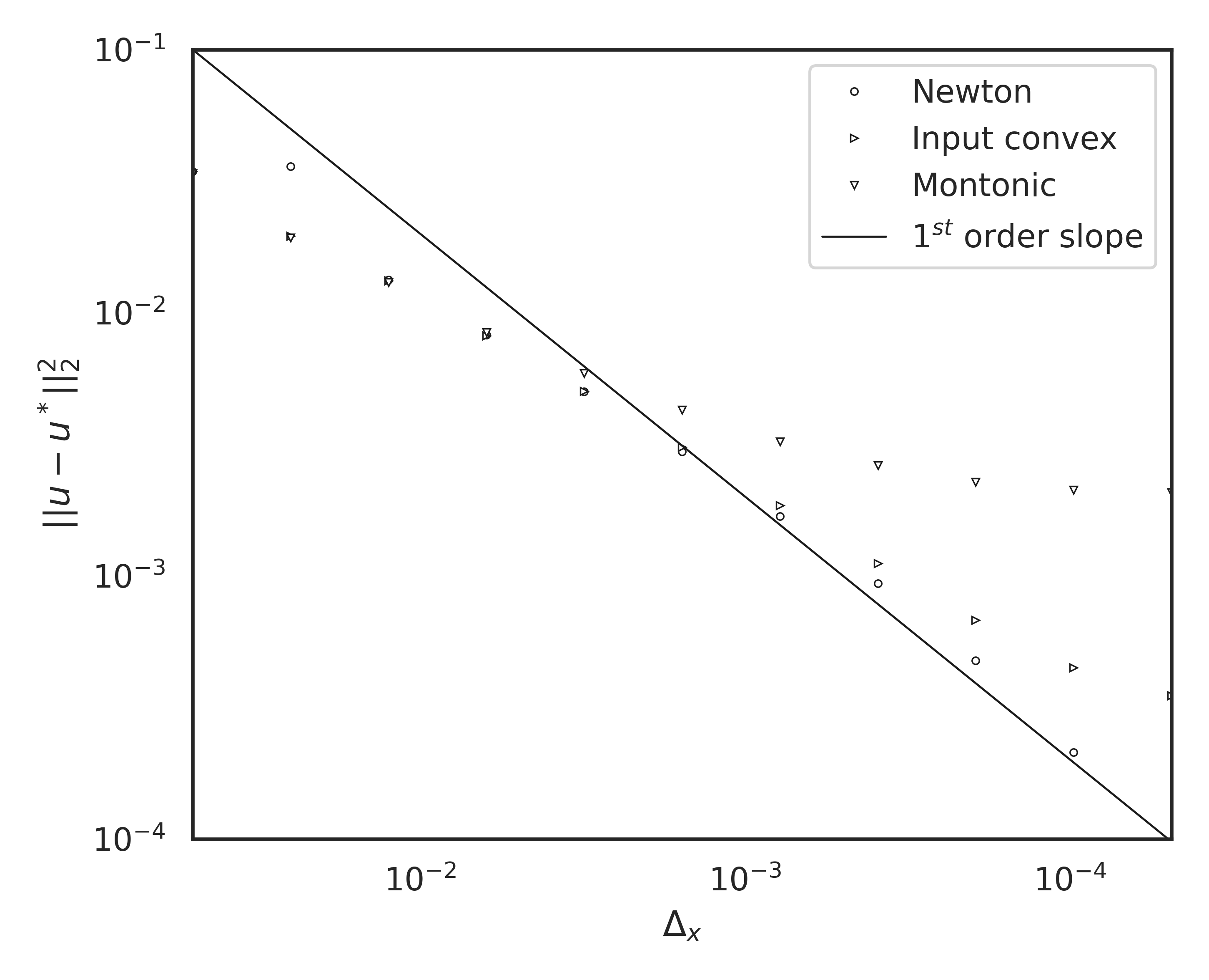}
       \centering{\\a) $1D$  $M_1$ closure}
    \end{minipage}
    \begin{minipage}{0.495\textwidth}
       \includegraphics[width=\textwidth]{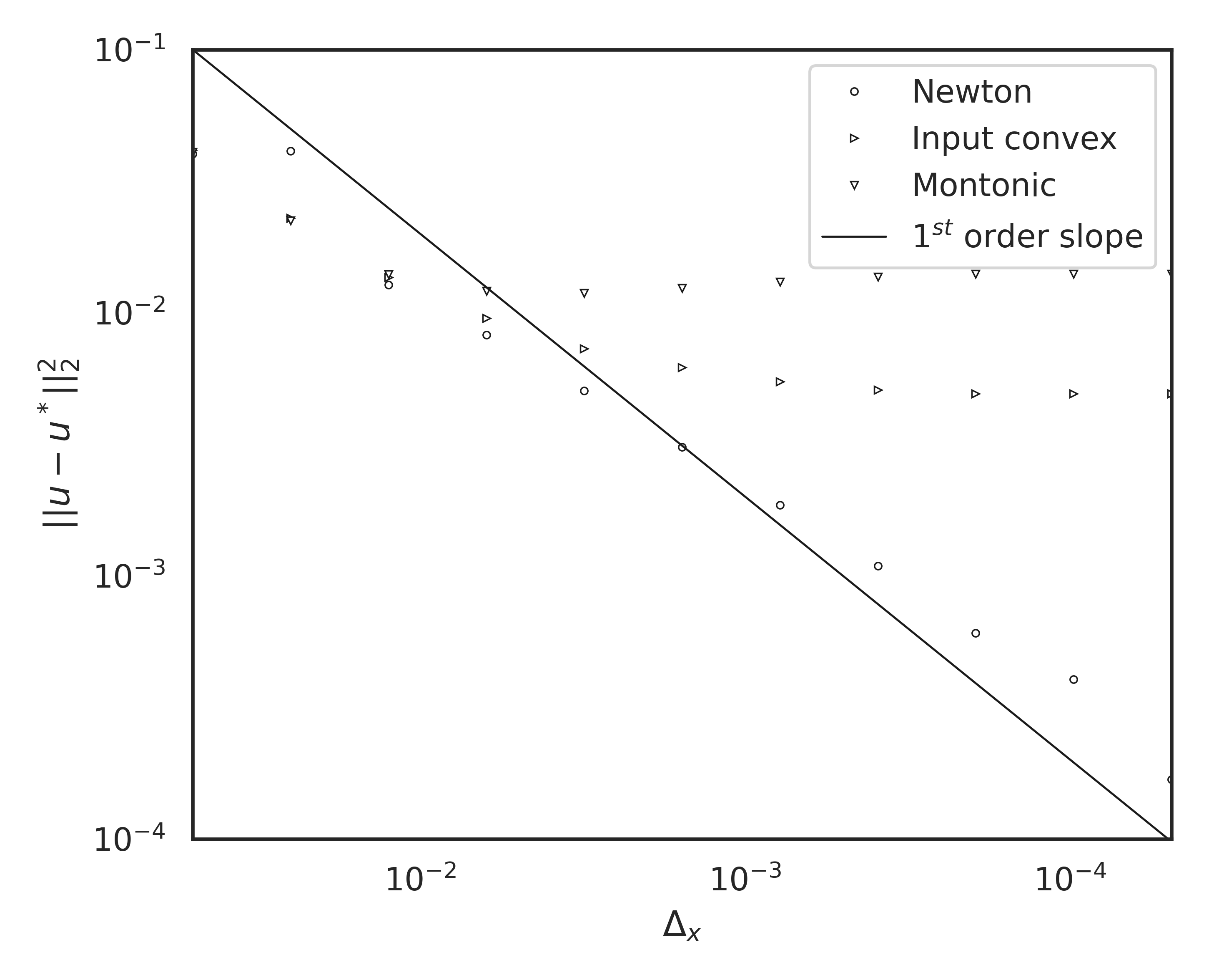}
      \centering{\\b) $1D$  $M_2$ closure}
    \end{minipage}
   \caption{Comparison of the convergence rate of a first order finite volume solver with different closures. The ground truth $u^*$ is given by the finest Newton based solution. The spatial cell size is denoted by $\Delta_x$.}
   \label{fig_1D_convergence_plot}
\end{figure}
\begin{figure}
   \begin{minipage}{0.495\textwidth}
       \includegraphics[width=\textwidth]{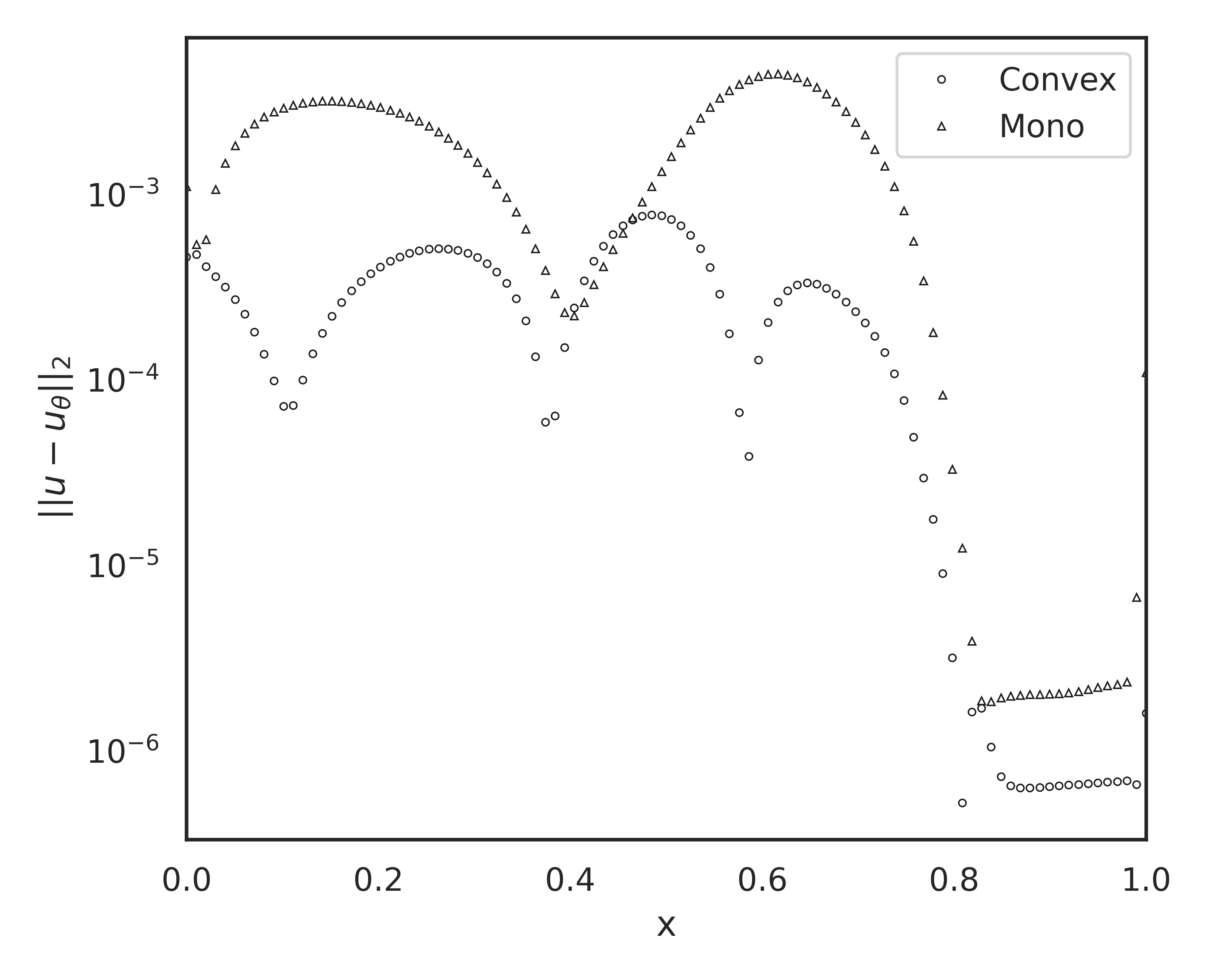}
       \centering{\\a)  $1D$ $M_1$ closure}
    \end{minipage}
    \begin{minipage}{0.495\textwidth}
       \includegraphics[width=\textwidth]{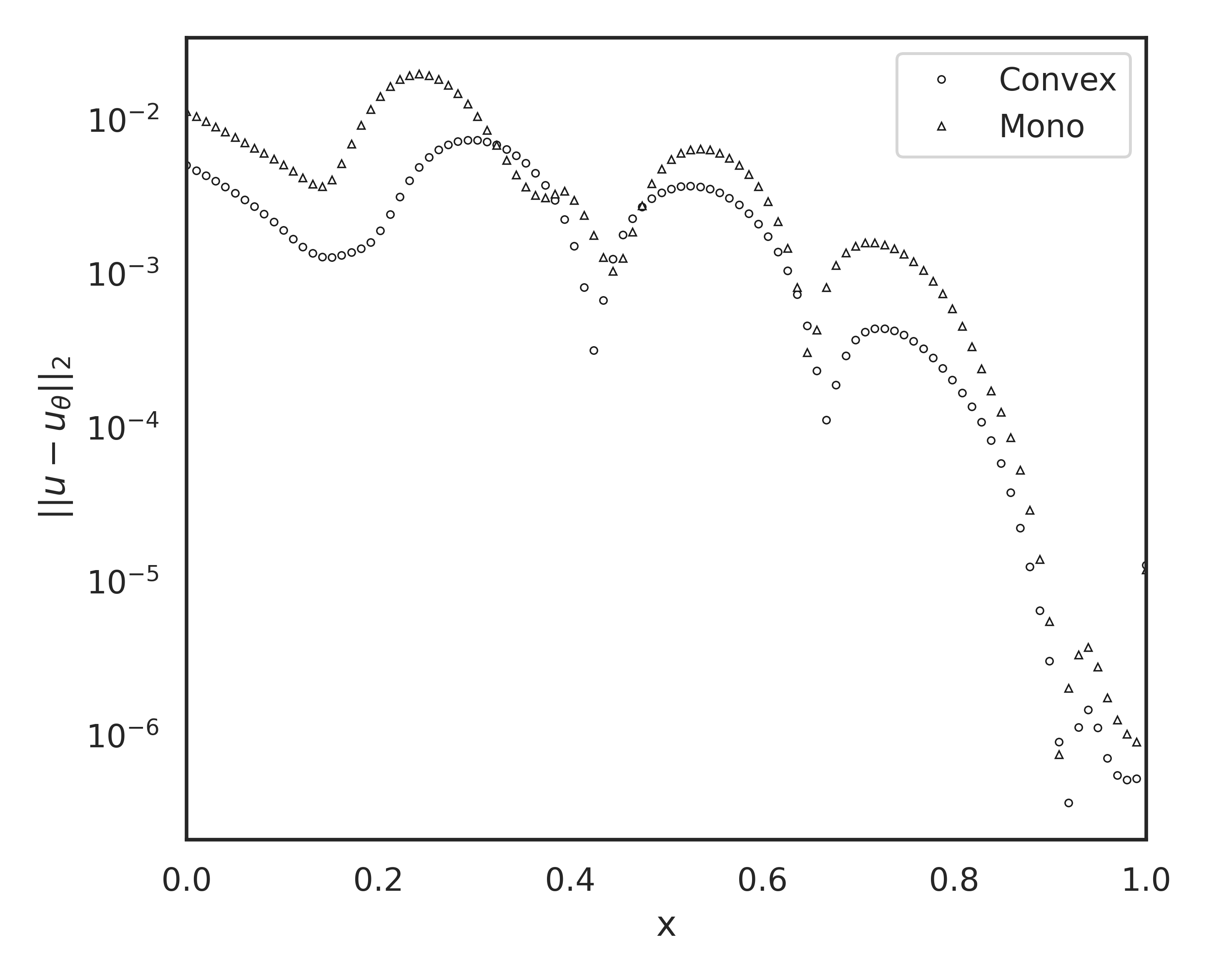}
      \centering{\\b) $1D$  $M_2$ closure}
    \end{minipage}
   \caption{Norm error at individual grid points of the neural network based closure with respect to the benchmark solution at time $t_f=0.7$. For readability, not every grid cell is displayed.}
   \label{fig_1D_closures_err}
\end{figure}
\begin{figure}
   \begin{minipage}{0.495\textwidth}
       \includegraphics[width=\textwidth]{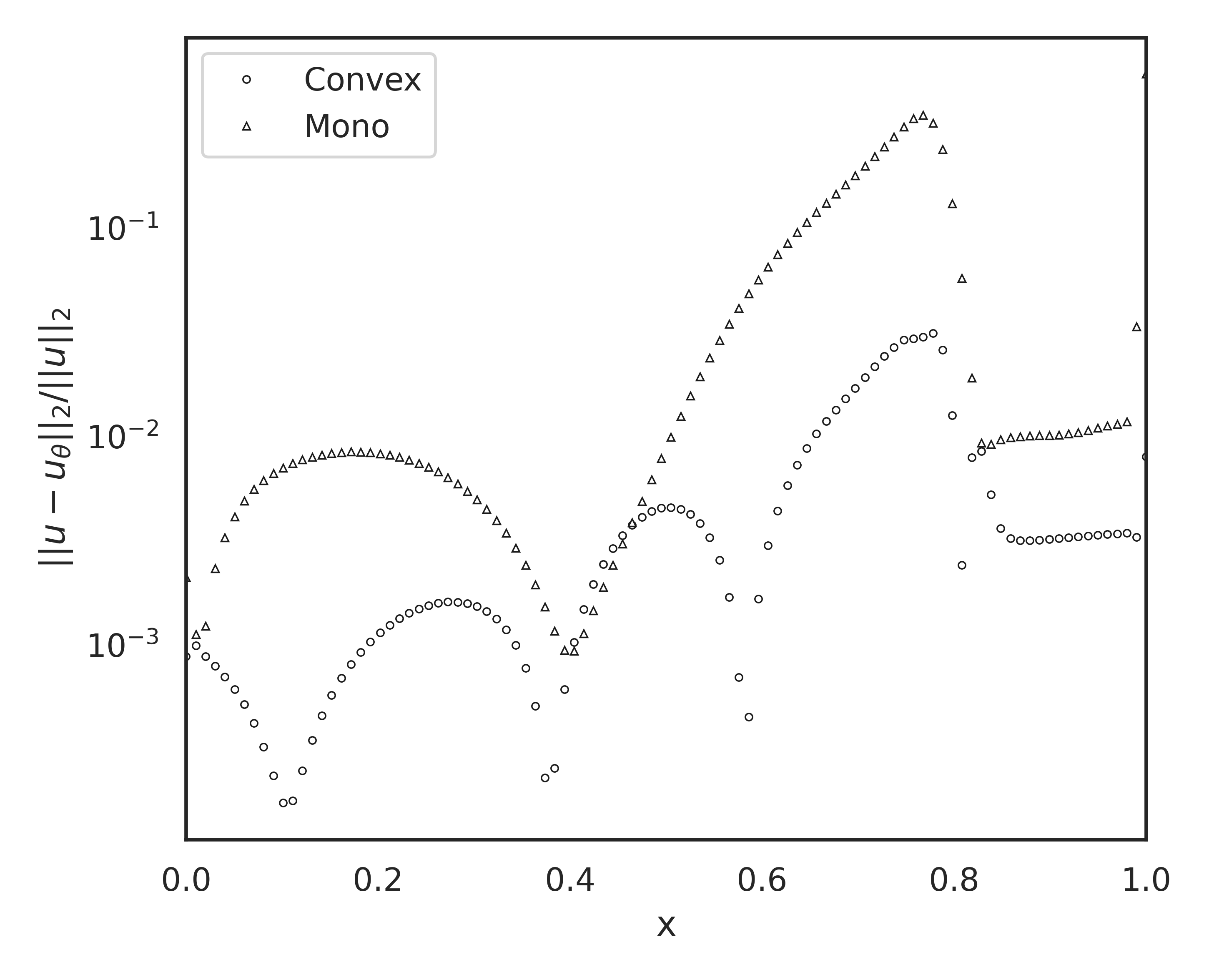}
       \centering{\\a)  $1D$ $M_1$ closure}
    \end{minipage}
    \begin{minipage}{0.495\textwidth}
       \includegraphics[width=\textwidth]{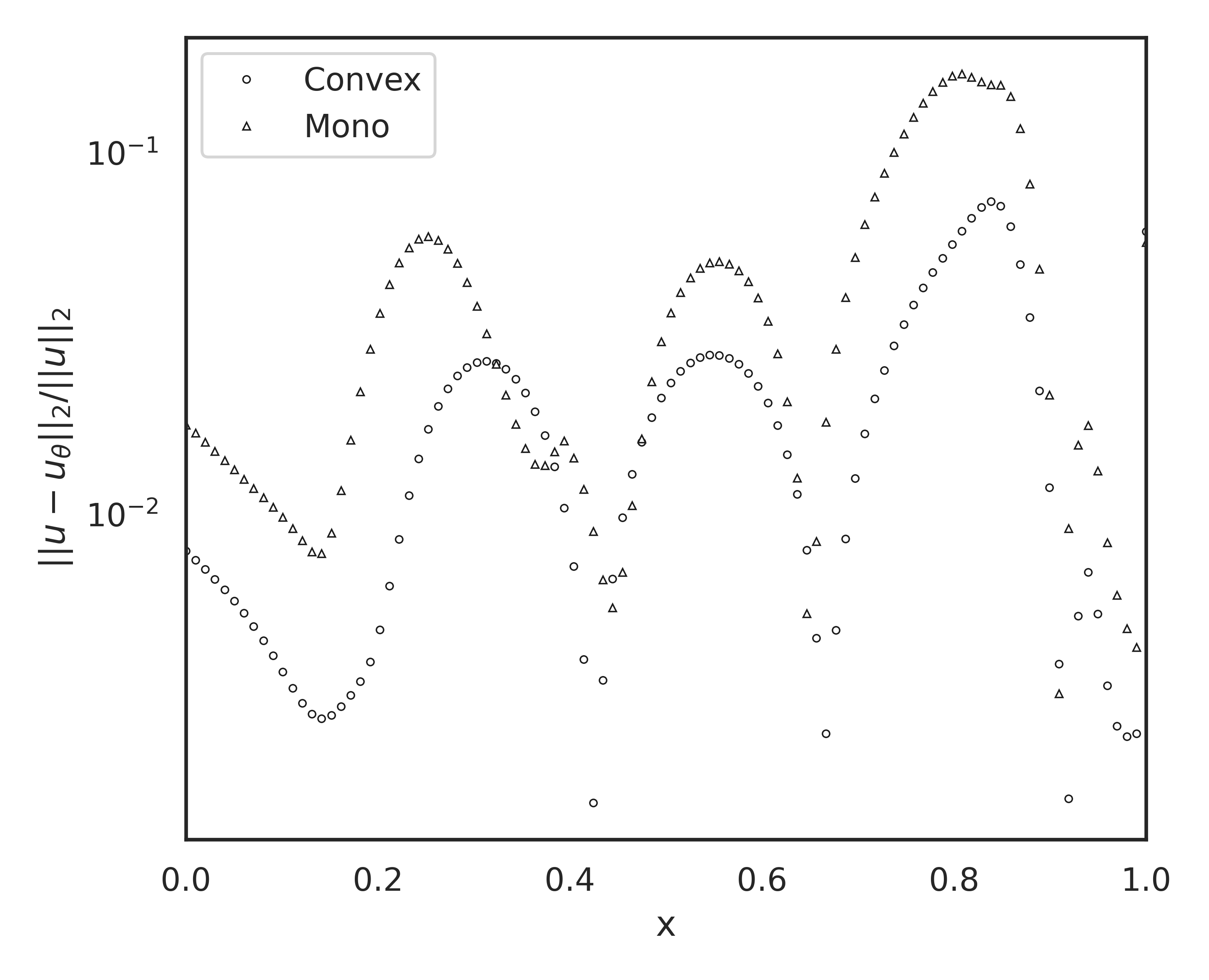}
       \centering{\\b)  $1D$ $M_2$ closure}
    \end{minipage}
   \caption{Relative norm error at individual grid points of the neural network based closure with respect to the benchmark solution at time $t_f=0.7$. For readability, not every grid cell is displayed}
   \label{fig_1D_closures_rel_err}
\end{figure}

\subsection{Particles in a 2D non scattering medium with periodic initial conditions}
We consider a rectangular domain in two spatial dimensions. The phase space of the Boltzmann equation is thus five dimensional, where $\mathbf{X}=[-1.5,1.5]^2$, $\mathbf{V}=\menge{\mathbf v\in\mathbb{R}^2:\norm{\mathbf v}_2 < 1}$ and $t>0$. We consider the $M_1$ closure with a monomial basis $m(\mathbf v)=[1,\mathbf v_{\mathbf x},\mathbf v_{\mathbf y}]^T$.
The velocity domain $\mathbf V$ is parametrized in spherical coordinates
\begin{align}
       \left(\begin{array}{c} \mathbf v_{\mathbf x} \\ \mathbf v_{\mathbf y} \end{array}\right)= \left(\begin{array}{c} \sqrt{1-\mu^2}\cos(\phi) \\ \sqrt{1-\mu^2}\sin(\phi) \end{array}\right), \qquad (\mu,\phi)\in [-1,1]\times [0,2\pi).
\end{align}
This test case considers a non scattering and non absorbing medium, i.e. $\sigma=\tau=0$, and the Boltzmann equation reduces to a transport equation of the form
\begin{align}
\begin{aligned}
    \partial_t f + \mathbf v_{\mathbf x}\partial_{\mathbf x} f + \mathbf v_{\mathbf y}\partial_{\mathbf y} f = 0.
\end{aligned}
\end{align}
The corresponding moment system with minimal entropy closure reads 
\begin{align}
\begin{aligned}
  \partial_t u +\partial_{\mathbf x}\inner{\mathbf v_{\mathbf x} mf_u}+\partial_{\mathbf y}\inner{\mathbf v_{\mathbf y} mf_u} &= 0 \\
    f_u &= \eta_*(\alpha_u\cdot m)
\end{aligned}
\end{align}
The Boltzmann equation is equipped with periodic initial conditions that translate to the $M_1$ moment equations
\begin{align}
    u_0 &= 1.5 + \cos(2\pi \mathbf x)\cos(2\pi \mathbf y), \qquad (\mathbf x,\mathbf y)\in\mathbf{X},\\
    u_1 &= 0.3 u_0, \\
    u_2 &= 0.3 u_0.
\end{align}
Periodic boundary conditions are imposed on the equations to get a well posed system of equations. 
Note that due to the absent of gain and loss terms and the choice of boundary conditions, the system is closed and cannot lose or gain particles. 
The $M_1$ system is solved using again a kinetic scheme with a $2D$ finite volume method in space, an explicit Euler scheme in time and a tensorized $2D$ Gauss-Legendre quadrature to compute the velocity integrals. The detailed solver configuration can be found in Table~\ref{tab_2D_setup}. Analogously to the $1D$ test cases, we compare the Newton based benchmark solution to the neural network based closures with the input convex and monotonic architectures. 
We run the simulation until a final time $t_f=10.0$, which translates to $33333$ time-steps. \\
We conduct a convergence analysis for the $2D$ $M1$ closures of both network architectures in Fig.~\ref{fig_convergence_plot_2d_periodic}. The convergence of the input convex neural network levels of at $\landauO(10^{-3})$ and the convergence of the monotonic network at $\landauO(10^{-2.5})$, which is in line with the findings of the $1D$ closures. The size of the spatial grid is chosen correspondingly.
Figure~\ref{fig_period2D} shows a snapshot of the flow field computed with the benchmark solver and Fig.~\ref{fig_period_err} displays snapshots of the relative error at each grid cell of the flow field at the same iteration as the benchmark solver in Fig.~\ref{fig_period2D}. The relative errors of both neural networks exhibit periodic behavior and are in the range of $\landauO(10^{-2})$ or lower. Similarly to the $1D$ test cases, the input convex architecture is again slightly more accurate than the monotonic counterpart. \\
Figure~\ref{fig_MRE_entropy}a) and b) display the relative norm error of both $\alpha_u$ and the moment $u$ of both neural network architectures at each time step  of the simulation averaged over the whole computational domain.
First, one can observe that in both figures again the relative error of the monotonic neural network is slightly bigger than the error of the input convex neural network.
Second, we can see that in the first time steps of the simulation, the error increases from $\landauO(10^{-4})$ to $\landauO(10^{-2})$ in case of the moments, respectively $\landauO(10^{-3})$ to $\landauO(10^{-1.5})$ in case of the Lagrange multipliers. After this initial increase, the error stays stable for the reminder of the simulation. The oscillations in the error curves stem from the periodic nature of the system's solution, in which the distance to $\partial\mathcal{R}$ of the appearing moments changes periodically as well.\\
\begin{figure}
 \centering
      \begin{minipage}{0.48\textwidth}
      \centering
            \includegraphics[width=\textwidth]{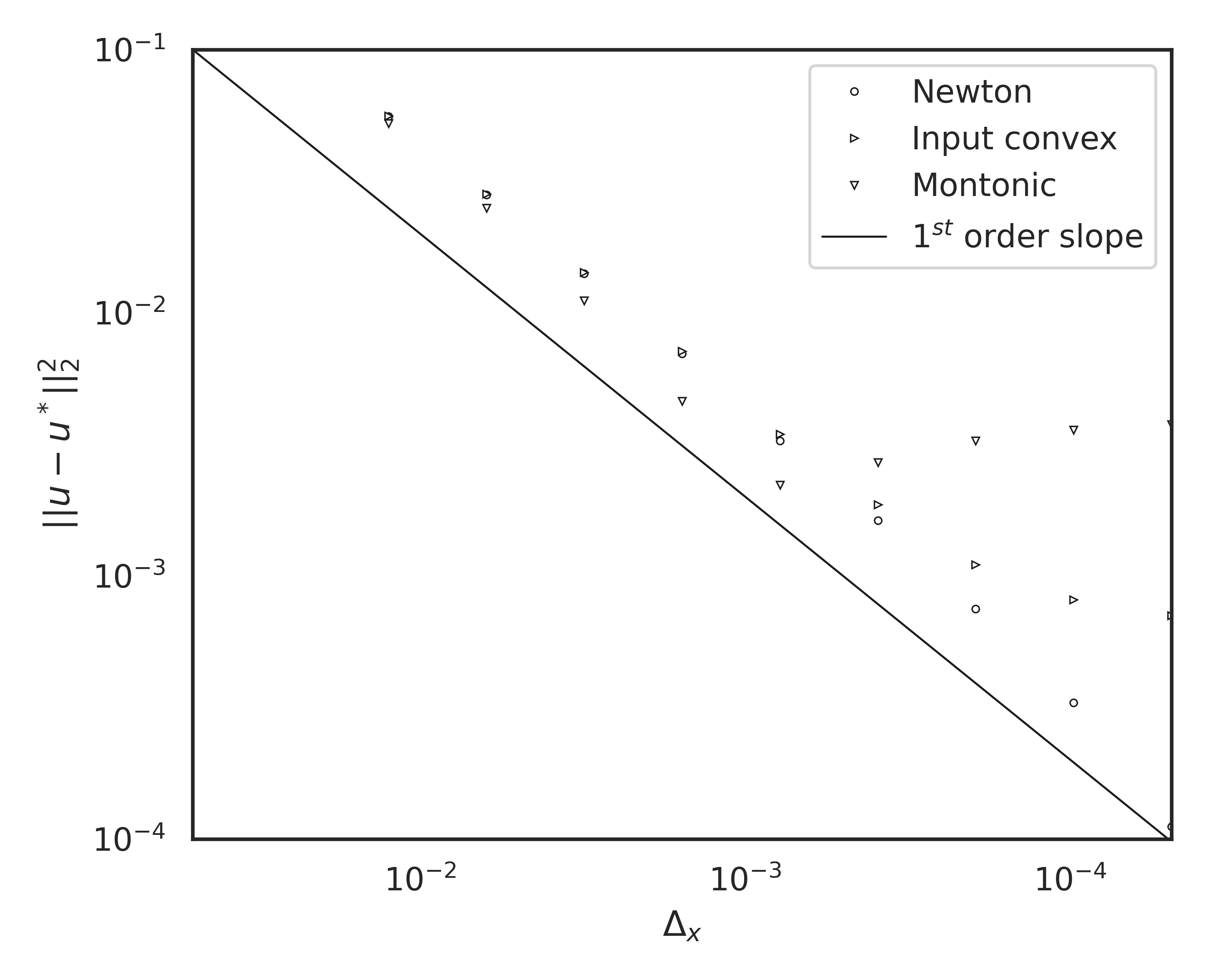}
               \caption{Comparison of the convergence rate of a first order finite volume solver with different closures. The ground truth $u^*$ is given by the finest Newton based solution. The spatial cell size is denoted by $\Delta_x$.}
    \label{fig_convergence_plot_2d_periodic}
   \end{minipage}
    \begin{minipage}{0.48\textwidth}
          \centering
         \includegraphics[width=0.94\textwidth]{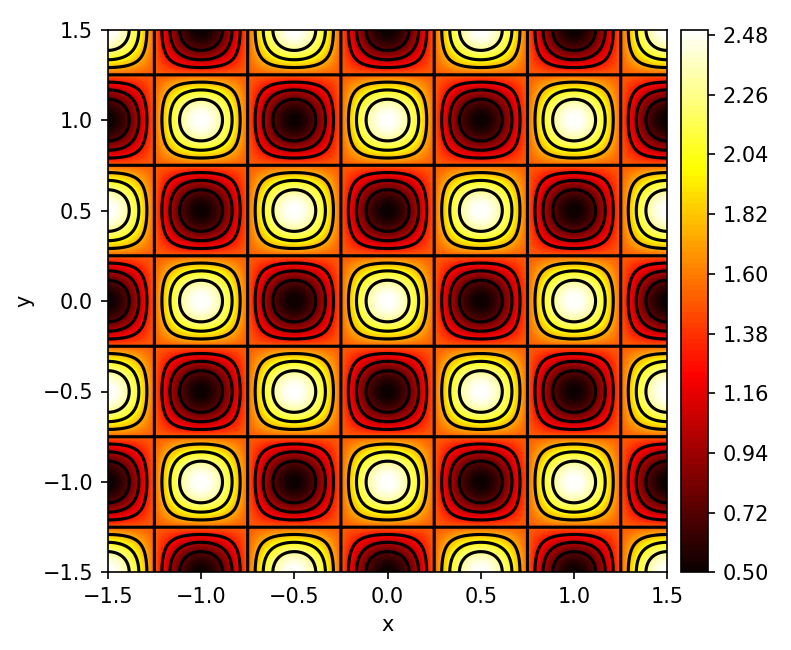}
         \caption{Snapshot of the benchmark solution of the $2D$ $M_1$ test case. The colorbar indicates the value of $u_0$ at a grid cell.}
          \label{fig_period2D}
            \end{minipage}
\end{figure}
\begin{figure}
    \centering
      \begin{minipage}{0.49\textwidth}
          \centering
            \includegraphics[width=\textwidth]{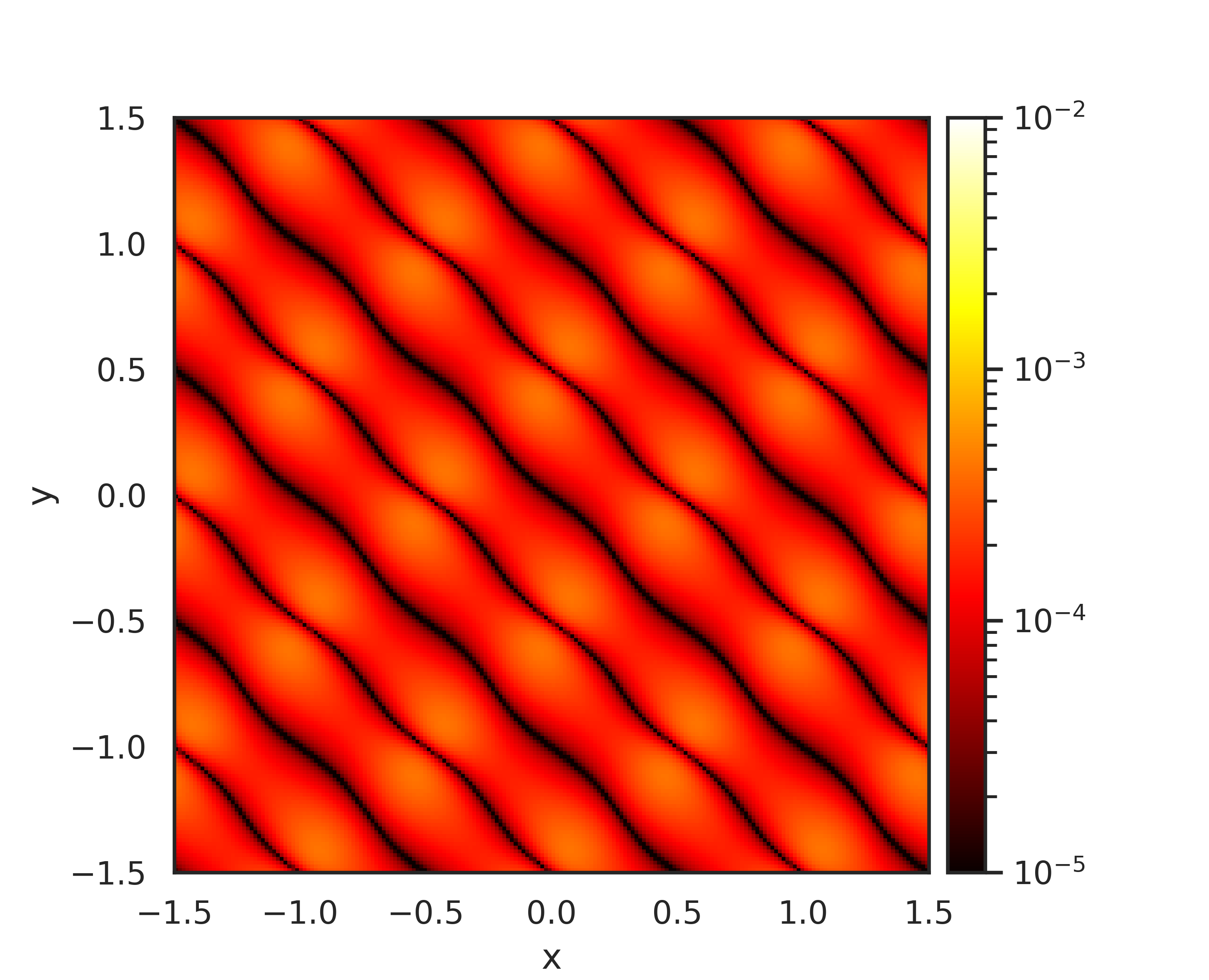}
           {\\ a) Input convex neural network }
   \end{minipage}
    \begin{minipage}{0.49\textwidth}
        \centering
            \includegraphics[width=\textwidth]{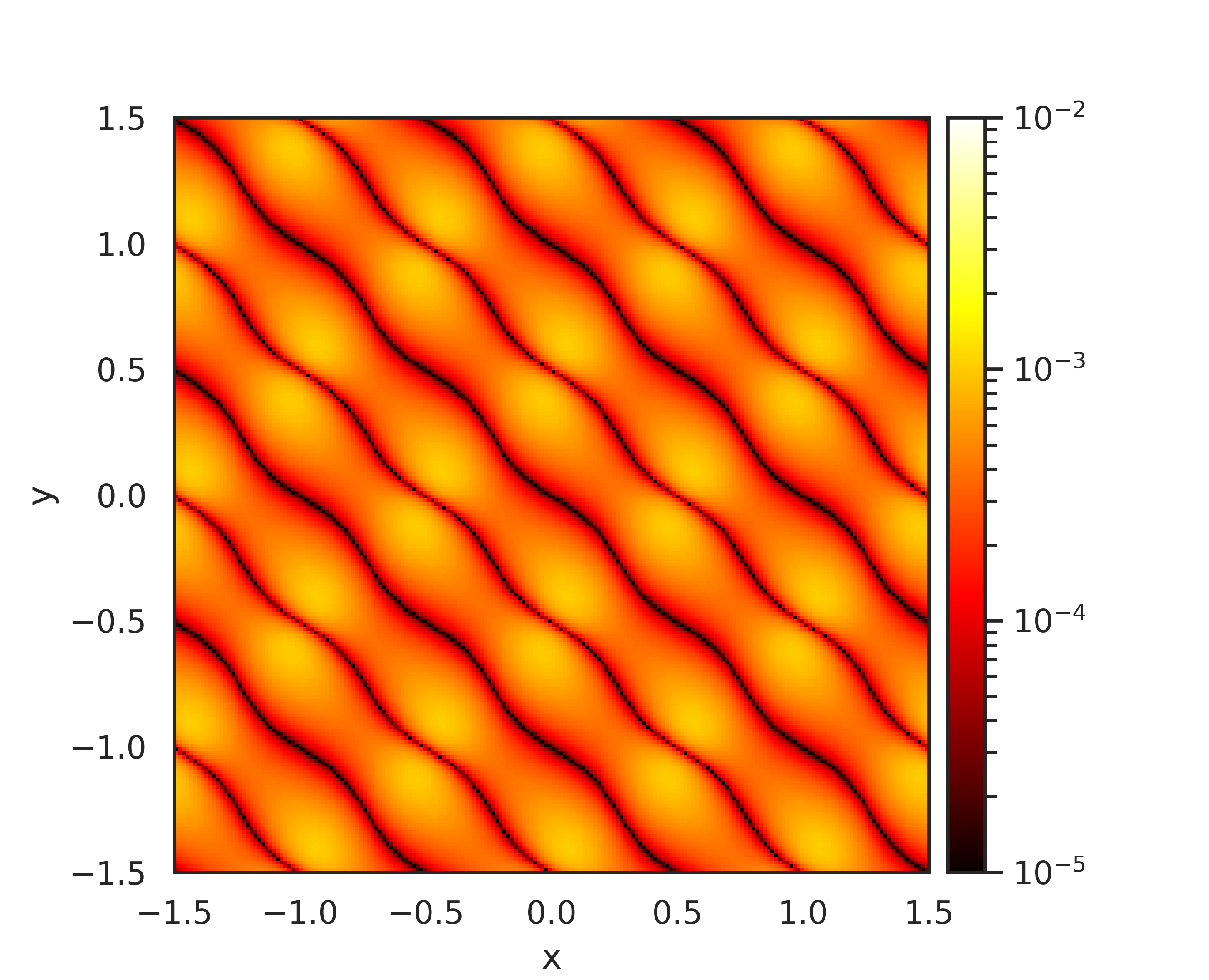}
             {\\ b) Monotonic neural network}
    \end{minipage}
    \caption{Snapshot of the relative norm error of $u_{\theta}$ with respect to the benchmark solution in the $2D$ $M_1$ test case. The colorbar indicates the value of $\norm{u-u_\theta}_2/\norm{u}_2$ at agrid cell.}\label{fig_period_err}
\end{figure}
\begin{figure}
    \centering
      \begin{minipage}{0.49\textwidth}
          \centering
            \includegraphics[width=\textwidth]{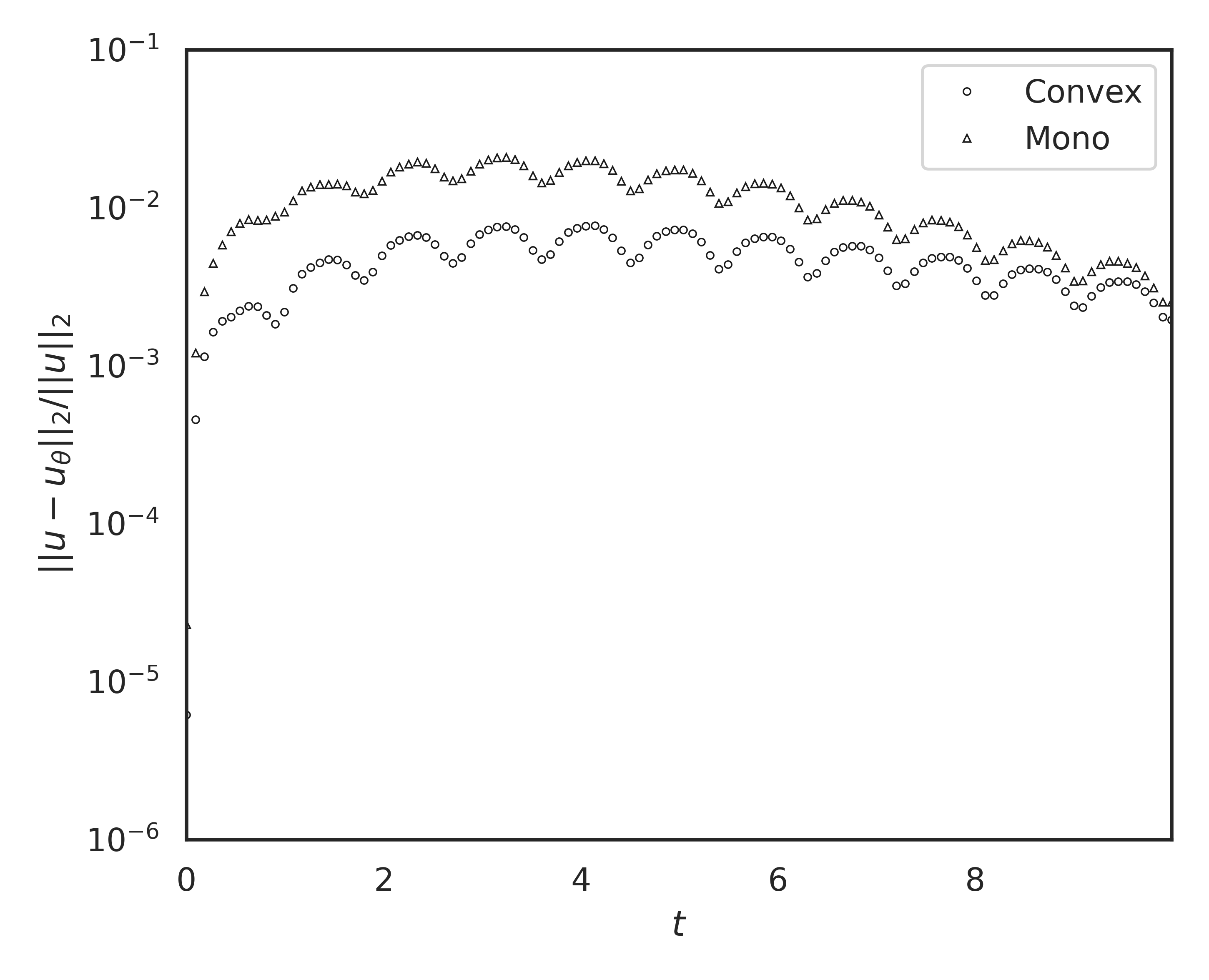}
           {\\ a) Relative error of $u_\theta$ at each time step}
   \end{minipage}
    \begin{minipage}{0.49\textwidth}
        \centering
            \includegraphics[width=\textwidth]{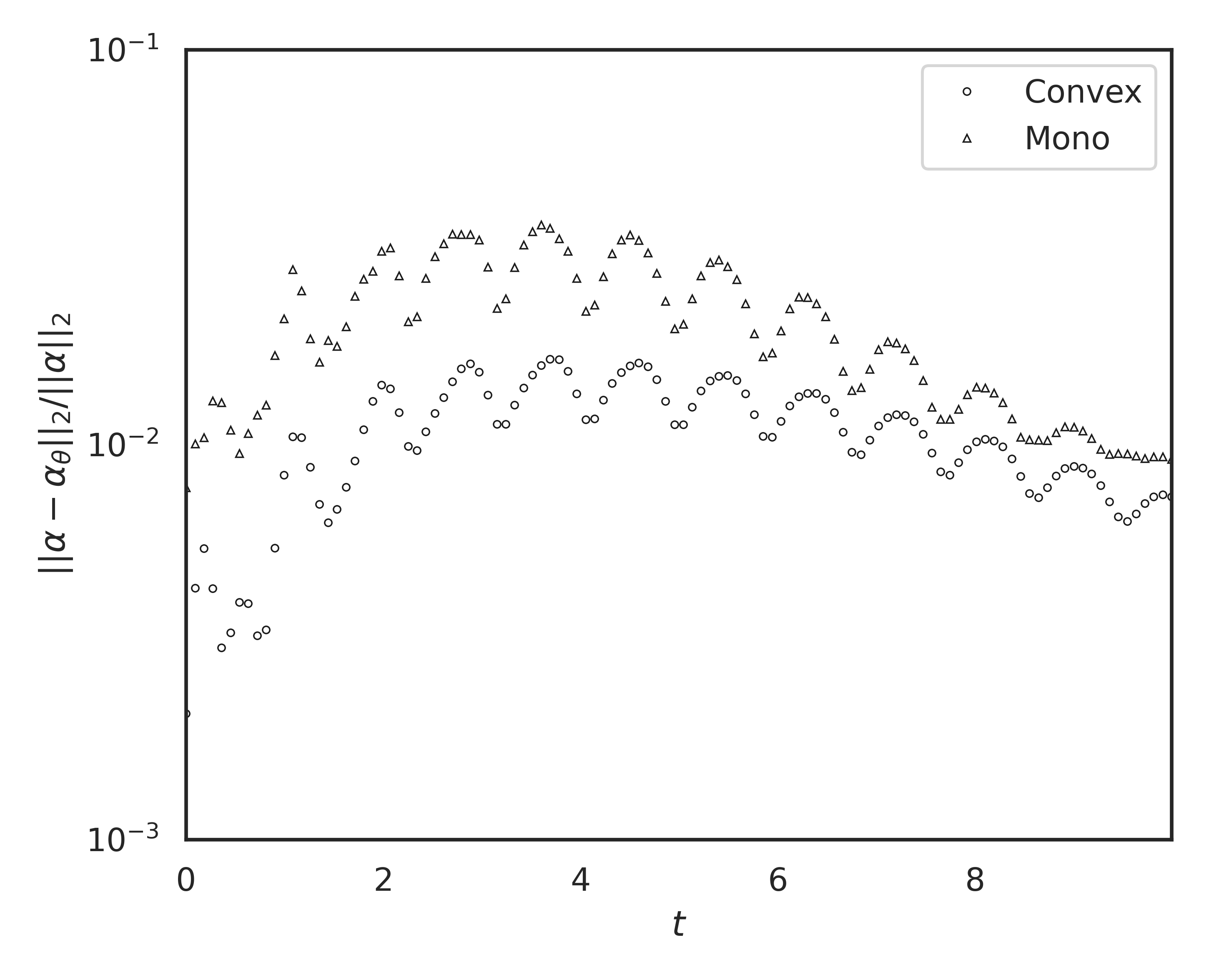}
             {\\ b) Relative error of $\alpha_\theta$ at each time step}
    \end{minipage}
    \caption{Mean over the spatial grid of the relative error of $u_\theta$ with respect to the benchmark solution over time. For better readability only a fraction of the time steps are displayed.}\label{fig_MRE_entropy}
\end{figure}
\begin{figure}
 \centering
      \begin{minipage}{0.49\textwidth}
            \includegraphics[width=\textwidth]{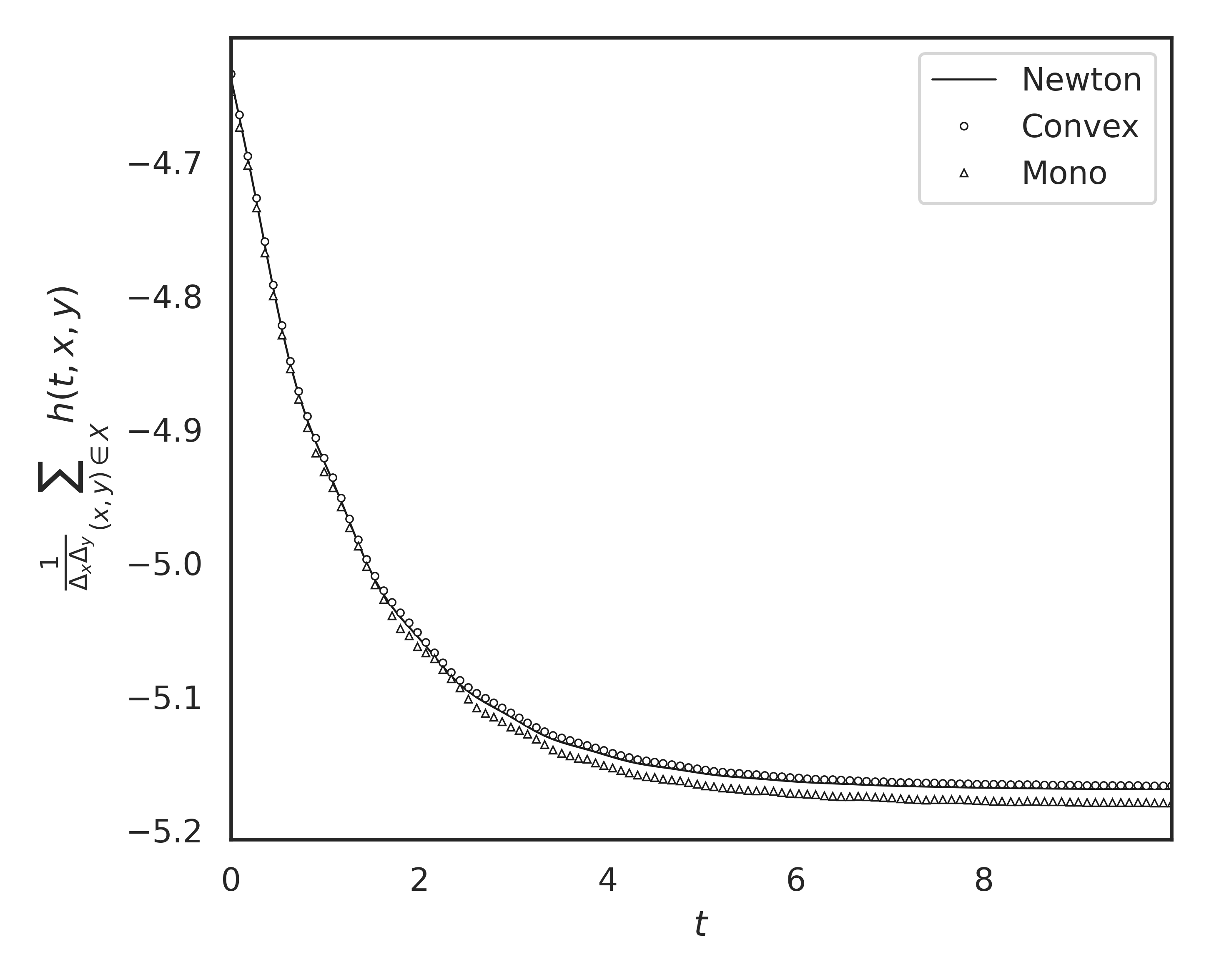}
   \end{minipage}
    \caption{Comparison of the total entropy of the systems using neural network closures and Newton based closures. For better readability only a fraction of the time steps are displayed.}
    \label{fig_sys_entropy}
\end{figure}
Lastly, we analyze the total entropy of the system at each time step. Due to the periodic boundary conditions and $\sigma=\tau=0$, we have no particle sinks or sources in the system and the system is closed. We have chosen the upwind scheme for the numerical flux of the moment system, which is an entropy dissipating scheme. Figure~\ref{fig_sys_entropy} shows the entropy dissipation of the system over time and compares the entropy of the reference solution with the two neural network architectures. All methods are entropy dissipating, however, the input convex neural network exhibits a smaller difference to the reference entropy. 
We can conclude, that the neural network based hybrid solver preserves the structural properties of the reference system and computes the numerical solution within reasonable accuracy.
\section{Summary and Conclusion}
In this paper we addressed the moment system of the Boltzmann equation, its minimal entropy closure, and the challenges of classical numerical approaches.
We introduced two novel neural network based approaches to close the moment hierarchy of the linear Boltzmann equation, once with an input convex neural network that approximates the entropy of the minimal entropy closure, and once with an monotonic neural network that approximates the Lagrange multipliers of the minimal entropy optimization problem.
In the numerical test cases, we have seen that both methods exhibit errors in a similar range, however, the input convex neural network exhibits a slightly better training and test performance than the monotonic neural network approach.
The nature of the entropy minimization problem allows clear definition of the convex set of all possible input data for the neural network. On the other hand, the problem is ill conditioned on the boundary of the  realizable set and thus poses significant challenges for generating training data on and near the boundary of it.\\
We have analyzed the generalization gap of a convex neural network which is trained in Sobolev norm and derived a bound on the generalization gap. Based on this bound we have built a strategy to sample data for entropy closures in arbitrary spatial dimension and moment order. 
We conducted analysis of the trained neural networks in a synthetic test case as well as several simulation tests.
We found a good agreement between the neural network based solutions and the reference solution within the boundaries of the training performance of the neural networks. As expected, the neural network based closures are significantly more efficient in terms of computational time  compared to a Newton solver. \\
Further research will consider treatment of the region near the boundary of the realizable set, where the neural networks exhibit the highest errors.
\section*{Acknowledgements}
The authors acknowledge support by the state of Baden-Württemberg through bwHPC. Furthermore, the authors would like to thank Max Sauerbrey for fruitful discussions about convex functions.
The work of Steffen Schotthöfer is funded by the Priority Programme SPP2298 "Theoretical Foundations of Deep Learning"  by the Deutsche Forschungsgemeinschaft.
The work of  Tianbai Xiao is funded by the Alexander von Humboldt Foundation (Ref3.5-CHN-1210132-HFST-P).
The work of Cory Hauck is sponsored by the Office of Advanced Scientific Computing Research, U.S. Department of Energy, and performed at the Oak Ridge National Laboratory, which is managed by UT-Battelle, LLC under Contract No. De-AC05-00OR22725 with the U.S. Department of Energy. The United States Government retains and the publisher, by accepting the article for publication, acknowledges that the United States Government retains a non-exclusive, paid-up, irrevocable, world-wide license to publish or reproduce the published form of this manuscript, or allow others to do so, for United States Government purposes. The Department of Energy will provide public access to these results of federally sponsored research in accordance with the DOE Public Access Plan (http://energy.gov/downloads/doe-public-access-plan).

\printbibliography
\end{document}